\documentclass[11pt]{amsart}
\usepackage{amssymb,amsmath,amsfonts,amsthm, mathrsfs,tikz, pgfplots,multicol, graphicx,hyperref,placeins,todonotes,forest}

\usepackage[all,cmtip]{xy}
\usepackage{tikz} 
\usepackage[margin=1in]{geometry}
\newtheorem{theorem}{Theorem}[section]
\newtheorem{conjecture}[theorem]{Conjecture}
\newtheorem{lemma}[theorem]{Lemma}

\newtheorem{defn}[theorem]{Definition}
\newtheorem{prop}[theorem]{Proposition}

\newcommand{\mf}[1]{\mathfrak{#1}}
\newcommand{\mbb}[1]{\mathbb{#1}}
\newcommand{\mc}[1]{\mathcal{#1}}
\newcommand{\PSL}{\operatorname{PSL}}
\newcommand{\PGL}{\operatorname{PGL}}
\newcommand{\SQ}{\mathcal{S}_{\mathbb{Q}(i)}}

\numberwithin{equation}{section}

\def\Acal{{\mathcal A}}

\def\Ccal{{\mathcal C}}

\def\Lcal{{\mathcal L}}

\def\Scal{{\mathcal S}}
\def\Tcal{{\mathcal T}}

\newcommand{\CC}{\mathbb{C}}

\newcommand{\PP}{\mathbb{P}}
\newcommand{\QQ}{\mathbb{Q}}
\newcommand{\RR}{\mathbb{R}}

\newcommand{\ZZ}{\mathbb{Z}}

\begin{document}
\title{The Dynamics of Super-Apollonian Continued Fractions}

\author[Sneha Chaubey]{Sneha Chaubey}
\address{%
Department of Mathematics, University of Illinois, 
1409 West Green Street,
 Urbana, IL 61801}
\email{chaubey2@illinois.edu}

\author[Elena Fuchs]{Elena Fuchs}
\address{%
Department of Mathematics,
UC Davis, One Shields Avenue, Davis, CA 95616}
\email{efuchs@math.ucdavis.edu}

\author[Robert Hines]{Robert Hines}
\address{%
Department of Mathematics, University of Colorado,
Campus Box 395, Boulder, Colorado 80309-0395}
\email{robert.hines@colorado.edu}

\author[Katherine E. Stange]{Katherine E. Stange}
\address{%
Department of Mathematics, University of Colorado,
Campus Box 395, Boulder, Colorado 80309-0395}
\email{kstange@math.colorado.edu}

\date{\today}

\thanks{
 The work of the fourth author was sponsored by the National Security Agency under Grant Number H98230-14-1-0106 and NSF DMS-1643552.  The United States Government is authorized to reproduce and distribute reprints notwithstanding any copyright notation herein.
The work of the second author was supported by NSF DMS-1501970 and a Sloan Research Fellowship.
}

\begin{abstract}
  We examine a pair of dynamical systems on the plane induced by a pair of spanning trees in the Cayley graph of the Super-Apollonian group of Graham, Lagarias, Mallows, Wilks and Yan.  The dynamical systems compute Gaussian rational approximations to complex numbers and are ``reflective'' versions of the complex continued fractions of A. L. Schmidt.  They also describe a reduction algorithm for Lorentz quadruples, in analogy to work of Romik on Pythagorean triples.  For these dynamical systems, we produce an invertible extension and an invariant measure, which we conjecture is ergodic.  We consider some statistics of the related continued fraction expansions, and we also examine the restriction of these systems to the real line, which gives a reflective version of the usual continued fraction algorithm. Finally, we briefly consider an alternate setup corresponding to a tree of Lorentz quadruples ordered by arithmetic complexity.
\end{abstract}

\maketitle

\section{Introduction}

This paper unites three distinct lines of research that have appeared over the past 40 years.  In what was chronologically the first, Asmus Schmidt developed a theory of Gaussian complex continued fractions based on \emph{Farey circles and triangles} analogous to the Farey subdivision of the real line \cite{S1, S2, S3, S4}.   In the second, Graham, Lagarias, Mallows, Wilks and Yan laid much of the algebraic groundwork for the study of Apollonian circle packings and super-packings \cite{GLMWY0, GLMWY1, GLMWY2}. In the third, Romik studied a dynamical system on the tree of Pythagorean triples which is conjugate to a Euclidean algorithm related to the Gauss map for real continued fractions \cite{R}.  

In this work, we define a dynamical system on the complex plane which can be viewed in at least three distinct ways:  as a ``reflective'' version of Gaussian complex continued fractions; as a system of reduction on the Descartes quadruples of all Apollonian circle packings; or as a dynamical system on the tree of Lorentz quadruples, i.e. solutions to $x^2 + y^2 + z^2 = t^2$.  Restricted to the real line, it produces a reflective continued fraction algorithm and Euclidean algorithm.

Our work fits into a long line of work on the number theoretical aspects of Apollonian circle packings.  One of the first papers on this subject was \cite{GLMWY0}, paving the way for many subsequent works on the arithmetic of such packings.  Indeed, many of the results in these subsequent works started as conjectures in \cite{GLMWY0} (see \cite{Fuchsbull} and \cite{Kontorovichbull} for a summary of these results).  By and large, the arithmetic questions inspired by the work of Graham et. al. have been of the following nature.  First, one observes that certain Apollonian circle packings are \emph{integral}, meaning their curvatures (inverse radii) are all integral.  Consider the set of integer curvatures (counted with or without multiplicity) appearing in such a fixed primitive integral Apollonian packing.  What can one say about this set of integers: how many primes of bounded size are there, can one describe the integers using local conditions alone, and what should these local conditions be, if so?

While the question of rational approximation of complex numbers is quite different in flavor than the number theoretic questions just described, it turns out both questions hang on studying the same symmetry group.  In the study of integer curvatures this is the Apollonian group, which lives naturally inside the Super-Apollonian group of \cite{GLMWY1,GLMWY2}, which, in turn, describes the collection of all integral Apollonian circle packings simultaneously.  However, the Super-Apollonian group can be realized as an index $48$ subgroup of the extended Bianchi group $\PGL_2(\ZZ[i])\rtimes\langle\mf{c}\rangle$.  This Bianchi group, analogous to $\PGL_2(\ZZ)$, is the natural setting in which to study Gaussian complex continued fraction expansions.  As an indication of this phenomenon, the Farey circles and triangles of Schmidt (see Figure \ref{fig:NS1}), when iterated, actually produce an Apollonian super-packing as illustrated in \cite{GLMWY2} or Figure \ref{fig:gaussianpacking5}.

The Super-Apollonian continued fraction expansion we describe in this paper was, however, initially inspired by a process described by Romik in \cite{R}.  In that paper, Romik used the group generated by 
\begin{equation*}
{
\gamma_1=\left(
\begin{array}{lll}
-1&2&2\\
-2&1&2\\
-2&2&3\\
\end{array}
\right),\quad
\gamma_2=\left(
\begin{array}{lll}
1&2&2\\
2&1&2\\
2&2&3\\
\end{array}
\right)},\quad 
\gamma_3=\left(
\begin{array}{lll}
1&-2&2\\
2&-1&2\\
2&-2&3\\
\end{array}
\right)
\end{equation*}
to create a dynamical system on the positive quadrant of the unit circle.  This is a subgroup of the orthogonal group $\textrm{O}_F(\mathbb Z)$ where $F(x,y,z)=x^2+y^2-z^2$, and any point $(x,y)$ on the unit circle corresponds to the solution $(x,y,1)$ of the equation $F(x,y,z)=0$, or the Pythagorean equation.  Indeed, the dynamical system Romik constructed acts naturally to walk to the root of a tree of primitive integer Pythagorean triples as depicted in Figure \ref{fig:pyth-tree}.

\begin{figure}
\begin{forest}
[{$\scriptscriptstyle{(3,4,5)}$} 
    [{$\scriptscriptstyle{(15,8,17)}$}, edge label= {node[midway,below] {$\scriptscriptstyle{\gamma_1}$}} 
      [{$\scriptscriptstyle{(35,12,37)}$},edge label={node[midway,left] {$\scriptscriptstyle{\gamma_1}$}} ] 
      [{$\scriptscriptstyle{(65,72,97)}$},edge label={node[midway,left] {$\scriptscriptstyle{\gamma_2}$}}] 
      [{$\scriptscriptstyle{(33,56,65)}$},edge label={node[midway,left] {$\scriptscriptstyle{\gamma_3}$}}]
    ]
    [{$\scriptscriptstyle{(21,20,29)}$}, edge label={node[midway,left] {$\scriptscriptstyle{\gamma_2}$}}
      [{$\scriptscriptstyle{(77,36,85)}$},edge label={node[midway,left] {$\scriptscriptstyle{\gamma_1}$}}] 
      [{$\scriptscriptstyle{(119,120,169)}$},edge label={node[midway,left] {$\scriptscriptstyle{\gamma_2}$}}] 
      [{$\scriptscriptstyle{(39,80,89)}$},edge label={node[midway,left] {$\scriptscriptstyle{\gamma_3}$}}]
  ] 
     [{$\scriptscriptstyle{(5,12,13)}$}, edge label={node[midway,below] {$\scriptscriptstyle{\gamma_3}$}}
      [{$\scriptscriptstyle{(45,28,53)}$},edge label={node[midway,left] {$\scriptscriptstyle{\gamma_1}$}}] 
      [{$\scriptscriptstyle{(55,48,73)}$},edge label={node[midway,left] {$\scriptscriptstyle{\gamma_2}$}}] 
      [{$\scriptscriptstyle{(7,24,25)}$},edge label={node[midway,left] {$\scriptstyle{\gamma_3}$}}]
  ] 
]
\end{forest}
\caption{Tree of Pythagorean Triples, shown to the third level}
\label{fig:pyth-tree}
\end{figure}
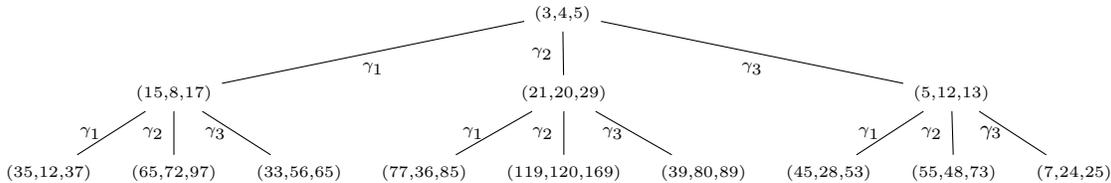

The fact that Pythagorean triples form a tree has been observed many times; see \cite{Berggren,Barning,Hall,Kanga,Price}.  Given any primitive Pythagorean triple $(a,b,c)$, which corresponds to the rational point $(a/c,b/c)$ on the unit circle, Romik's dynamical system allows one to quickly determine the finite word in $\gamma_1,\gamma_2,\gamma_3$ which describes the path in the tree from $(3,4,5)$ to $(a,b,c)$.  It also associates an \emph{infinite} word or expansion in $\gamma_1,\gamma_2,\gamma_3$ to any irrational point on the first quadrant of the unit circle, and this gives a type of continued fraction expansion.  We refer the reader to \cite{R} for further details on this beautiful work.

The present paper can be thought of as a version of Romik's work in four dimensions.  As we discuss below, Apollonian packings are naturally connected to primitive \emph{Lorentz quadruples}, or coprime quadruples of integers satisfying $x^2+y^2+z^2=t^2.$  In this higher dimensional setting, a number of new difficulties arise: for example, the analog of the ternary Pythagorean tree above is no longer a tree, and there is not a unique path from any given vertex to the root (see Figure~\ref{fig:lor-graph} for a picture of the graph).  

We will now briefly describe the idea behind Super-Apollonian continued fractions.  An Apollonian packing is obtained by starting with four pairwise tangent circles,  and repeatedly inscribing into  every region bounded by three existing pairwise tangent circles the unique circle which is tangent to all three.  The number theoretical interest in these packings comes from the fact that if the starting four circles all have integer curvature (i.e. the reciprocal of each of the radii is an integer), then all of the circles in the packing have integer curvature.  This process is depicted in Figure~\ref{pack}.  
\begin{figure}[htp]
\centering
\includegraphics[height = 35 mm]{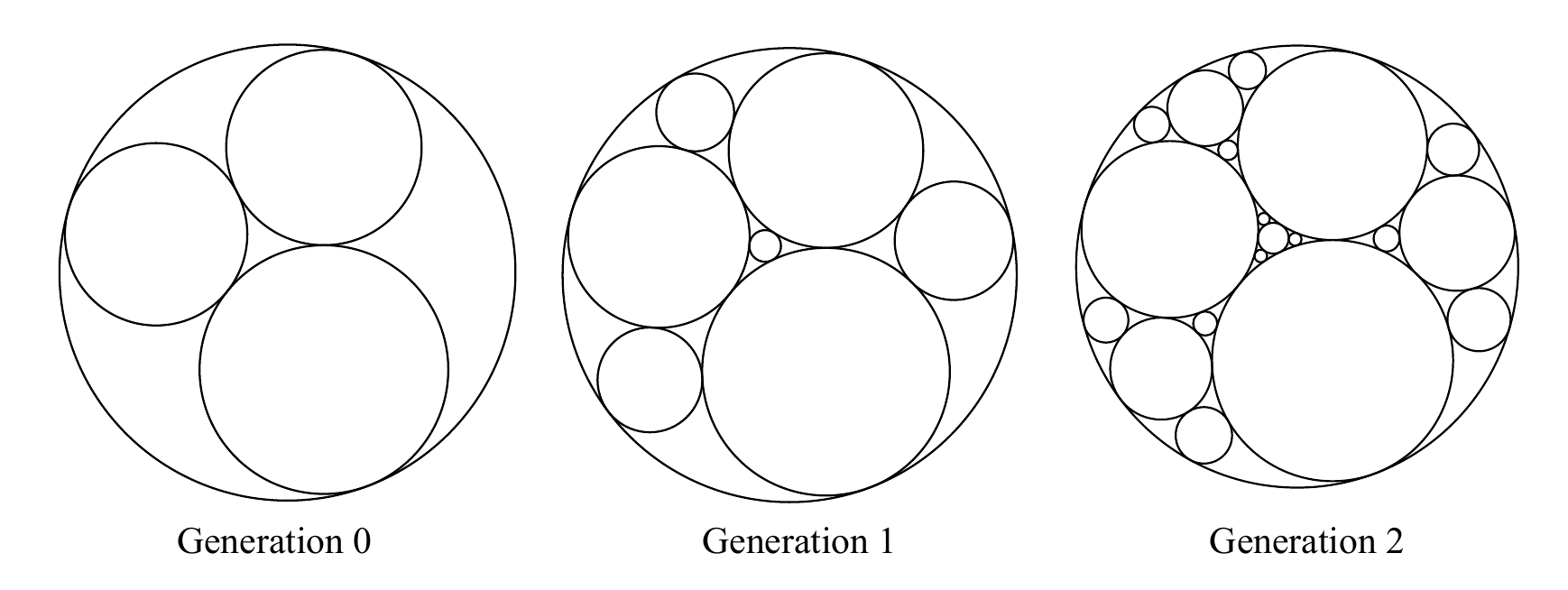}
\caption{Constructing an Apollonian packing}\label{pack}
\end{figure}

These packings have a symmetry which can be described both algebraically and geometrically.  The algebraic interpretation stems from Descartes' theorem from 1643, that if four pairwise circles have curvatures $a,b,c,d$, then $Q(a,b,c,d)=2(a^2+b^2+c^2+d^2)-(a+b+c+d)^2=0$; the symmetries of the packing are symmetries of this quadratic form.  On the other hand, geometrically, one can embed the Apollonian circle packing in the complex plane and describe the symmetries as M\"obius transformations.

While the curvatures of circles in any one primitive integral Apollonian packing give only a thin subset of all primitive integer solutions to $Q(a,b,c,d)=0$, one can use the geometric interpretation of the algebraic symmetry to choose a natural extension to the Super-Apollonian group of \cite{GLMWY1}.  The Super-Apollonian group acts on Descartes quadruples.  The orbit of one Descartes quadruple under this larger group will now include all primitive integral Apollonian circle packings, nested one inside another densely in the complex plane; see Figure \ref{fig:gaussianpacking5}.  

To approximate a complex number $z \in \CC$, we approach $z$ with a sequence of Descartes quadruples of circles, in the sense that the tangency points of the quadruples (all of which are rational) approach $z$.  The sequence of quadruples is generated by repeated applications of generators of the Super-Apollonian group, so that the continued fraction expansion is a word in the generators.  Figure~\ref{approxpiei} depicts this process when approximating the number $\pi+e i$.
In contrast to the real or Pythagorean case, since the Super-Apollonian group is not free, the Cayley graph is not a tree, and there are multiple approximation sequences for a given $z \in \CC$; a natural choice is given by the normal form of Graham, Lagarias, Mallows, Wilks and Yan \cite{GLMWY1}.


\begin{figure}[htp]
\centering
\includegraphics[height = 75 mm]{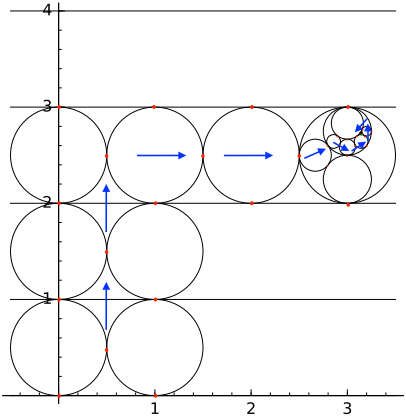}
\caption{Estimating $\pi+ei$ using tangency points in an Apollonian super-packing}\label{approxpiei}
\end{figure}

The Apollonian super-packing of Figure \ref{fig:gaussianpacking5} is the same fractal subdivision of $\CC$ used in the Gaussian continued fraction algorithm of Schmidt.  However, Schmidt's setup involves a different choice of group and generators to navigate the subdivision.  We develop a dynamical system, based on the Super-Apollonian group, which acts on the plane to generate the continued fraction algorithm; define an invertible extension to an action on hyperbolic geodesics; and find an invariant measure.  This process closely follows the development of Schmidt's system by Schmidt and Nakada, as detailed in an appendix.

Having described these dynamical systems, we go on to analyse some statistics of the continued fraction expansion.  We also include an analysis of the restriction of the Super-Apollonian dynamics to the real line, where we recover ``reflective'' real continued fractions.  Finally, in part to emphasize the variability in possible systems, we briefly consider another dynamical system with the goal of organizing Lorentz quadruples by arithmetic complexity.

In future work, we plan to provide a proof of ergodicity for our continued fraction algorithm, and to develop parallel theories for other imaginary quadratic fields such as $\QQ(\sqrt{-2})$.


The paper proceeds as follows.  Section \ref{sec:back} contains background information on classical continued fractions, Lorentz and Descartes quadruples, the Apollonian and Super-Apollonian groups and their geometric interpretation as M\"obius transformations of the complex plane.  This section includes a discussion of Graham, Lagarias, Mallows, Wilks and Yan's swap normal form and the associated spanning tree of the Cayley graph of the Super-Apollonian group.  Section \ref{sec:dyn} defines a dynamical system associated to this choice of spanning tree.  First a pair of dynamical systems on the complex plane are described, which compute a Gaussian continued fraction expansion of a complex number.  We describe an invertible extension and invariant measure.  In Section \ref{sec:lor-des}, the closely related dynamical systems of Lorentz and Descartes quadruples are described.  These systems travel to the root along the the swap normal and invert normal form spanning trees.  We relate these to the Reduction Algorithm for Descartes quadruples of Graham, Lagarias, Mallows, Wilks and Yan.  In Section \ref{sec:app}, we give some consequences of the conjectural ergodicity of the systems, and compare to experimental data.  We plan to demonstrate ergodicity in a follow-up paper.  In Section \ref{sec:real}, we restrict this system to the real line, and recover a ``reflective'' variation on the classical continued fraction algorithm.  In Section \ref{sec:lor}, we briefly consider an alternate dynamical system which respects the arithmetic complexity of Lorentz quadruples.  Finally, in an appendix, we review the Gaussian continued fraction algorithm of Schmidt and Nakada, and its explicit connection to our work.  

\subsection*{A note on figures}  Figures and experimental data were created with Sage Mathematics Software \cite{SAGE} and Mathematica \cite{MATH}.

\subsection*{Acknowledgements} We would like to thank Dan Romik for writing the paper that inspired this work, and for several helpful conversations.  We also thank Jayadev Athreya for his helpful comments and suggestions.

\section{Quadruples and the Super-Apollonian Group}\label{sec:back}

\subsection{Simple continued fractions}

We begin with a brief overview of simple continued fractions on the real line, described from the perspective and in the language we plan to use for complex continued fractions in the remainder of this paper.  The purpose is to provide an explicit analogy for much of what follows.

Over the integers, the Euclidean algorithm is the iteration of the division algorithm, which, for $a,b \in \ZZ$, returns $q, r \in \ZZ$ so that $a=bq+r$, $0 \le r < |b|$.  Replacing the pair $(a,b)$ with $(b,r)$, we repeat.  We can illustrate this as $(a,b)\overset{q}{\mapsto}(b,r)$, e.g.
$$
(355,113)\overset{3}{\mapsto}(113,16)\overset{7}{\mapsto}(16,1)\overset{16}{\mapsto}(1,0).
$$
The transformation taking $(a,b)$ to $(b,r)$ can be written in terms of matrices as
$$
\begin{pmatrix}  b \\ r \end{pmatrix}
=\left(
\begin{array}{cc}
0&1\\
1&0\\
\end{array}
\right)
\left(
\begin{array}{cc}
1&-q\\
0&1\\
\end{array}
\right)
\begin{pmatrix}  a \\ b \end{pmatrix}
=\begin{pmatrix}
  0 & 1 \\ 1 & -q
\end{pmatrix}
\begin{pmatrix}  a \\ b \end{pmatrix}
$$
If $a$, $b$ are coprime, the Euclidean algorithm terminates at $(1,0)$.  We can word backwards from $(1,0)$ to recover $(a,b)$, e.g.
$$
\begin{pmatrix}  355 \\ 113 \end{pmatrix}
=\left(
\begin{array}{cc}
3&1\\
1&0\\
\end{array}
\right)
\left(
\begin{array}{cc}
7&1\\
1&0\\
\end{array}
\right)
\left(
\begin{array}{cc}
16&1\\
1&0\\
\end{array}
\right)
\begin{pmatrix}  1 \\ 0 \end{pmatrix}
$$
The Euclidean algorithm gives rise to a map on rationals $$\frac{b}{a} \overset{q}{\mapsto} \frac{r}{b} = \frac{a}{b} - q,$$ (say with $a,b$ coprime), and the algorithm produces expressions such as
$$
\frac{355}{113}=3+\frac{1}{7+\frac{1}{16}}.
$$
We can extend this map to the interval $(0,1)$, i.e. the Gauss map $T(x)=\{1/x\}$ mapping $x\in(0,1)$ to the fractional part of $1/x$, which is defined piecewise by M\"{o}bius transformations
$$
T(x)=\frac{-qx+1}{x}, \ x\in\left[\frac{1}{q+1},\frac{1}{q}\right).
$$
The effect of this is an infinite continued fraction
$$
x=[a_0;a_1,\ldots,a_n,\ldots]=a_0+\frac{1}{a_1+\frac{1}{a_2+\dots}}, \ a_0=\lfloor x\rfloor, \ a_n=\left\lfloor\frac{1}{T^{n-1}(x-a_0)}\right\rfloor,
$$
and a sequence of convergents $p_n/q_n$ to $x$:
$$
\left(
\begin{array}{cc}
p_n&p_{n-1}\\
q_n&q_{n-1}\\
\end{array}
\right)
=
\left(
\begin{array}{cc}
a_0&1\\
1&0\\
\end{array}
\right)
\left(
\begin{array}{cc}
a_1&1\\
1&0\\
\end{array}
\right)\cdots
\left(
\begin{array}{cc}
a_n&1\\
1&0\\
\end{array}
\right), \ \lim_{n\to\infty}\frac{p_n}{q_n}=x.
$$
For irrational $x\in(0,1)$, $T$ is the left shift map on $\mbb{N}^{\mbb{N}}$, $T([0;a_1,a_2,\ldots])=[0;a_2,a_3,\ldots]$.

The Gauss map is $\infty$-to-1, but we can find an invertible extension $\widetilde{T}$ defined on pairs $(y,x)\in(-\infty,-1)\times(0,1)$,
$$
\widetilde{T}(y,x)=(1/y-\lfloor x \rfloor,T(x)).
$$
The pairs $(y,x)$ can be identified with oriented geodesics in the hyperbolic plane, realized as the upper half plane above the real line.  Specifically, the points $y$ and $x$ are the past and future points on the ideal boundary in the upper half-plane model.  The extension $\widetilde{T}$ acts piecewise by isometries, $\text{Isom}(\mbb{H}^2)\cong\PGL_2(\mbb{R})$. (The space of geodesics $(-\infty,-1)\times(0,1)$ corresponds to Gauss' reduced indefinite binary quadratic forms.)  The space of geodesics carries an isometry invariant measure $\frac{dxdy}{(x-y)^2}$ coming from hyperbolic area or Haar measure on $\operatorname{SL}_2(\mbb{R})$, which restricts to a $\widetilde{T}$-invariant measure on $(-\infty,-1)\times(0,1)$ since $\widetilde{T}$ is defined piecewise by isometries.  Pushing this measure forward to the second coordinate gives the $T$-invariant Gauss measure $d\mu(x)=\frac{dx}{1+x}$ on $(0,1)$, which is ergodic:
$$
d\mu(x)=dx\int_{-\infty}^{-1}\frac{dy}{(x-y)^2}.
$$
For more on continued fractions, see \cite{Kh} (basic properties), \cite{wS} (Diophantine approximation), \cite{B} (ergodicity), \cite{EW} (connection to the geodesic flow and hyperbolic dynamics).  For an excellent exposition of this view of the Gauss measure, see \cite{titbit}.

\subsection{Lorentz and Descartes quadruples}\label{sec:quads}

We consider the Lorentz $(1,3)$--form 
\[
        Q_L(\mathbf{x}) =  x_0^2 - x_1^2 - x_2^2 - x_3^2,
\]
and the solutions to $Q_L(\mathbf{x})=0$ are called \emph{Lorentz quadruples}.  Consider the following matrices:
\[
        L_1 := \begin{pmatrix}
                2 & -1 & -1 & -1 \\
                1 & 0 & -1 & -1 \\
                1 & -1 & 0 & -1 \\
                1 & -1 & -1 & 0
        \end{pmatrix},
        L_2 := \begin{pmatrix}
                2 & -1 & 1 & 1 \\
                1 & 0 & 1 & 1 \\
                -1 & 1 & 0 & -1 \\
                -1 & 1 & -1 & 0
        \end{pmatrix},
\]
\[
        L_3 := \begin{pmatrix}
                2 & 1 & -1 & 1 \\
                -1 & 0 & 1 & -1 \\
                1 & 1 & 0 & 1 \\
                -1 & -1 & 1 & 0
        \end{pmatrix},
        L_4 := \begin{pmatrix}
                2 & 1 & 1 & -1 \\
                -1 & 0 & -1 & 1 \\
                -1 & -1 & 0 & 1 \\
                1 & 1 & 1 & 0
        \end{pmatrix},
\]
\[
        L_1^\perp := \begin{pmatrix}
                2 & 1 & 1 & 1 \\
                -1 & 0 & -1 & -1 \\
                -1 & -1 & 0 & -1 \\
                -1 & -1 & -1 & 0
        \end{pmatrix},
        L_2^\perp := \begin{pmatrix}
                2 & 1 & -1 & -1 \\
                -1 & 0 & 1 & 1 \\
                1 & 1 & 0 & -1 \\
                1 & 1 & -1 & 0
        \end{pmatrix},
\]
\[
        L_3^\perp := \begin{pmatrix}
                2 & -1 & 1 & -1 \\
                1 & 0 & 1 & -1 \\
                -1 & 1 & 0 & 1 \\
                1 & -1 & 1 & 0
        \end{pmatrix},
        L_4^\perp := \begin{pmatrix}
                2 & -1 & -1 & 1 \\
                1 & 0 & -1 & 1 \\
                1 & -1 & 0 & 1 \\
                -1 & 1 & 1 & 0
        \end{pmatrix}.
\]
We create these matrices by conjugating $L_1$ by all eight possible diagonal matrices with diagonals $(1, \pm 1, \pm 1, \pm 1)$.  Each of the eight resulting matrices is an involution, and preserves $Q_L$ in the sense that
\[
  L_i^T G_L L_i = G_L, \quad (L_i^\perp)^T G_L L_i^\perp = G_L
\]
for the Gram matrix
\[
  G_L = \begin{pmatrix} 1 & 0 & 0 & 0 \\
    0 & -1 & 0 & 0 \\
    0 & 0 & -1 & 0 \\
    0 & 0 & 0 & -1
  \end{pmatrix}
\]
of $Q_L$.

We create a Cayley graph, $\mathcal{C}_L$, whose vertices are the elements of the group generated by the $L_i, L_i^\perp$ and where two vertices $M_1$ and $M_2$ are joined by an edge exactly when $M_1 M_2^{-1}$ is one of the $L_i, L_i^\perp$.  If one takes a quotient of this Cayley graph by considering the action of the $L_i, L_i^\perp$ on Lorentz quadruples, one obtains Figure \ref{fig:lor-graph}.

\begin{figure}[htp]
\begin{center}
\begin{tikzpicture}
    \node (r) at ( 0, 0) {$\scriptscriptstyle{(1,1,0,0)}$}; 
    \node (r3) at ( 0, 3) {$\scriptscriptstyle{(3,-1,2,-2)}$};
    \node (r4) at ( 3,0) {$\scriptscriptstyle{(3,-1,-2,2)}$};
    \node (r1t) at ( 0,-3) {$\scriptscriptstyle{(3,-1,-2,-2)}$};
    \node (r2t) at ( -3,0) {$\scriptscriptstyle{(3,-1,2,2)}$};
     \node (r13) at ( -3, 3) {$\scriptscriptstyle{(7,3,6,2)}$}; 
    \node (r23) at ( -2,5) {$\scriptscriptstyle{(7,3,-2,6)}$};
    \node (r43) at ( 0,6) {$\scriptscriptstyle{(9,-7,-4,4)}$};
    \node (r1t3) at ( 2,5.5) {$\scriptscriptstyle{(5,-3,0,-4)}$};
    \node (r2t3) at ( 2.5,4.5) {$\scriptscriptstyle{(5,-3,4,0)}$};
    \node (r3t3) at ( 3,3) {$\scriptscriptstyle{(11,7,-6,6)}$};
    \node (r14) at ( 6,0) {$\scriptscriptstyle{(7,3,2,6)}$};
     \node (r24) at ( 5.3,.9) {$\scriptscriptstyle{(7,3,-6,-2)}$};
      \node (r34) at ( 4.5,1.6) {$\scriptscriptstyle{(9,-7,4,-4)}$};
       \node (r1t4) at ( 4,-2) {$\scriptscriptstyle{(5,-3,-4,0)}$};
        \node (r2t4) at ( 5.3,-.9) {$\scriptscriptstyle{(5,-3,0,4)}$};
         \node (r4t4) at ( 4.5,-1.6) {$\scriptscriptstyle{(11,7,6,-6)}$}; 
            \node (r11t) at ( 1.7,-5.4) {$\scriptscriptstyle{(11,7,6,6)}$};
        \node (r2t1t) at ( -1.7,-5.4) {$\scriptscriptstyle{(9,-7,4,4)}$};
         \node (r3t1t) at ( 3,-4.4) {$\scriptscriptstyle{(7,3,-6,2)}$}; 
 \node (r4t1t) at ( -3,-4.4) {$\scriptscriptstyle{(7,3,2,-6)}$}; 
 \node (r22t) at (-5.7,0) {$\scriptscriptstyle{(11,7,-6,-6)}$};
        \node (r1t2t) at ( -5.7,1.7) {$\scriptscriptstyle{(9,-7,-4,-4)}$};
         \node (r3t2t) at ( -5.7,-1.7) {$\scriptscriptstyle{(7,3,-6,2)}$};

    \begin{scope}[every path/.style={->}]
       
       \draw[red] (r) --  (r3) node [color=black,midway, right, fill=white] {$\scriptscriptstyle{L_3}$};
       \draw (r) -- (r4) node [midway, above, fill=white] {$\scriptscriptstyle{L_4}$};
       \draw[blue] (r) -- (r1t) node [color=black,midway, right, fill=white] {$\scriptscriptstyle{L_1^\perp}$};
       \draw (r) -- (r2t) node [midway, above, fill=white] {$\scriptscriptstyle{L_2^\perp}$};
       \draw (r3) -- (r13) node [midway, below, fill=white] {$\scriptscriptstyle{L_1}$};
        \draw (r3) -- (r23) node [midway, below, ] {$\scriptscriptstyle{L_2}$};
        \draw (r2t) -- (r23) node [very near end, left, ] {$\scriptscriptstyle{L_4^\perp}$};
         \draw (r3) -- (r43) node [midway, left, fill=white] {$\scriptscriptstyle{L_4}$};
          \draw[red] (r3) -- (r1t3) node [color=black,midway, left] {$\scriptscriptstyle{L_1^\perp}$};
           \draw[blue] (r1t) .. controls +(right:4.1cm) and +(right:4.1cm) ..  (r1t3) node [color=black,near end, right] {$\scriptscriptstyle{L_3}$};
           \draw (r3) -- (r2t3) node [midway, above] {$\scriptscriptstyle{L_2^\perp}$};
\draw (r2t) -- (r2t3) node [very near start, left] {$\scriptscriptstyle{L_3}$};
            \draw (r3) -- (r3t3) node [midway, below, fill=white] {$\scriptscriptstyle{L_3^\perp}$};
             \draw (r4) --  (r14) node [near end, above] {$\scriptscriptstyle{L_1}$};
              \draw (r4) --  (r24) node [midway, right] {$\scriptscriptstyle{L_2}$};
               \draw (r4) --  (r34) node [midway, right] {$\scriptscriptstyle{L_3}$};
                \draw (r4) --  (r1t4) node [midway, left] {$\scriptscriptstyle{L_1^\perp}$};
                 \draw (r4) --  (r2t4) node [midway, right] {$\scriptscriptstyle{L_2^\perp}$};
                  \draw (r4) --  (r4t4) node [midway, right] {$\scriptscriptstyle{L_4^\perp}$};
                  \draw (r1t) -- (r1t4) node [midway, above] {$\scriptscriptstyle{L_4}$};
                  \draw (r1t) -- (r11t) node [midway, above] {$\scriptscriptstyle{L_1}$};
                  \draw (r1t) -- (r2t1t) node [midway, above] {$\scriptscriptstyle{L_2^\perp}$};
                  \draw (r1t) -- (r3t1t) node [midway, above] {$\scriptscriptstyle{L_3^\perp}$};
                  \draw (r1t) -- (r4t1t) node [midway, above] {$\scriptscriptstyle{L_4^\perp}$};
                  \draw (r2t) -- (r22t) node [midway, below] {$\scriptscriptstyle{L_2}$};
                  \draw (r2t) -- (r1t2t) node [very near end, below] {$\scriptscriptstyle{L_1^\perp}$};
                  \draw (r2t) -- (r3t2t) node [midway, below] {$\scriptscriptstyle{L_3^\perp}$};
                  \draw (r2t) .. controls +(down:2.1cm) and +(down:1.7cm) .. (r2t4) node [very near start, below] {$\scriptscriptstyle{L_4}$};
    \end{scope}  
    \path[]
    (r) edge [in=150,out=120,loop, above] node {$\scriptscriptstyle{L_1}$} (r)
    (r) edge [in=30,out=60,loop, above] node {$\scriptscriptstyle{L_2}$} (r)
    (r) edge [in=-60,out=-30,loop, below] node {$\scriptscriptstyle{L_3^\perp}$} (r)
    (r) edge [in=-120,out=-150,loop, below] node {$\scriptscriptstyle{L_4^\perp}$} (r)
    (r3) edge [in=-120,out=-150,loop, left] node {$\scriptscriptstyle{L_4^\perp}$} (r3)
    (r4) edge [loop above] node {$\scriptscriptstyle{L_3^\perp}$} (r4)
     (r1t) edge [in=30,out=60,loop, right] node {$\scriptscriptstyle{L_2}$} (r1t)
     (r2t) edge [in=135,out=105, loop, above] node {$\scriptscriptstyle{L_1}$} (r2t);
    \end{tikzpicture}
\end{center}
\caption{Depth two piece of the Lorentz graph. Red path shows swap normal form expansion of $(5,-3,0,-4)$, blue path shows invert normal form expansion of $(5,-3,0,-4)$.  See Definition \ref{def:normalforms}.}
\label{fig:lor-graph}
\end{figure}
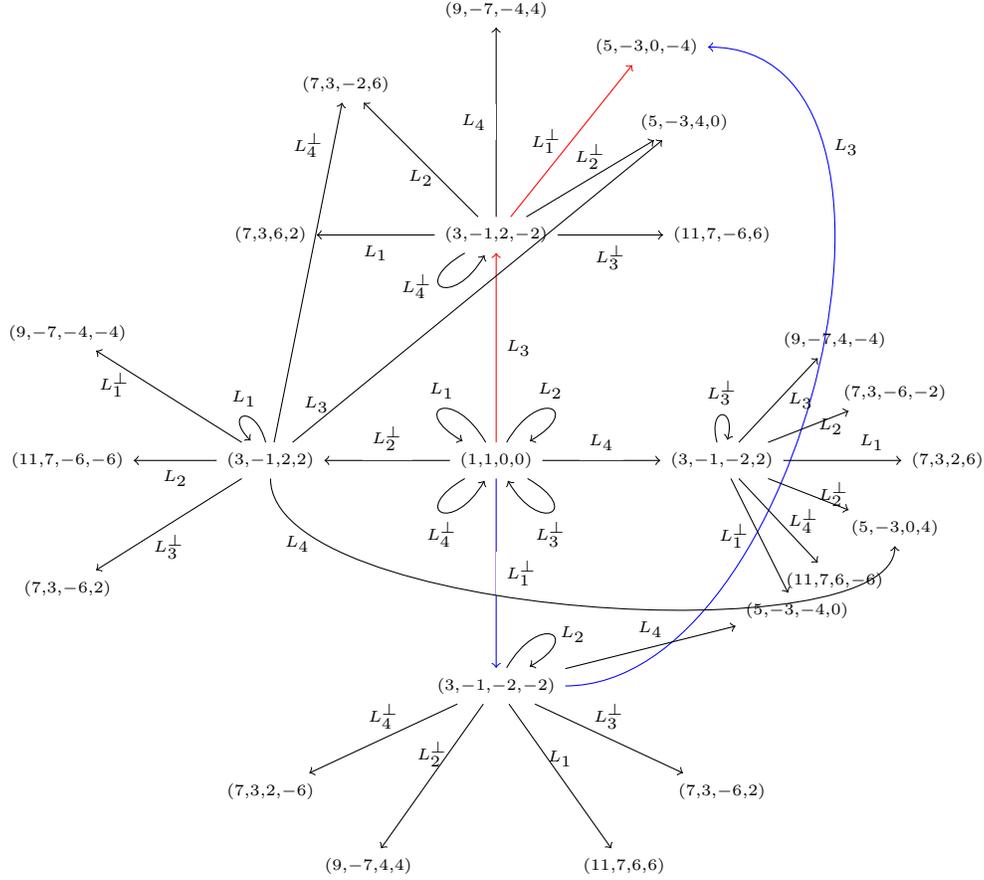

The form $Q_L$ is equivalent to the Descartes form 
$$ Q_D(\mathbf{y}) := (y_0+y_1+y_2+y_3)^2 - 2(y_0^2 + y_1^2 + y_2^2 + y_3^2)$$ 
by the relationship $2Q_D(J\mathbf{x}) = Q_L(\mathbf{x})$ where
\begin{equation}\label{Jmatrix}
        J = \frac{1}{2}
        \begin{pmatrix}
                1 & 1 & 1 & 1 \\
                1 & 1 & -1 & -1 \\
                1 & -1 & 1 & -1 \\
                1 & -1 & -1 & 1 
        \end{pmatrix}.
\end{equation}
Note that $J^{-1}=J$.
The Descartes form has Gram matrix
\[
  G_D = 2 J^T G_L J
  = \begin{pmatrix}
    -1 & 1 & 1 & 1 \\
    1 & -1 & 1 & 1 \\
    1 & 1 & -1 & 1 \\
    1 & 1 & 1 & -1
  \end{pmatrix}
\]
and is central to the study of Apollonian circle packings.  The solutions to the Descartes form $Q_D(\mathbf{y})=0$ are called \emph{Descartes quadruples}.  If one considers the matrices above under the change of variables 
\[
  L_i \mapsto JL_iJ, \quad
  L_i^\perp \mapsto JL_i^\perp J,
\]
one obtains the \emph{Super-Apollonian group} $\Acal^\Scal$ of \cite{GLMWY1}, whose generators are (in order corresponding to the $L_i$ above):
\[
        S_1 := \begin{pmatrix}
                -1 & 2 & 2 & 2 \\
                0 & 1 & 0 & 0  \\
                0 & 0 & 1 & 0  \\
                0 & 0 & 0 & 1
        \end{pmatrix},
        S_2 := \begin{pmatrix}
                1 & 0 & 0 & 0  \\
                 2 & -1 & 2 & 2 \\
                0 & 0 & 1 & 0  \\
                0 & 0 & 0 & 1
        \end{pmatrix},
\]
\[
        S_3 := \begin{pmatrix}
                 1 & 0 & 0 & 0  \\
                0 & 1 & 0 & 0  \\
                 2 & 2 & -1 & 2 \\
                0 & 0 & 0 & 1
         \end{pmatrix},
        S_4 := \begin{pmatrix}
                  1 & 0 & 0 & 0  \\
                0 & 1 & 0 & 0  \\
                0 & 0 & 1 & 0 \\
                 2 & 2 & 2 & -1 \\
         \end{pmatrix},
\]
\[
        S_1^\perp := \begin{pmatrix}
                -1 & 0 & 0 & 0 \\
                2 & 1 & 0 & 0 \\
                2 & 0 & 1 & 0 \\
                2 & 0 & 0 & 1
        \end{pmatrix},
        S_2^\perp := \begin{pmatrix}
                1 & 2 & 0 & 0 \\
                0 & -1 & 0 & 0 \\
                0 & 2 & 1 & 0 \\
                0 & 2 & 0 & 1
        \end{pmatrix},
\]
\[
        S_3^\perp := \begin{pmatrix}
                1 & 0 & 2 & 0 \\
                0 & 1 & 2 & 0 \\
                0 & 0 & -1 & 0 \\
                0 & 0 & 2 & 1 \\
        \end{pmatrix},
        S_4^\perp := \begin{pmatrix}
                1 & 0 & 0 & 2 \\
                0 & 1 & 0 & 2 \\
                0 & 0 & 1 & 2 \\
                0 & 0 & 0 & -1 
        \end{pmatrix}.
\]
The group $\mathcal{A}^S$ preserves $Q_D$ in the sense that $M^tG_DM=G_D$ for all $M\in\mc{A}^S$.  The corresponding Cayley graph $\mathcal{C}_D$ is isomorphic to $\mathcal{C}_L$.

One presentation of the Super-Apollonian group is \cite[Section 6]{GLMWY1}
\[
        \mc{A}^S = \left\langle S_1, S_2, S_3, S_4, S^\perp_1, S^\perp_2, S^\perp_3, S^\perp_4 : S_j^2 = (S^\perp_j)^2 = 1, S_jS^\perp_k = S^\perp_kS_j\; (j \neq k) \right\rangle.
\]
It is a right-angled hyperbolic reflection group generated by involutions $S_i$, $S_i^\perp$.

The group $\mathcal{A}_S$ is of index $48$ in the full orthogonal group of $Q_D$ \cite[Theorem 7.1]{GLMWY2}.

\subsection{Descartes quadruples, normal form and a spanning tree}

The motivation for studying $\mathcal{A}^S$ came from the study of Apollonian circle packings; in this section we describe this geometry.  We consider circles in $\widehat{\CC} = \CC \cup \{\infty\}$ identified with $\PP^1(\CC)$, having projective equation
\[
  a X \overline{X} -  b Y \overline{X} - \overline{b} \overline{Y} X + c Y \overline{Y} = 0,  \quad [X:Y] \in \PP^1(\CC), \ ac-b\bar{b}=-1.
\]
Such a circle is said to have \emph{curvature} $a$ (inverse of the radius in $[X:1]$), \emph{co-curvature} $c$ (curvature of the circle in $[1:Y]$), and \emph{curvature-center} $b$ (center times curvature).  We denote it by a four-dimensional real vector $(c,a,b_1,b_2)$, where $b=b_1+b_2i$.  These are known as ACC coordinates (\emph{augmented curvature-center coordinates}) \cite{GLMWY1}.  When $a=0$, the zero set in $[X:1]$ is a line and $b$ becomes a unit normal vector.

Four circles $C_i$ as row vectors of a $4\times 4$ matrix $C$ in ACC coordinates are said to be in \emph{Descartes configuration} if
$$
C^tG_DC=\left(
\begin{array}{cccc}
0&-4&0&0\\
-4&0&0&0\\
0&0&2&0\\
0&0&0&2\\
\end{array}
\right).
$$
It is a theorem of Graham, Lagarias, Mallows, Wilks and Yan that circles are in \emph{Descartes configuration} according to this algebraic condition if and only if, as circles, they are all mutually tangent with disjoint interiors, where sign of the curvature indicates orientation (hence interior) \cite[Theorem 3.3]{LMW} and \cite[Theorem 3.2]{GLMWY1}.  This is nicely explained by interpreting the Descartes form as a bilinear pairing on circles which measure angle of intersection or hyperbolic distance; see for example \cite{Kocik}.

Since it preserves $Q_D$, the Super-Apollonian group $\mathcal{A}^S$ acts on Descartes quadruples, in the form of such $4 \times 4$ matrices, from the left.
There is a nice interpretation of this action geometrically \cite{GLMWY1}.
For such a configuration of circles, there is a dual Descartes quadruple consisting of circles passing orthogonally through the first quadruple, sharing the same set of tangency points (Figure \ref{fig:swap-invert}).  Then $S_i$ acts as what we call a ``swap" replacing $C_i$ with its inversion in the dual circle orthogonal to the other three, and $S_i^{\perp}$ fixes $C_i$ while replacing the other three circles with their inversions in $C_i$ (we refer to the action of $S_i^{\perp}$ as simply an ``inversion'').  See Figure \ref{fig:swap-invert}.

\begin{figure}[htp]
\begin{center}
\includegraphics[width=2in]{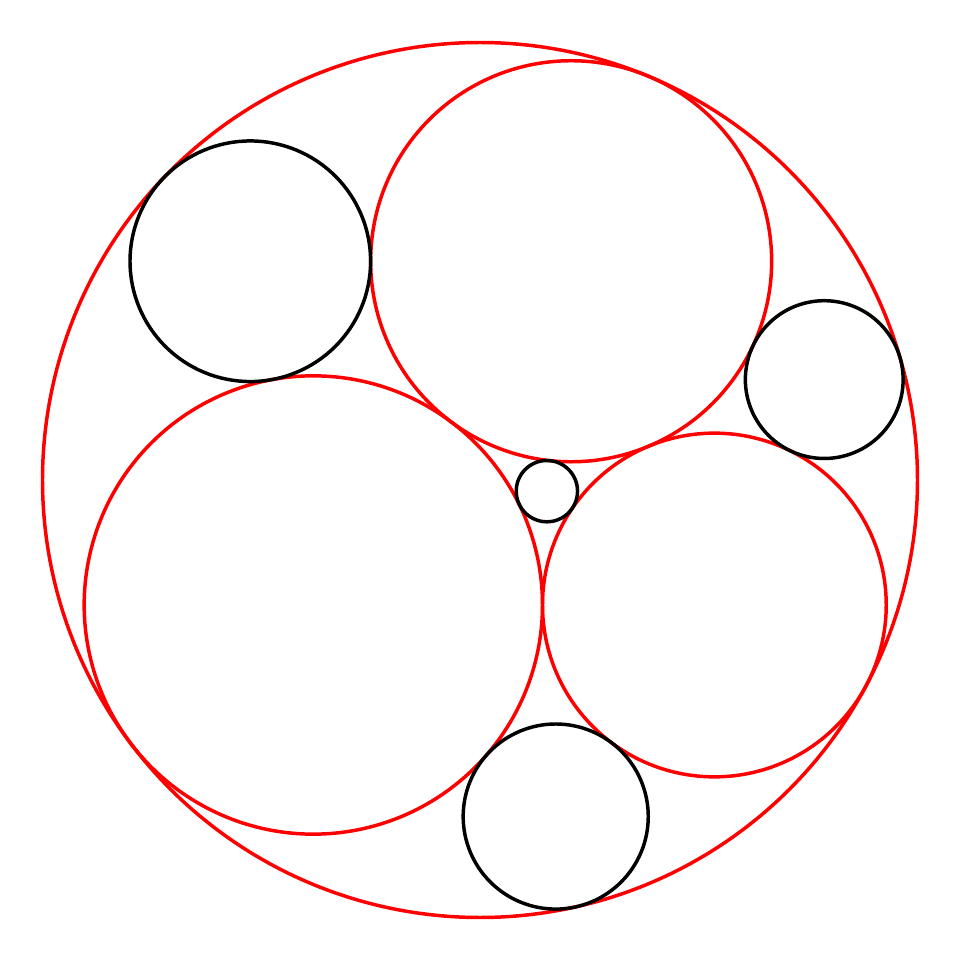} \quad
\includegraphics[width=2in]{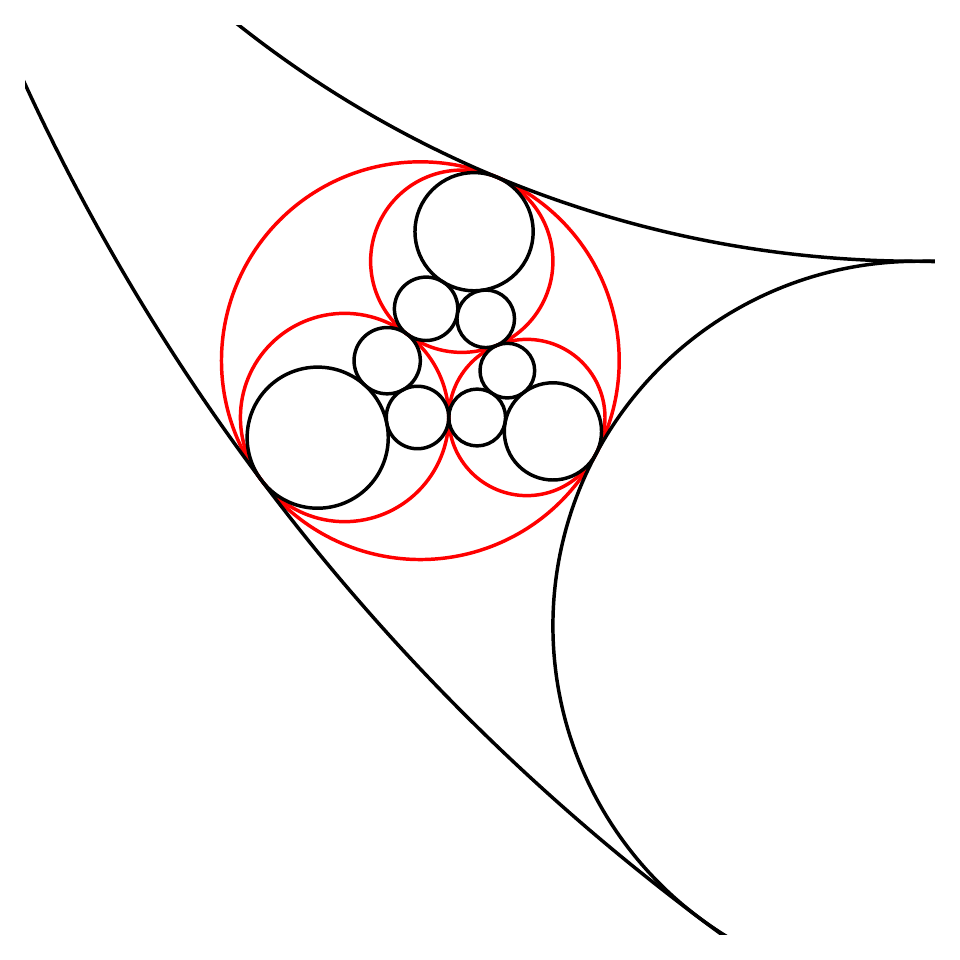} \quad
\includegraphics[width=2in]{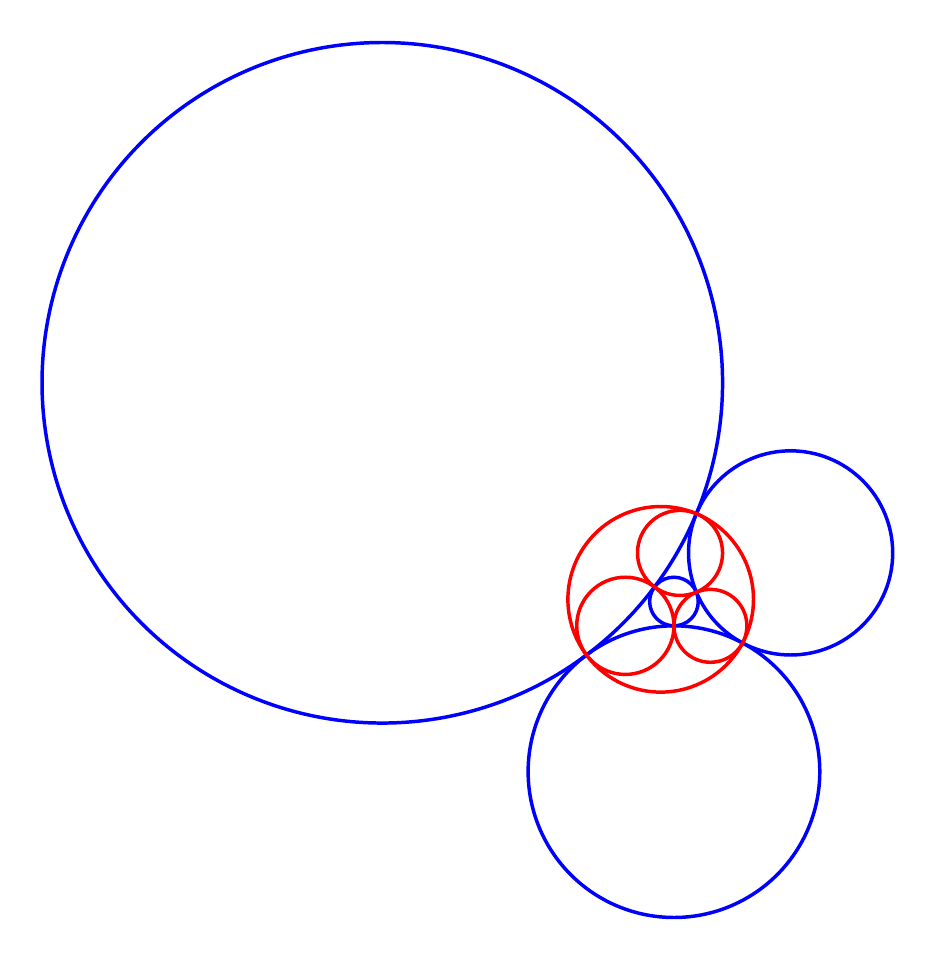}
\caption{In all images, in red, is a Descartes quadruple.  At left, the new circles produced by swaps are shown in black.  At center, the new circles produced by inversions are shown in black.  At right, the dual quadruple is shown in blue.}
\label{fig:swap-invert}
\end{center}
\end{figure}

There are two ``natural'' ways of uniquely writing elements of $\mc{A}^S$ given the commutation relations in the group (i.e. two natural spanning trees for the Cayley graph of $\mc{A}^S$ with respect to the given generators), and we will be working with both of them.  These were first defined in \cite{GLMWY1}.

\begin{defn}
  \label{def:normalforms}
A word $W=M_nM_{n-1}\cdots M_1$ in the Super-Apollonian generators is in {\rm swap normal form} if $M_i\neq M_{i+1}$ and if whenever $M_i=S_j$ and $M_{i+1}=S_k^{\perp}$, then $j=k$; i.e. the ``swaps'' are pushed as far left as possible (equivalently the ``inversions'' are as far right as possible).

A word $W=M_nM_{n-1}\cdots M_1$ in the Super-Apollonian generators is in {\rm invert normal form} if $M_i\neq M_{i+1}$ and if whenever $M_i=S_j^\perp$ and $M_{i+1}=S_k$, then $j=k$; i.e. the ``inversions'' are pushed as far left as possible (equivalently the ``swaps'' are as far right as possible).
\end{defn}

The swap (invert) normal form of an element of the Super-Apollonian group is unique \cite{GLMWY1}.  
In the Cayley graph $\Ccal_Q$, travelling a path of length $n$ to the origin, one can read off labels as $M_n, \cdots, M_1$ where $M_n$ is the distal edge and $M_1$ the proximal (to the origin).  This gives a word $M_n \cdots M_1$ associated to the path.

\begin{defn}
        Define the subgraph $\Tcal_S$ of $\Ccal_Q$ to be the union of all paths to the origin labelled by swap normal form words.  This is called the \emph{swap down tree}.
\end{defn}

The reason for the terminology ``swap down tree'' is that, as one travels toward the origin in $\Ccal_Q$ along the swap down tree, if one has a choice of $S_j^\perp$ followed by $S_i$ or $S_i$ followed by $S_j^\perp$ (both two-move sequences having the same endpoint closer to the origin), one must choose the latter, which is to say, one must swap before inverting.  The following proposition asserts that besides being a spanning tree, it is minimal in a certain way.

\begin{prop}
        \label{prop:swapdowntree}
        The swap down tree $\Tcal_S$ is a spanning tree of $\Ccal_Q$, and, from any vertex, the path to the origin along the tree is of minimal length among paths to the origin in $\Ccal_Q$. 
\end{prop}

\begin{proof}
        Every word can be put into a unique normal form, without increasing its length, by cancelling any double letters and moving each $S_i^\perp$ as far to the right as possible using the commutativity relations \cite[Proof of Theorem 6.1]{GLMWY1}.  
        Therefore there is a unique path to the origin in $\Tcal_S$ from any vertex of $\Ccal_Q$.   We may conclude that $\Tcal_S$ is connected, is a tree, and spans $\Ccal_Q$.  
        Minimality follows from the observation that changing to swap normal form never increases length.
\end{proof}

There are exactly analogous statements for the corresponding \emph{invert down tree}.

\subsection{Geometric realization of the Super-Apollonian group}\label{sec:geosuper}

The group $\PGL_2(\ZZ[i])$ acts on the extended complex plane $\widehat{\CC} = \CC \cup \{ \infty \}$ by the M\"obius action
\[
\begin{pmatrix} \alpha & \gamma \\ \beta & \delta \end{pmatrix} \cdot z = \frac{\alpha z + \gamma }{\beta z + \delta}.
\]
This action can be extended to include complex conjugation,
\[
        \mathfrak{c} \cdot z = \overline{z},
\]
giving rise to the group $B[-1] = \PGL_2(\ZZ[i]) \rtimes\langle \mathfrak{c}\rangle$ of M\"obius transformations, the \textit{extended Bianchi group}, a maximal discrete subgroup of $\PSL_2(\mbb{C})\rtimes\langle\mf{c}\rangle\cong\text{Isom}(H^3)$.  The group $B[-1]$ acts on the collection of circles of $\widehat{\CC}$ (recall that lines are circles through $\infty$).  In what follows, we will identify $\widehat{\CC}$ with $\PP^1(\CC)$ when convenient.

The orbit of the circle $\widehat{\RR} = \RR \cup \{ \infty \}$ under $\PSL_2(\ZZ[i])$ is a dense collection of nested circles called a \emph{Schmidt arrangement}, denoted $\SQ$.  An image is shown in Figure \ref{fig:gaussianpacking5}.  For more on Schmidt arrangements, see \cite{stange1, stange2}. 

To describe our dynamical system, we choose a particular Descartes quadruple $R_B$, and its dual $R_A$, whose circles are the rows of
$$
R_B=\left(
\begin{array}{cccc}
0&0&0&-1\\
2&0&0&1\\
0&2&0&1\\
2&2&2&1\\
\end{array}
\right), \ 
R_A=R_B^{\perp}=\left(
\begin{array}{cccc}
2&2&1&2\\
0&2&1&0\\
2&0&1&0\\
0&0&-1&0\\
\end{array}
\right)
$$
in ACC coordinates (see Figure \ref{fig:basequad}), respectively.  We call these the dual base quadruples.  The terminology refers to the fact that $R_A$ consists of the unique Descartes quadruple orthogonal to $R_B$ and having the same intersection points (and vice versa).  In particular, the ``swaps'' of $R_A$ are the ``inversions'' of $R_B$ and vice versa.  Other choices of base quadruple would of course be possible, but the choice here coincides with a natural subset of the Schmidt arrangement and has particularly simple tangency points.

We embed the Super-Apollonian group into $\PSL_2(\mbb{C})\rtimes\langle\mathfrak{c}\rangle$ using a form of the \emph{exceptional isomorphism} $\PGL_2(\CC) \rtimes \langle \mathfrak{c}\rangle \cong O_{3,1}^+(\RR)$.  To do so, we map each element of $\Acal^\Scal$ to the M\"obius transformation which acts the same way on $R_B$.  The M{\"o}bius transformations corresponding to the Super-Apollonian generators are
$$
\mf{s}_1=\frac{(1+2i)\bar{z}-2}{2\bar{z}-1+2i}, \ \mf{s}_2=\frac{\bar{z}}{2\bar{z}-1}, \ \mf{s}_3=-\bar{z}+2, \ \mf{s}_4=-\bar{z},
$$
$$
\mf{s}_1^{\perp}=\bar{z}, \ \mf{s}_2^{\perp}=\bar{z}+2i, \ \mf{s}_3^{\perp}=\frac{\bar{z}}{-2i\bar{z}+1}, \ \mf{s}_4^{\perp}=\frac{(1-2i)\bar{z}+2i}{-2i\bar{z}+1+2i}
$$
(the $\mf{s}_i$ are inversions in the circles of $R_A$ and the $\mf{s}_i^{\perp}$ are inversions in the circles of $R_B$).  Let $\Gamma$ denote the M\"obius group generated by these generators; it is isomorphic to $\Acal^S$.  Considering the Poincar\'e extension of the M\"obius action to the upper-half-space model of hyperbolic space, one sees that $\Gamma$ is the finite covolume Kleinian group generated by reflections in the sides of a right-angled ideal octahedron whose faces lie on the geodesic planes defined by the circles of $R_B$ and $R_A$.

The orbit of the Super-Apollonian group on a particular Descartes quadruple is known as an Apollonian super-packing \cite{GLMWY1}.  For this choice of base quadruple, as a collection of circles, the corresponding Apollonian super-packing coincides with the Schmidt arrangement of $\SQ$ \cite{stange2}.
This orbit gives a sequence of partitions $\mc{P}_n$ of the plane into triangles and circles, each refining the last (see Figure \ref{fig:gaussianpacking5}).  The regions of the partition $\mc{P}_n$ are indexed by the swap normal form words of length $n$ in the Super-Apollonian generators; see Figure \ref{fig:length2partition}.  A word $W'\in \mc{P}_{n+1}$ refines $W\in\mc{P}_n$ if and only if $W$ is an initial segment of $W'$ (right initial, speaking about $\Acal^S$).  This will allow us to coordinatize the plane by infinite words in swap normal form.  The coordinates of a point will be produced by a dynamical system described below.

\begin{figure}[htp]
\begin{center}
\includegraphics[width=6in]{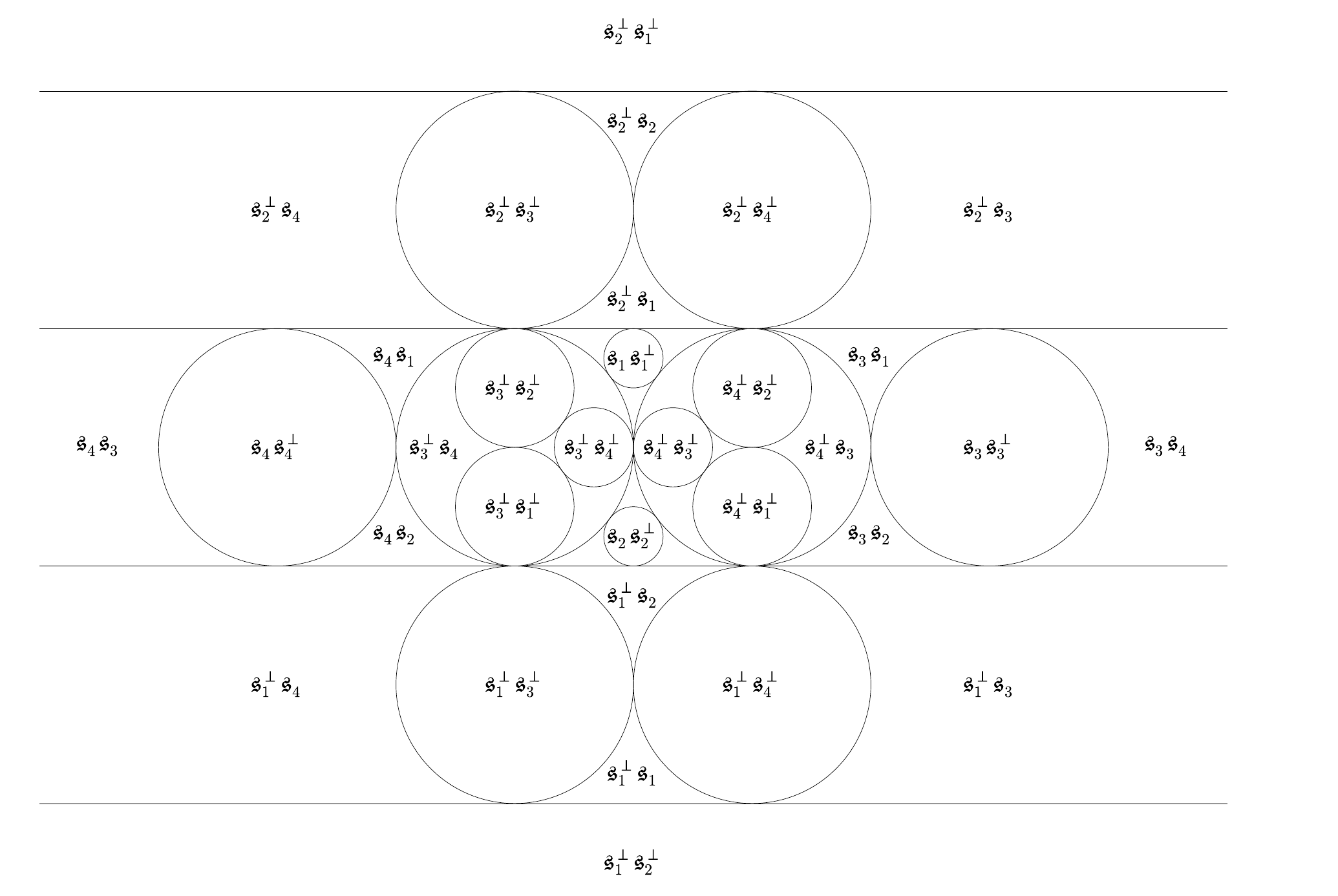}
\caption{Partition of the complex plane induced by words of length two in swap normal form in the M{\"o}bius generators.}
\label{fig:length2partition}
\end{center}
\end{figure}

\begin{figure}[h]
\begin{center}
\includegraphics[width=4in]{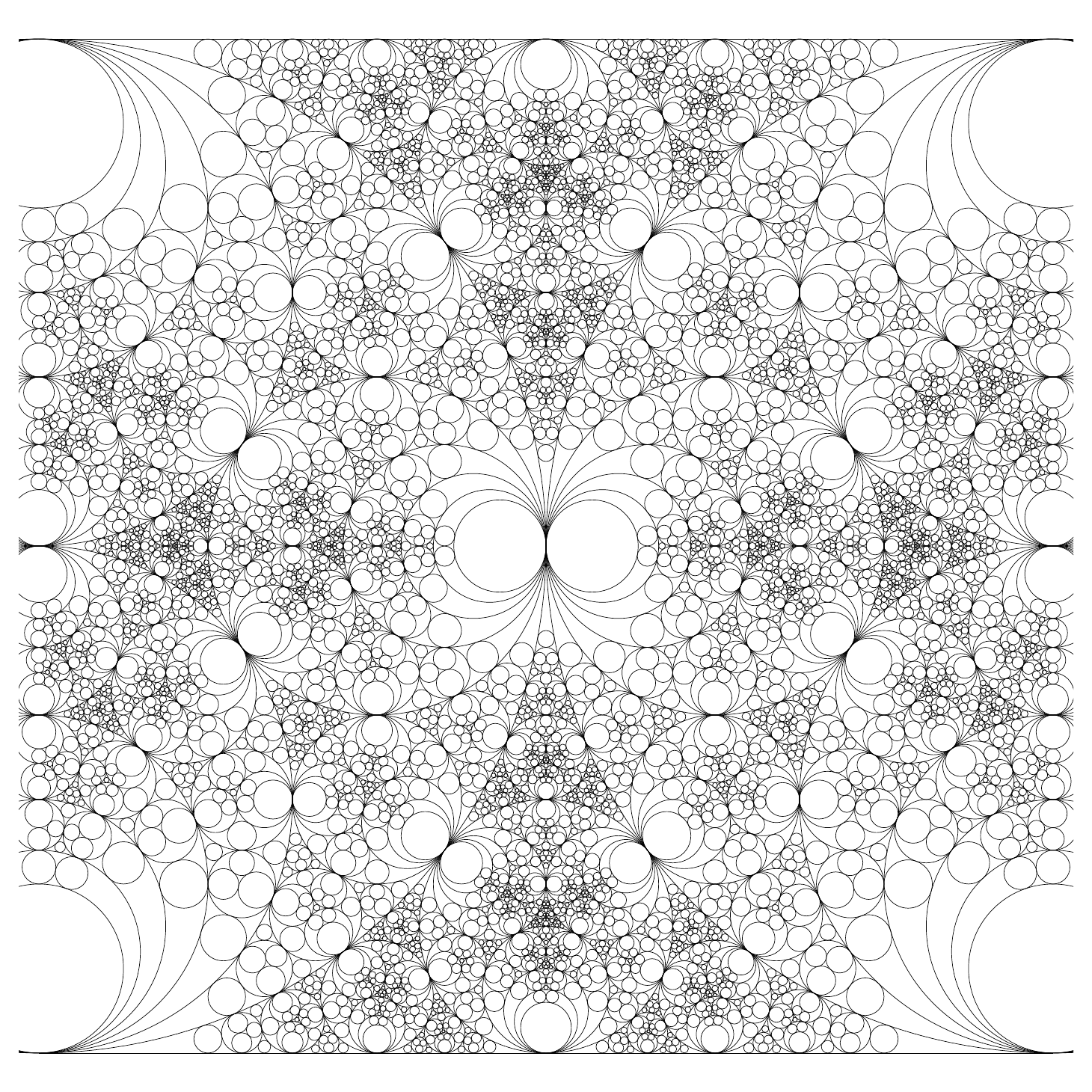}
\caption{Circles in $[0,1]\times[0,1]$ generated by $\mathcal{A}^S$ swap normal form words of length $\leq5$ acting on $R_B$.  The resulting image consists of a subcollection of the circles of the Schmidt arrangement $\SQ$, which, as word length increases, will eventually include any circle of $\SQ$.}
\label{fig:gaussianpacking5}
\end{center}
\end{figure}

If $W=M_nM_{n-1}\cdots M_1$, $M_i\in\{S_j,S_j^{\perp}\}$ is a word in the Super-Apollonian generators and $Q=WR_B$, then in terms of the M{\"o}bius transformations the circles of $Q$ are $\mf{m}_1\mf{m}_2\dots\mf{m}_nc_i$ where $c_i$ are the circles of $R_B$.  The swap normal form for $\mc{A}^S$ passes to words in the M{\"o}bius generators, \textbf{reversing order as just noted}.  Compare the following definition to Definition \ref{def:normalforms}, noting the order reversal.

\begin{defn}
A word $\mf{w}=\mf{m}_1\mf{m}_2\dots\mf{m}_n$ in the Super-Apollonian M{\"o}bius generators is in {\rm swap normal form} if $\mf{m}_i\neq\mf{m}_{i+1}$ and if whenever $\mf{m}_i=\mf{s}_j$ and $\mf{m}_{i+1}=\mf{s}_k^{\perp}$, then $j=k$; i.e. the ``swaps'' are pushed as far right as possible (equivalently the ``inversions'' as far left as possible).

A word   $\mf{w}=\mf{m}_1\mf{m}_2\dots\mf{m}_n$ is in \emph{invert normal form} if  $\mf{m}_i\neq\mf{m}_{i+1}$ and if whenever $\mf{m}_i=\mf{s}_j^\perp$ and $\mf{m}_{i+1}=\mf{s}_k$, then $j=k$; i.e. ``inversions'' are as far right as possible.
\end{defn}

Finally, we note that $\mf{s}_i^{\perp}=\mf{d}\mf{s}_i\mf{d}$, where
\begin{equation}\label{mfd}
\mf{d}=\frac{\bar{z}-1+i}{(1-i)\bar{z}+i}=\mf{d}^{-1}
\end{equation}
is the isometry of the octahedron switching opposite faces, the M{\"o}bius version of the ``duality operator''
$$
D=\frac{1}{2}\left(
\begin{array}{cccc}
-1&1&1&1\\
1&-1&1&1\\
1&1&-1&1\\
1&1&1&-1\\
\end{array}
\right)
$$
from \cite{GLMWY1}.  This defines an involution on Super-Apollonian words in swap normal form, namely taking the transpose (or reversing order and conjugating by $D$)
$$
M=M_nM_{n-1}\dots M_1, \ M^{\perp}=M_1^{\perp}M_2^{\perp}\dots M_n^{\perp}.
$$
On the level of M{\"o}bius transformations, $\mf{m}^{\perp}=\mf{d}\mf{m}^{-1}\mf{d}$.  See \cite{GLMWY1}, \cite{GLMWY2} for more information.


\section{A pair of dynamical systems}\label{sec:dyn}

\subsection{Dynamics on $\PP^1(\mbb{C})$}
\label{sec:dyn-p1}

We now define a pair of dynamical systems on $\PP^1(\mbb{C})$ associated to the base quadruples $R_B$ and $R_A$.  Let $B_i$, $B_i'$, be the open circular and closed triangular regions of the plane coming from the base quadruple $R_B$ ($A_i$ and $A_i'$ defined similarly, see Figure \ref{fig:basequad}), and define
$$
T_B(z)=\left\{
\begin{array}{cc}
\mf{s}_iz&z\in B_i',\\
\mf{s}_i^{\perp}z&z\in B_i,\\
\end{array}
\right., \ 
T_A(w)=\left\{
\begin{array}{cc}
\mf{s}_iw&w\in A_i,\\
\mf{s}_i^{\perp}w&w\in A_i'.\\
\end{array}
\right.
$$
In words, if the point $z$ is in one of the four closed triangular regions, we swap and if $z$ is in one of the open circular regions, we invert in that circle.  Each of $T_A$ and $T_B$ has six fixed points, the points of tangency $\{0,1,\infty,i,i+1,1/(1-i)\}$.  

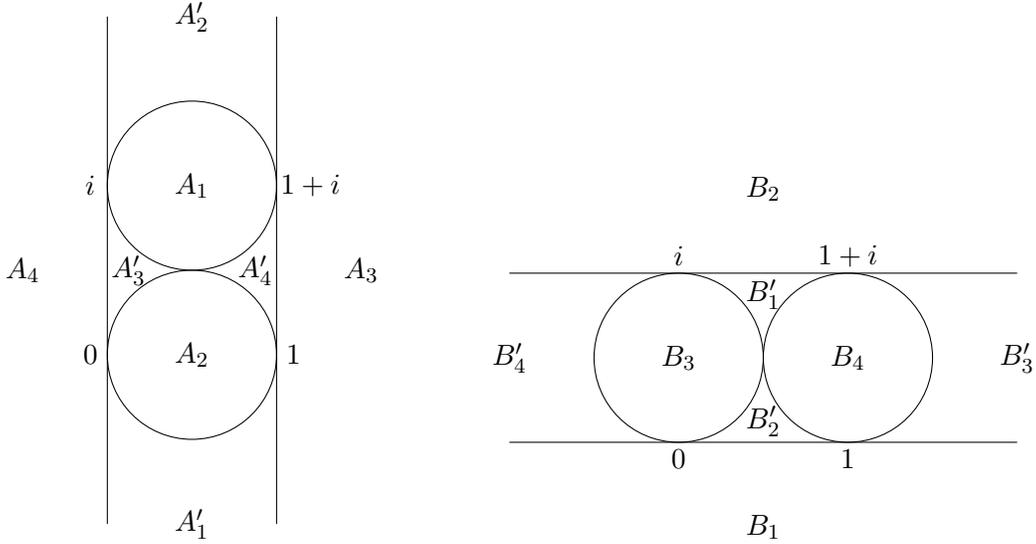
\begin{figure}[htp]
\begin{center}
\begin{tikzpicture}[scale=2.25]
\draw(1/2,0)circle(1/2);
\draw(1/2,0)node{$A_2$};
\draw(1/2,1)circle(1/2);
\draw(1/2,1)node{$A_1$};
\draw(0,-1)--(0,2);
\draw(3/2,1/2)node{$A_3$};
\draw(1,-1)--(1,2);
\draw(-1/2,1/2)node{$A_4$};

\draw(1/2,-1)node{$A_1'$};
\draw(1/2,2)node{$A_2'$};
\draw(1/8,1/2)node{$A_3'$};
\draw(7/8,1/2)node{$A_4'$};

\draw(-.1,0)node{$0$};
\draw(1.1,0)node{$1$};
\draw(-.1,1)node{$i$};
\draw(1.2,1)node{$1+i$};
\end{tikzpicture}
\hspace{1cm}
\begin{tikzpicture}[scale=2.25]
\draw(0,1/2)circle(1/2);
\draw(0,1/2)node{$B_3$};
\draw(1,1/2)circle(1/2);
\draw(1,1/2)node{$B_4$};
\draw(-1,0)--(2,0);
\draw(1/2,-1/2)node{$B_1$};
\draw(-1,1)--(2,1);
\draw(1/2,3/2)node{$B_2$};

\draw(1/2,7/8)node{$B_1'$};
\draw(1/2,1/8)node{$B_2'$};
\draw(2,1/2)node{$B_3'$};
\draw(-1,1/2)node{$B_4'$};

\draw(0,-.1)node{$0$};
\draw(1,-.1)node{$1$};
\draw(0,1.1)node{$i$};
\draw(1,1.1)node{$1+i$};
\end{tikzpicture}
\caption{The regions $A_i$, $A_i'$, $B_i$, $B_i'$.}
\label{fig:basequad}
\end{center}
\end{figure}

\begin{theorem}\label{finiteterminate}
Under the dynamical systems $T_A$ and $T_B$, every Gaussian rational $z\in\mathbb{Q}(i)$ reaches one of the fixed points in finite time.  The fixed point reached is determined by the ``parity'' of the numerator and denominator of $z=p/q$, i.e. one of the six equivalence classes under the equivalence relation
$$
\frac{p}{q}\sim\frac{r}{s}\quad \Longleftrightarrow  \quad ps\equiv qr \pmod 2.
$$
\end{theorem}

\begin{proof}
Recall that $B[-1] = \PGL_2(\ZZ[i]) \rtimes\langle \mathfrak{c}\rangle$.  The group $\Gamma$ is the kernel of the surjective map
$$
\PGL_2(\mbb{Z}[i])\rtimes\langle\mathfrak{c}\rangle\to \PGL_2(\mbb{Z}[i]/(2))
$$
since $\Gamma$ is in the kernel and both are of index $48$ in $B[-1]$ (we have $[B[-1]:\Gamma]=48$ comparing fundamental domains and $|\PGL_2(\mbb{Z}[i]/(2))|=48$ by direct computation).  Hence $\Gamma$ preserves parity, e.g.
$$
\mf{s}_1\left(\frac{p}{q}\right)=\frac{(1+2i)\bar{p}-2\bar{q}}{2\bar{p}+\bar{q}(-1+2i)}\equiv\frac{p}{q}\bmod 2.
$$
Termination in finite time follows from the following version(s) of the Euclidean algorithm in $\mbb{Z}[i]$.  ``Homogenizing'' $T_A$ and $T_B$ gives dynamical systems on pairs $(0,0)\neq(p,q)\in\mbb{Z}[i]\times\mbb{Z}[i]$ that terminate when $p/q\in\{0,1,\infty,i,1+i,\frac{1}{1-i}\}$, i.e. act on pairs via complex conjugation and the matrices implied in the definitions of the $\mf{s}_i$, $\mf{s}_i^{\perp}$.  For instance,
$$
T_B(p,q):=(\bar{p},2\bar{p}-\bar{q}) \ \text{ for } \ \frac{p}{q}\in B_2', \ \text{ where } \ T_B(p/q)=\mf{s}_2(p/q)=\frac{\overline{p/q}}{2(\overline{p/q})-1}.
$$
We'll consider the case of $T_B$, noting that the proof for $T_A$ is nearly identical: 
$$
(p,q)\mapsto\left\{
\begin{array}{ll}
s_1(p,q)=((1+2i)\bar{p}-2\bar{q},2\bar{p}+(2i-1)\bar{q})&p/q\in B_1'\setminus\{i,1+i,\frac{1}{1-i}\}\\
s_2(p,q)=(\bar{p},2\bar{p}-\bar{q})&p/q\in B_2'\setminus\{0,1,\frac{1}{1-i}\}\\
s_3(p,q)=(2\bar{q}-\bar{p},\bar{q})&p/q\in B_3'\setminus\{1,1+i,\infty\}\\
s_4(p,q)=(-\bar{p},\bar{q})&p/q\in B_4'\setminus\{0,i,\infty\}\\
s_1^{\perp}(p,q)=(\bar{p},\bar{q})&p/q\in B_1\setminus\{0,1,\infty\}\\
s_2^{\perp}(p,q)=(\bar{p}+2i\bar{q},\bar{q})&p/q\in B_2\setminus\{i,i+1,\infty\}\\
s_3^{\perp}(p,q)=(\bar{p},\bar{q}-2i\bar{p})&p/q\in B_3\setminus\{0,i,\frac{1}{1-i}\}\\
s_4^{\perp}(p,q)=((1-2i)\bar{p}+2i\bar{q},(2i+1)\bar{q}-2i\bar{p})&p/q\in B_4\setminus\{1,1+i,\frac{1}{1-i}\}.\\
\end{array}
\right.
$$
The inequalities defining the regions $B_i$, $B_i'$ show that $|q|$ is reduced whenever $s_1$, $s_2$, $s_3^{\perp}$, or $s_4^{\perp}$ is applied.  For example, when applying $s_1$, the fact that $p/q$ is in the triangle $B_1'\setminus\{i,1+i,\frac{1}{1-i}\}$ (or the circle $A_1$) shows that $|q|$ is reduced when applying $s_1$, as follows:
$$
\frac{p}{q}\in B_1'\setminus\left\{i,1+i,\frac{1}{1-i}\right\}\Longrightarrow\left|\frac{p}{q}-\frac{1+2i}{2}\right|^2<\frac{1}{4}\quad \Longrightarrow\quad |2\bar{p}+(1-2i)\bar{q}|<|q|.
$$
Note that the inequalities above define $A_1$, but since $B_1'$ is contained in $A_1$ they hold true for $B_1'$ as well.  Similarly, application of $s_3$ and $s_2^{\perp}$ both reduce $|p|$.  Applying $s_4$ maps $B_4'$ onto $B_1'\cup B_2'\cup B_4\cup B_3'$ (from which one of $|p|$, $|q|$ will be reduced as just discussed).  Finally, $s_1^{\perp}$ maps $B_1$ onto one of the other seven regions.  Hence the algorithm terminates.
\end{proof}

Iteration of the map $T_B$ or $T_A$ with input $z$ produces a word $\mf{z}=\mf{m}_1\cdots\mf{m}_n \cdots$ in the M{\"o}bius generators $\mf{s}_1, \mf{s}_2, \mf{s}_3, \mf{s}_4, \mf{s}_1^\perp, \mf{s}_2^\perp, \mf{s}_3^{\perp}, \mf{s}_4^\perp$, where $\mf{m}_n$ is defined by $T_B^nz=\mf{m}_n(T_B^{n-1}z)$.  We take the word to be finite for $z\in\mbb{Q}(i)$, ending when a fixed point is reached.  An example of this process is shown in Figure \ref{fig:approx}.

\begin{figure}[h]
\begin{center}
\includegraphics[scale=.5]{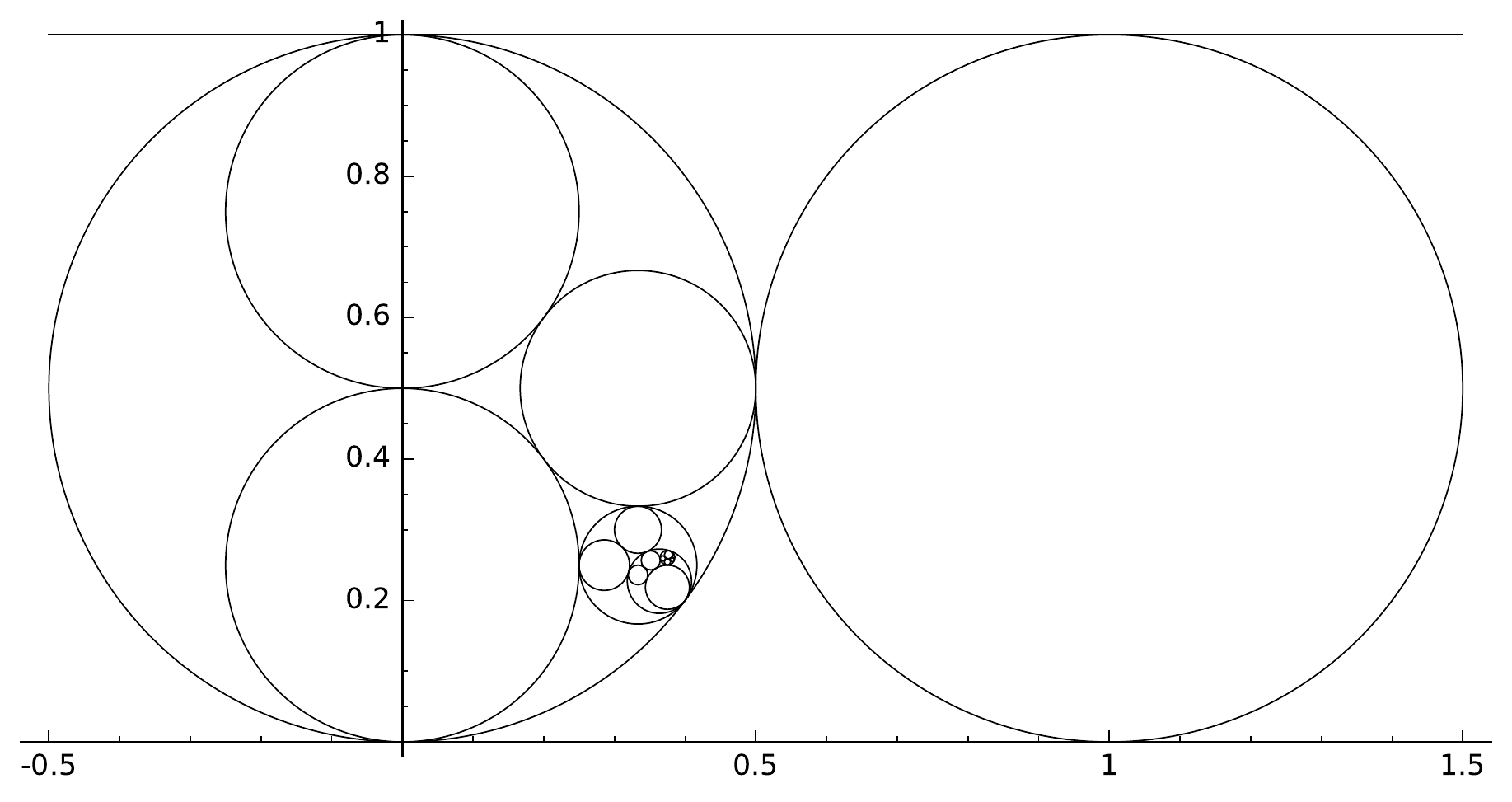}
\caption{Approximating the irrational point $0.3828008104\ldots+i0.2638108161\ldots$ using $T_B$ gives the normal form word $\mf{s}_3^{\perp} \mf{s}_2 \mf{s}_2^{\perp} \mf{s}_3^{\perp} \mf{s}_1 \mf{s}_1^{\perp} \mf{s}_4 \mf{s}_2 \mf{s}_4 \mf{s}_1 \mf{s}_1^{\perp} \mf{s}_3 \mf{s}_2 \mf{s}_3 \mf{s}_3^{\perp} \mf{s}_4 \mf{s}_4^{\perp} \mf{s}_2 \mf{s}_1 \mf{s}_2 \cdots$.}
\label{fig:approx}
\end{center}
\end{figure}

The collection $\mathcal{A}^S(n)$ of length $n$ Super-Apollonian words in swap normal form partitions the plane into a collection of $9\cdot5^{n-1}-1$ triangles and circles, which we call Farey circles and triangles following Schmidt \cite{S1}.  Specifically, we associate to each word in $\mathcal{A}^S(n)$ an open circular or closed triangular region, the notation being
$$
F_B(\mf{m})=\mf{m}_1 \cdots \mf{m}_{n-1} B_i \ (\text{circular}), \ \mf{m}_1 \cdots \mf{m}_{n-1} B_i'  \ (\text{triangular}), 
$$
for $\mf{m}=\mf{m}_1\cdots\mf{m}_n$ with $\mf{m}_n=\mf{s}_i^{\perp}$ (circular) or $\mf{s}_i$ (triangular).  This definition is set up so that the words of length one correspond to the eight regions of the base quadruple.

\begin{theorem}\label{thm:z-mfz}
  A word $\mf{z} = \mf{m}_1 \mf{m}_2 \cdots$ produced by iteration by $T_B$ (respectively $T_A$) on $z \in \PP^1(\CC)$ is in swap (respectively invert) normal form.  
  Furthermore,
  \begin{enumerate}
    \item If $z$ is rational, then 
  \[
    z = \mf{z}b
  \]
  for $b \in \{0,1,\infty,i,1+i,\frac{1}{1-i} \}$ matching $z$ in parity as described in Theorem~\ref{finiteterminate}.
\item If $z$ is not rational, then $\mf{z}$ is an infinite word with the property that
  \[
  \{z\} = \bigcap_{n\ge 1} F(\mf{m}_1 \cdots \mf{m}_n ).
   \]
  \end{enumerate}
\end{theorem}

This gives a bijection $z \leftrightarrow \mf{z}$, under which $T_B$ (respectively $T_A$) can be considered to act on words, and this action is via the left-shift, $T_B\left(\mf{m}_1\mf{m}_2 \cdots \right)=\mf{m}_2 \mf{m}_3 \cdots$.

\begin{proof}
That $\mf{z}$ is in swap normal form is clear; the only circular region in $\mf{s}_i(B_i')$ is $B_i$.

Two Farey sets are either disjoint or one is contained in the other:  if $\mf{m}$ is a (left) initial segment of $\mf{n}$, which we denote $\mf{m}\leq\mf{n}$, then $F_B(\mf{m})\supseteq F_B(\mf{n})$.  For $z\in\mathbb{C}\setminus\mbb{Q}(i)$, the infinite swap normal form word $\mf{z}=\mf{m}_1 \mf{m}_2 \cdots$ produced by iterating $T_B$ determines $z$ since $z=\cap_nF(\mf{m}_1\dots\mf{m}_n)$.  For rational points, to determine $z$, we need both the finite word and the parity of the rational:  then $z = \mf{z}b$, where $b$ is the element of $\{0, 1, \infty, i, 1+i, \frac{1}{1-i} \}$ of the specified parity.  
\end{proof}

From now on, we consistently use the variables $z$, $\mf{z}$ for the $B$ coordinate system, and $w$, $\mf{w}$ for the $A$ coordinate system, since we will be using both codings simultaneously.

\subsection{Covering the boundary of the ideal octahedron and first approximation constant for $\mbb{Q}(i)$}
The purpose of this section and the next is to relate the Super-Apollonian continued fraction algorithm to classical statements of Diophantine approximation.  In this section, we give the first value of the Lagrange spectrum for complex approximation by Gaussian rationls.  With this as a point for comparison, in the second section we describe the goodness of the approximations obtained by the algorithm.

The ``good'' rational approximations to an irrational $z\in\mbb{C}$
$$
|z-p/q|\leq C/|q|^2
$$
are determined by the collection of horoballs
\begin{align*}
B_C(p/q)&=\{(z,t)\in H^3 : |z-p/q|^2+(t-C/|q|^2)^2\leq C^2/|q|^4\},\\
B_C(\infty)&=\{(z,t)\in H^3 : t\geq1/2C\},
\end{align*}
through which the geodesic $\overrightarrow{\infty z}$ passes (or through which any geodesic $\overrightarrow{wz}$ eventually passes).  Over $\mbb{Q}(i)$, the smallest value of $C$ with the property that every irrational $z$ has infinitely many rational approximations satisfying the above inequality was determined by Ford in \cite{F}.  Here we give a short proof of this fact using the geometry of the ideal octahedron.
\begin{prop}
Every $z\in\mbb{C}\setminus\mbb{Q}(i)$ has infinitely many rational approximations $p/q\in\mbb{Q}(i)$ such that
$$
\left|z-\frac{p}{q}\right|\leq\frac{C}{|q|^2}, \quad C = \frac{1}{\sqrt{3}} \simeq 0.577350269\ldots
$$
The constant $1/\sqrt{3}$ is the smallest possible, as witnessed by $z=\frac{1+\sqrt{-3}}{2}$.
\end{prop}
\begin{proof}
The value of $C$ for which the horoballs with parameter $C$ based at the ideal vertices $\{0,1,\infty,i,1+i,1/(1-i)\}$ cover the boundary of the fundamental octahedron is easily found to be $1/\sqrt{3}$ (one need only cover the face with vertices $\{0,1,\infty\}$, see Figure~\ref{fig:face}).  Hence as we follow the geodesic $\overrightarrow{\infty z}$ through the tesselation by octahedra, at least one of the six vertices of each octahedron satisfies the above inequality with $C=1/\sqrt{3}$, one or two as it enters and one or two as it exits (some of which may coincide).  This gives the smallest value of $C$ for which the inequality above has infinitely many solutions for all irrational $z$, noting the the geodesic from $e^{-\pi i/3}$ to $e^{\pi i/3}$ passes orthogonally through the ``centers'' of the opposite faces of the octahedra through which it passes.
\end{proof}

The sequence of octahedra we consider in our continued fraction algorithm are not necessarily along the geodesic path, but we do capture all rationals with $|z-p/q|<C/|q|^2$ with $C=1/(1+1/\sqrt{2})$ as detailed in the next section.
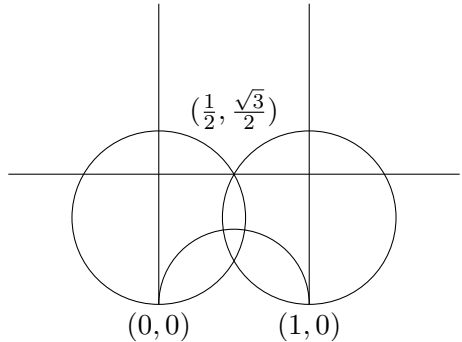
\begin{figure}[htp]
\begin{center}
\begin{tikzpicture}[scale=2]
\draw(0,0.577)circle(0.577);
\draw(1,0.577)circle(0.577);
\draw(-1,0.866)--(2,0.866);
\draw(0,0)--(0,2);
\draw(1,0)--(1,2);
\draw(1,0)arc(0:180:0.5);
\draw(0,-0.15)node{$(0,0)$};
\draw(1,-0.15)node{$(1,0)$};
\draw(0.5,1.3)node{$(\frac{1}{2},\frac{\sqrt{3}}{2})$};
\end{tikzpicture}
\caption{Horoball covering of the ideal triangular face with vertices $\{0,1,\infty\}$ with coordinates $(z,t)$.}
\label{fig:face}
\end{center}
\end{figure}

\subsection{Quality of rational approximation}
To any complex number we associate six sequences of Gaussian rational approximations by following the inverse orbit of the six points of tangency of our base quadruples $R_A$, $R_B$.  Namely, if $z=\prod_{i=1}^{\infty}\mf{z}_i=\prod_{i=1}^{\infty}\mf{w}_i=w$ in the two codings, then the convergents $p^A_{n,\alpha}/q^A_{n,\alpha}$, $p^B_{n,\alpha}/q^B_{n,\alpha}$ are given by
$$
\frac{p^A_{n,\alpha}}{q^A_{n,\alpha}}=\left(\prod_{i=1}^{n}\mf{w}_i\right)(\alpha), \ \frac{p^B_{n,\alpha}}{q^B_{n,\alpha}}=\left(\prod_{i=1}^{n}\mf{z}_i\right)(\alpha), \ \alpha\in\{0,1,\infty,i,i+1,1/(1-i)\},
$$
with the property that
$$
\lim_{n\to\infty}\frac{p^A_{n,\alpha}}{q^A_{n,\alpha}}=w, \ \lim_{n\to\infty}\frac{p^B_{n,\alpha}}{q^B_{n,\alpha}}=z
$$
for all $\alpha$ and $w\in\mathbb{C}\setminus\mbb{Q}(i)$, $z\in\mathbb{C}\setminus\mbb{Q}(i)$.

The following theorem is equivalent to a statement about approximation by Schmidt's continued fractions, given as Theorem 2.5 in \cite{S1}.  In particular, the approximations given by Schmidt's algorithm and the Super-Apollonian algorithm coincide.  In \cite{S1}, it is stated without proof; here we provide a proof.

\begin{theorem}
If $p/q$ is such that 
      \[
	\left|z_0-p/q\right|<\frac{C}{|q|^2}, \quad C=\frac{\sqrt{2}}{1+\sqrt{2}} \simeq 0.585786437\ldots,
      \]
then $p/q$ is a convergent to $z_0$ (with respect to both $T_A$ and $T_B$).  Moreover, the constant $C$ is the largest  possible.
\end{theorem}

\begin{proof}
Note that the Apollonian super-packings associated to the root quadruples $R_A$, $R_B$, are invariant under the action of $\PSL_2(\mbb{Z}[i])\rtimes\langle\mf{c}\rangle$.  Consider the quadruple where $p/q$ first appears as a convergent to $z_0$, and let $\gamma(z)=\frac{-Qz+P}{-qz+p}\in \PSL_2(\mbb{Z}[i])$ take this quadruple to the base quadruple (say $R_B$) with $p/q$ mapping to infinity and infinity mapping to $Q/q$.  For any value of $C$, the disk of radius $C/|q|^2$ centered at $p/q$ gets mapped by $\gamma$ to the exterior of the disk of radius $1/C$ centered at $Q/q$:
\begin{align*}
w&=\frac{-Qz+P}{-qz+p} \Rightarrow |z-p/q|=\frac{1}{|w-Q/q||q|^2},\\
\frac{C}{|q|^2}&\geq|z-p/q|=\frac{1}{|w-Q/q||q|^2} \Rightarrow |w-Q/q|\geq1/C.
\end{align*}
Consider the ways in which $p/q$ can first appear as a convergent to $z_0$ in the sequence of partitions of the plane.
\begin{itemize}
  \item We might invert into a circle containing $z_0$.  In particular, then, $p/q$ is in the interior of the circle of inversion (since it is its first appearance as a convergent).  In this case, by the discussion in Section \ref{sec:dyn-p1}, all $z$ inside the circle also include this inversion in their expansion.  Therefore, all $z$ inside the circle will have $p/q$ as a convergent.  Our goal is to show that the circle of radius $1/|q|^2$ around $p/q$ is contained in the circle of inversion.

Under $\gamma$ above (perhaps after applying some binary tetrahedral symmetry of the base quadruple), the circles $A$, $B$, get mapped to $A'$, $B'$ as in the figures below, with $Q/q=\gamma(\infty)$ lying in the triangle inside $B'$ as shown.  The exterior of a disk of radius one centered at $Q/q$ does not meet the interior of $B'$, hence, applying $\gamma^{-1}$, the disk of radius $1/|q|^2$ centered at $p/q$ does not meet $B$.  Therefore $p/q$ is a convergent to any $z$ with $|z-p/q|<1/|q|^2$.
\begin{center}
\includegraphics[scale=.7]{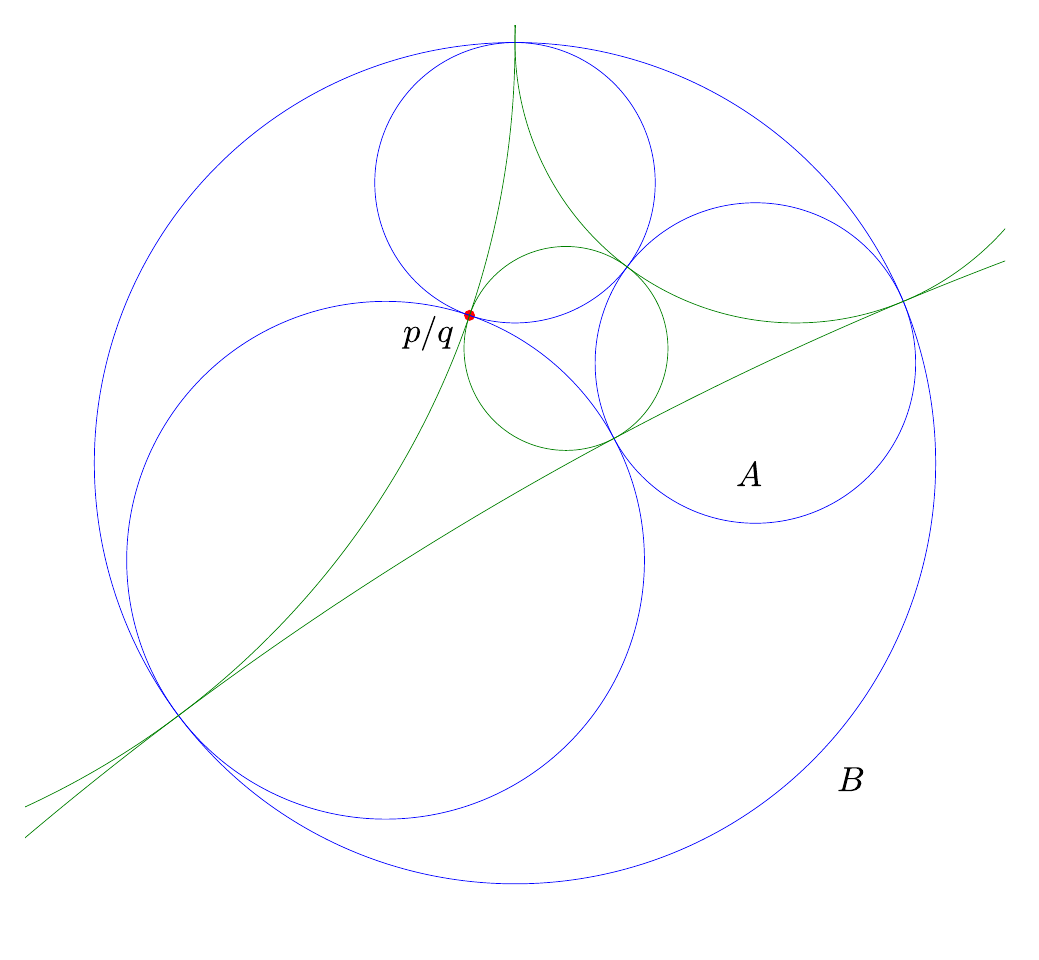}
\includegraphics[scale=.7]{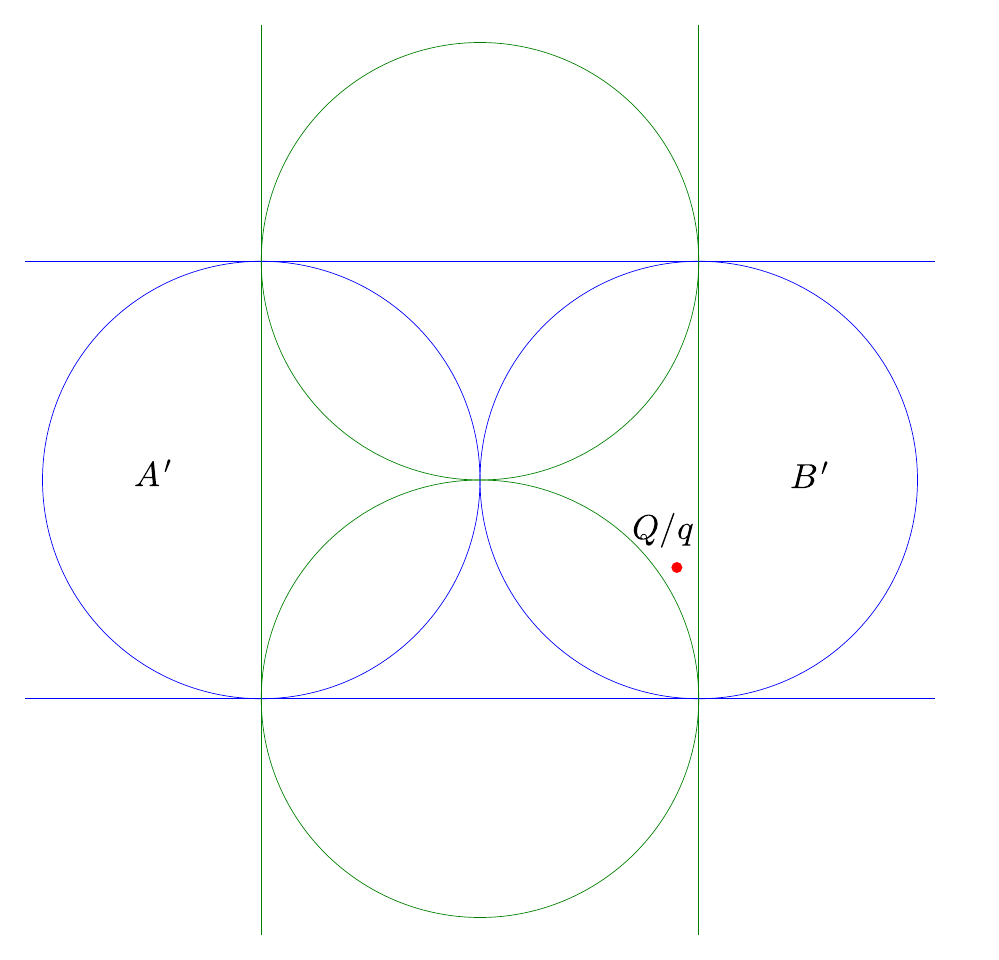}
\end{center}
\item We might swap into a triangle containing $z_0$ producing $p/q$ as a point of tangency on an edge of this triangle.  In the image, $z_0$ is in the quadrangle with sides formed by $A,B,C,D$; this is the union of two triangles.  The dotted circles in the figures indicate the two possible swaps associated to the initial creation of $p/q$ as a convergent.  In this case, any $z$ in the indicated quadrangle will have $p/q$ as a convergent.  We aim to show that a circle of radius $C/|q|^2$ around $p/q$ is contained in this region.

Under $\gamma$ above (perhaps after applying some binary tetrahedral symmetry of the base quadruple), the circles $A$, $B$, $C$, $D$, and $E$ are mapped to $A'$, $B'$ $C'$, $D'$ and $E'$, the three circles tangent to $p/q$ are mapped to the lines in the second picture, with $Q/q=\gamma(\infty)$ lying in the intersection of the disks defined by $B'$ and $E'$.  The exterior of any circle of radius $1/C=1+1/\sqrt{2}$ centered inside $E'$ avoids the interiors of $A'$, $B'$, $C'$, and $D'$.  Applying $\gamma^{-1}$ shows that the disk of radius $C/|q|^2$ around $p/q$ does not meet $A$, $B$, $C$, or $D$, so that $p/q$ is a convergent to any $z$ with $|z-p/q|<C/|q|^2$.
\begin{center}
\includegraphics[scale=.7]{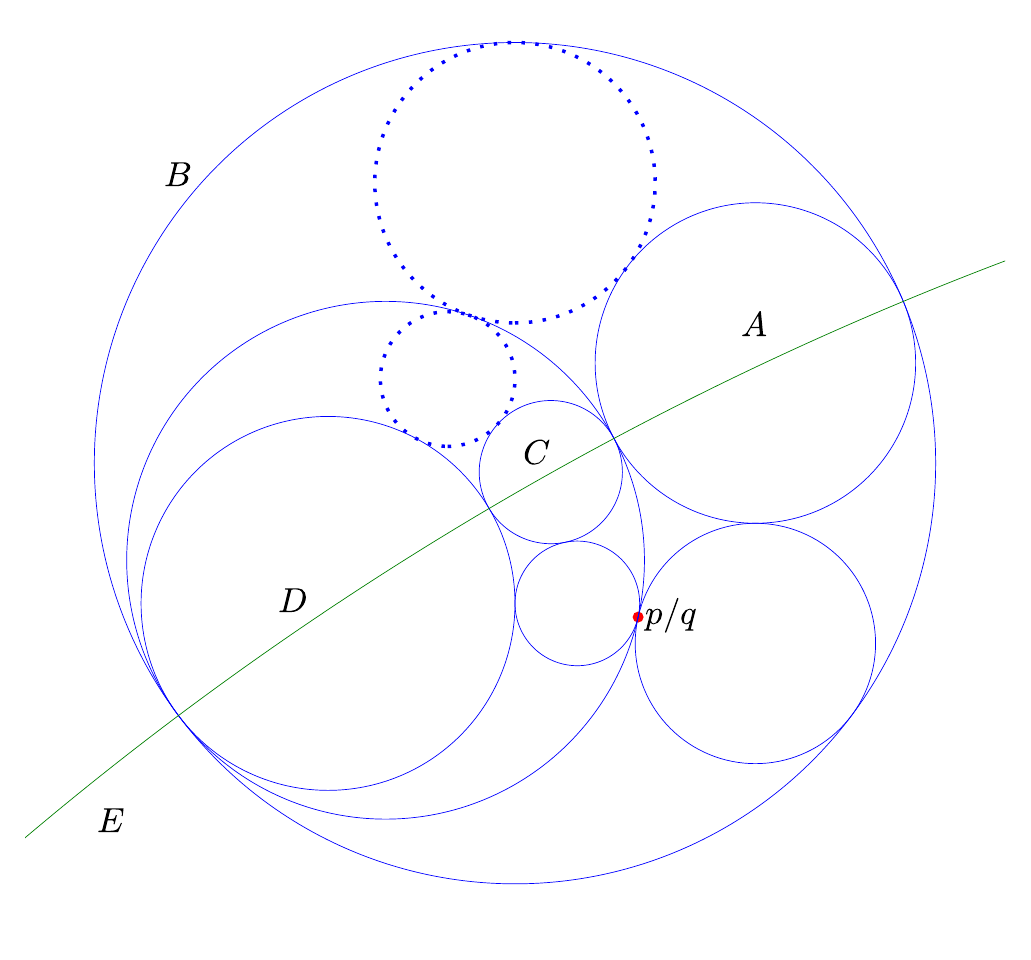}
\includegraphics[scale=.7]{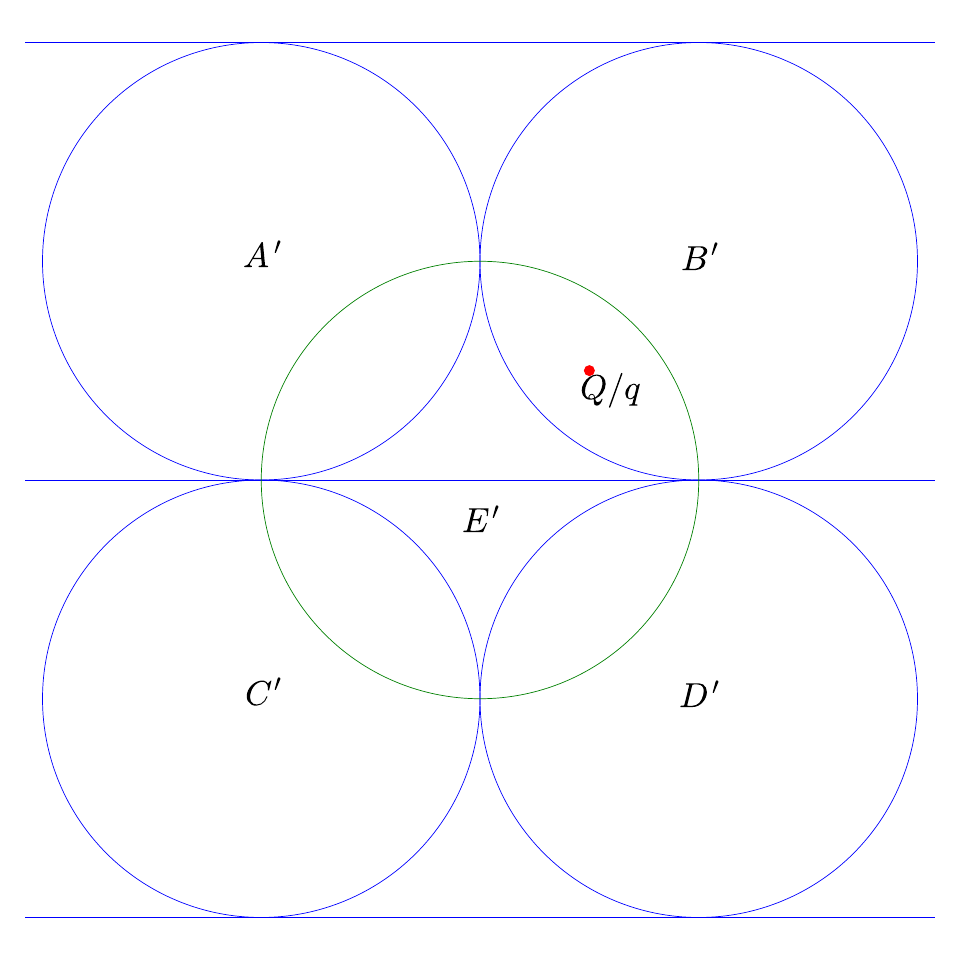}
\end{center}

\end{itemize}

\end{proof}


\subsection{Invertible extension and invariant measures}
\label{sec:inv}

In this section, we derive an invertible extension $T$ of $T_B$, with the property that $T^{-1}$ extends $T_A$, along with an invariant measure for $T$.  This is done with a goal of eventually producing an ergodic measure preserving system, as is done both in \cite{R} and \cite{S2}.  For this purpose, let $H^3$ denote hyperbolic 3-space, having boundary $\widehat{\mbb{C}} \sim \mathbb{P}^1(\mbb{C})$ (e.g. the Poincar\'e upper half space model).

The space of oriented geodesics in $H^3$, identified with pairs in $\PP^1(\mbb{C})\times \PP^1(\mbb{C})\setminus\Delta$, where $\Delta$ denotes the diagonal, carries an isometry invariant measure $$|z-w|^{-4}du\;dv\;dx\;dy, \; w=u+iv, \; z=x+iy.$$  We restrict this measure to geodesics in the set
$$
\mathcal{G}=\left(\bigcup_iA_i\times B_i\right)\bigcup\left(\bigcup_{i\neq j}A_i'\times B_j\right)\bigcup\left(\bigcup_{i\neq j}A_i\times B_j'\right)\bigcup\left(\bigcup_{i,j}A_i'\times B_j'\right),
$$
consisting of geodesics between disjoint $A$ and $B$ regions of the base quadruples (see Figure \ref{fig:basequad}).  In what follows we use $A$ coordinates for $w$ in the first coordinate and $B$ coordinates for $z$ in the second coordinate.  Define $T:\mathcal{G}\to\mathcal{G}$ by
$$
T(w,z)=\left\{
\begin{array}{cc}
(\mf{s}_iw,\mf{s}_iz)&z\in B_i, \ \mf{z}=\mf{s}_i\dots\\
(\mf{s}_i^{\perp}w,\mf{s}_i^{\perp}z)&z\in B_i', \ \mf{z}=\mf{s}_i^{\perp}\dots\\
\end{array}
\right.
$$
where $\mf z$ is the M{\"o}bius transformation corresponding to $z$ as described in Theorem~\ref{thm:z-mfz}. Equivalently,
$$
T^{-1}(w,z)=\left\{
\begin{array}{cc}
(\mf{s}_iw,\mf{s}_iz)&w\in A_i, \ \mf{w}=\mf{s}_i\dots\\
(\mf{s}_i^{\perp}w,\mf{s}_i^{\perp}z)&w\in A_i', \ \mf{w}=\mf{s}_i^{\perp}\dots\\
\end{array}
\right..
$$
In other words, $T$ is applying $T_B$ diagonally depending on the second coordinate.  In terms of the shifts on pairs $(\mf{w},\mf{z})=(\prod_i\mf{w}_i,\prod_i\mf{z}_i)$ corresponding to $(w,z)$, we have
$$
T(w,z)=(\mf{z}_1\mf{w},\prod_{i=2}^{\infty}\mf{z}_i)=(\mf{z}_1w,T_B(z)), \ T^{-1}(w,z)=(\prod_{i=2}^{\infty}\mf{w}_i,\mf{w}_1\mf{z})=(T_A(w),\mf{w}_1z).
$$
See Figure \ref{fig:invext} for a visualization of the invertible extension.

\begin{figure}
  \includegraphics[width=3in]{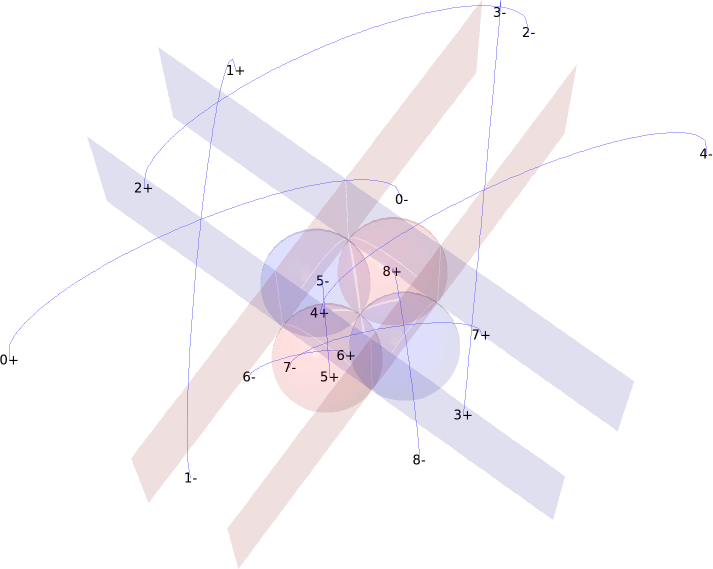}
  \quad
  \includegraphics[width=3in]{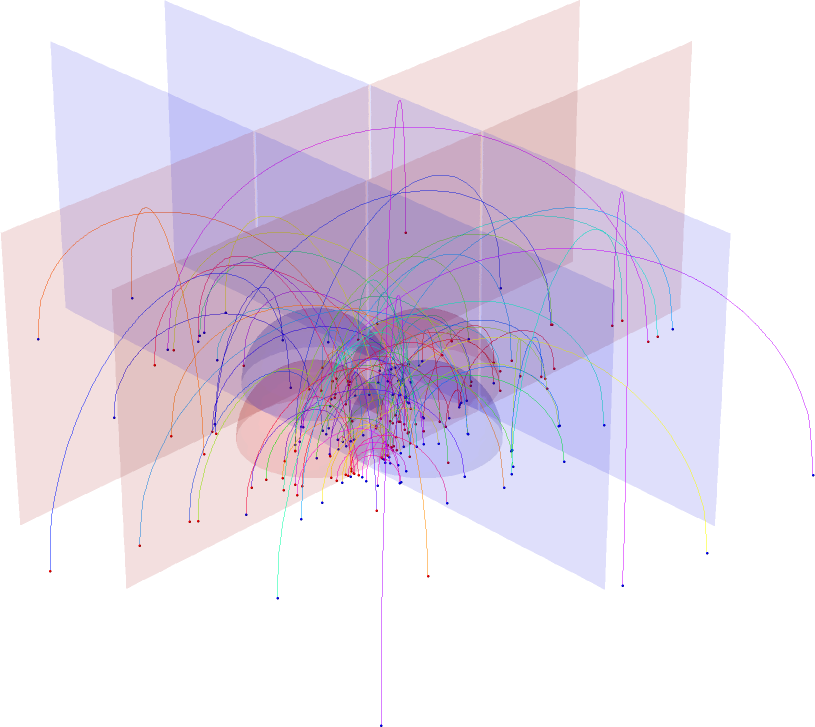}
  \caption{At left, eight iterations of the map $T$, labelled in sequence (with $+$ and $-$ indicating the orientation).  At right, 100 iterations of $T$, with rainbow coloration indicating time.  In both pictures, the geodesic planes of the base quadruple and its dual are shown.}
  \label{fig:invext}
\end{figure}

Define the following regions (see Figure \ref{fig:manyregions}):
\begin{align*}
\mathcal{A}_i&=B_i\cup(\cup_{j\neq i}B_j')\\
\mathcal{A}_i'&=(\cup_j B_j')\cup(\cup_{j\neq i}B_j)\\
\mathcal{B}_i&=A_i\cup(\cup_{j\neq i}A_j')\\
\mathcal{B}_i'&=(\cup_j A_j')\cup(\cup_{j\neq i}A_j)
\end{align*}

Here $\mathcal{A}_i$ consists of the $B$ regions not intersecting $A_i$, and so on.

\begin{theorem}\label{measurethm}
  The function $T: \mathcal{G} \rightarrow \mathcal{G}$ is a measure-preserving bijection.
Consequently, the push-forward of this measure onto the first or second coordinate gives invariant measures $\mu_A$ and $\mu_B$ for $T_A$ and $T_B$.  Specifically, we have
$$
d\mu_A(w)=f_A(w)\;du\;dv=\left\{
\begin{array}{cc}
f_{A_i}(u,v)\;du\;dv=du\;dv\int_{\mathcal{A}_i}|z-w|^{-4}\;dx\;dy,& \ w\in A_i\\
f_{A_i'}(u,v)\;du\;dv=du\;dv\int_{\mathcal{A}_i'}|z-w|^{-4}\;dx\;dy,& \ w\in A_i'\\
\end{array}
\right.
$$
and
$$
d\mu_B(z)=f_B(z)\;dx\;dy=\left\{
\begin{array}{cl}
f_{B_i}(x,y)\;dx\;dy=dx\;dy\int_{\mathcal{B}_i}|z-w|^{-4}\;du\;dv,& \ z\in B_i\\
f_{B_i'}(x,y)\;dx\;dy=dx\;dy\int_{\mathcal{B}_i'}|z-w|^{-4}\;du\;dv,& \ z\in B_i'\\
\end{array}
\right..
$$
\end{theorem}

See Figure \ref{fig:planar_density} for a graph of the invariant density $f_B$.

\begin{figure}
\begin{center}
\includegraphics[scale=.5]{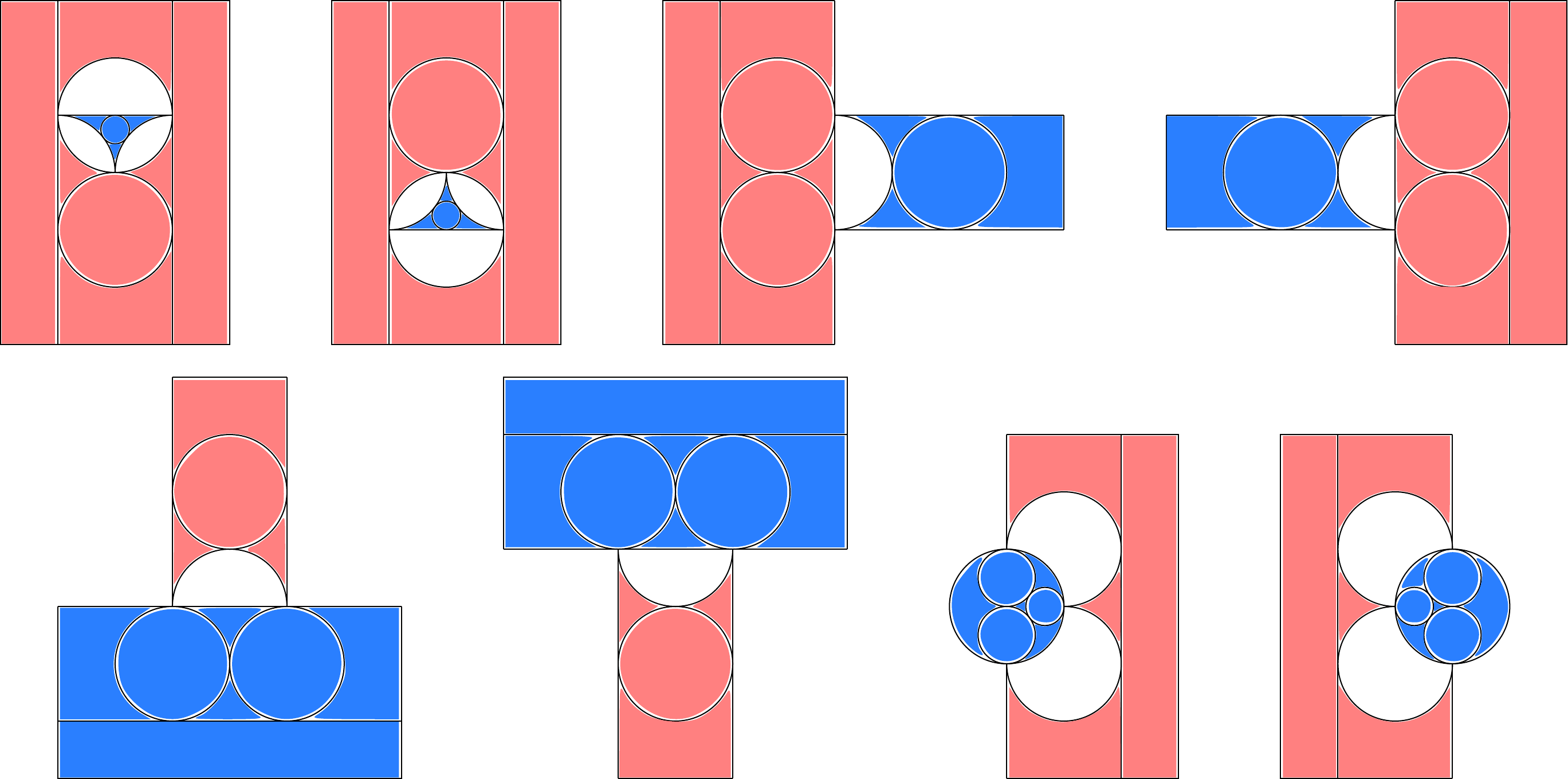}
\caption{Top, left to right: regions $\mathcal{B}_i'\times B_i'$, $i=1,2,3,4$.  Bottom, left to right: regions $\mathcal{B}_i\times B_i$, $i=1,2,3,4$.  Shown red$\times$blue with subdivisions.}
\label{fig:manyregions}
\end{center}
\end{figure}

  \begin{proof} Note that $\mathcal{G}=\cup_i(A_i\times\mathcal{A}_i\cup A_i'\times\mathcal{A}_i')=\cup_i(\mathcal{B}_i\times B_i\cup\mathcal{B}_i'\times B_i')$, which is readily seen in Figure \ref{fig:manyregions}.  Since
\begin{align*}
T(\mathcal{B}_i\times B_i)&=A_i'\times\mathcal{A}_i',\\
T(\mathcal{B}_i'\times B_i')&=A_i\times\mathcal{A}_i,
\end{align*}
we immediately get that $T$ is a bijection.

As $T$ is a bijection defined piecewise by isometries, it preserves the measure described above.  

\end{proof}

Theorem~\ref{measurethm} defines the functions $f_A$ and $f_B$ implicitly.  We now compute what they are explicitly, starting with $f_B$.  Computing the relevant integrals in Theorem~\ref{measurethm} gives $\pi/4$ times hyperbolic area on the triangular regions $B_i'$:
$$
f_B(x,y)=\left\{
\begin{array}{cl}
\frac{\pi}{4\left(\frac{1}{4}-d^2\right)^2}&z\in B_1', \ d^2=\left(x-\frac{1}{2}\right)^2+(y-1)^2\\
\frac{\pi}{4\left(\frac{1}{4}-d^2\right)^2}&z\in B_2', \ d^2=\left(x-\frac{1}{2}\right)^2+y^2\\
\frac{\pi}{4(1-x)^2}&z\in B_3'\\
\frac{\pi}{4x^2}&z\in B_4'\\
\end{array}
\right.,
$$
and on the circular regions $B_i$ we have
$$
f_B(x,y)=\left\{
\begin{array}{cc}
H(x,y)&z\in B_1\\
H(x,1-y)&z\in B_2\\
G(x,y)&z\in B_3\\
G(1-x,y)&z\in B_4\\
\end{array}
\right.,
$$
where
\begin{align*}
H(x,y)&=h(x,y)+h(1-x,y)+h(x^2-x+y^2,y),\\
G(x,y)&=h(x,y^2-y+x^2)+h(x^2-x+y^2,y^2-y+x^2)+h(x^2-x+(1-y)^2,y^2-y+x^2),\\
h(x,y)&=\frac{\arctan(x/y)}{4x^2}-\frac{1}{4xy}.
\end{align*}
Furthermore, we have the relationship 
$$f_A(w)=f_B(\rho w)=f_B(\mf{d}w)$$
where $\rho$ is rotation by $\pi/2$ around $1/(1-i)$ and $\mf{d}$ is the isometry switching opposite faces of the octahedron (the duality operator \ref{mfd}).  The measures $\mu_A$, $\mu_B$ have $S_3$ symmetry on each of the $A_i$, $A_i'$, $B_i$, $B_i'$.  For instance the isometries permuting $\{0,1,\infty\}$ on $A_1'$, $B_1$ preserve the measure (generators shown for a transposition and three-cycle)
$$
(0,1)\sim-\bar{z}+1, \ (0,1,\infty)\sim\frac{-1}{z-1}.
$$
The measures also have $S_4$ symmetry on $\mbb{C}$, permuting the pairs $\{A_i,A_i'\}$, $\{B_i,B_i'\}$ (transpositions $(i,i+1)$ shown)
$$
(1,2)\sim\bar{z}+i, \ (2,3)\sim\frac{1}{\bar{z}}, \ (3,4)\sim-\bar{z}+1.
$$
The total measure assigned to each of the $A_i$, $A_i'$, $B_i$, $B_i'$, is $\pi^2/4$, so to normalize $\mu_A$, $\mu_B$, we divide by $2\pi^2$.

All of the above should be compared with Nakada's extension of Schmidt's system, described in an appendix.

\begin{figure}[h]
\begin{center}
\includegraphics[width=5in]{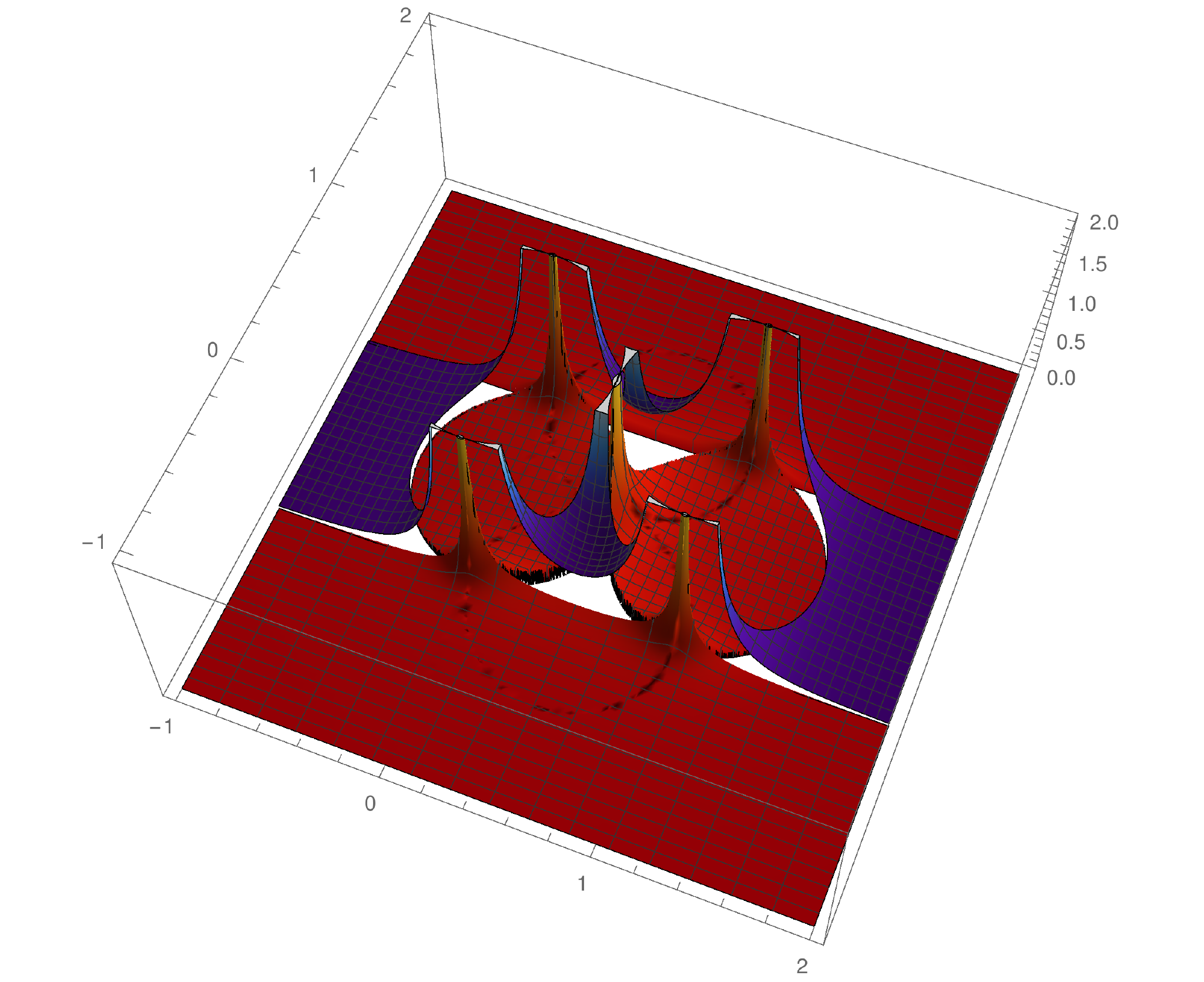}
\caption{Graph of $f_B(z)$, the invariant density for $T_B$.}
\label{fig:planar_density}
\end{center}
\end{figure}

The following lemmas compare the measures of some Farey circles and triangles via the involutions $\perp$ and $^{-1}$, simplifying computations in Section 5.  For use in the proofs, we note the equality of regions
\begin{align}\label{regionequality}
\mf{s}_i^{\perp}\mathcal{B}_i&=A_i'=\mf{d}B_i',\nonumber\\
\mf{s}_i\mathcal{B}_i'&=A_i=\mf{d}B_i, \\ 
\mf{s}_i\mathcal{A}_i&=B_i'=\mf{d}A_i',\nonumber \\
\mf{s}_i^{\perp}\mathcal{A}_i'&=B_i=\mf{d}A_i,\nonumber
\end{align}
where $\mf{d}$ is as in (\ref{mfd}).

\begin{lemma}
  \label{lem:measurecompare1}
  For $\mf{m}$ in swap normal form, and $\mf{n}$ in invert normal form, there are equalities of measure 
$$
\mu_B(F_B(\mf{m}))=\mu_B(F_B(\mf{m}^{\perp})), \quad \mu_A(F_A(\mf{n}))=\mu_A(F_A(\mf{n}^{\perp})).
$$
\end{lemma}
\begin{proof}
Consider the case where $\mf{m}=\mf{s}_i\dots\mf{s}_j$ so that $F_B(\mf{m})=\mf{m}\mf{s}_jB_j'\subseteq B_i'$ and $F_B(\mf{m}^{\perp})=\mf{m}^{\perp}\mf{s}_i^{\perp}B_i\subseteq B_j$ (all other cases are analogous) and let $\omega=|z-w|^{-4}\;du\;dv\;dx\;dy$ be the invariant form on the space of geodesics.  Recall $\mf{m}^{\perp}=\mf{d}\mf{m}^{-1}\mf{d}$ from (\ref{mfd}).  We have
$$
\mu_B(F_B(\mf{m}))=\int_{\mathcal{B}_i'}\int_{F_B(\mf{m})}\omega=\int_{\mathcal{B}_i'}\int_{\mf{m}\mf{s}_jB_j'}\omega
$$
while
\begin{align*}
\mu_B(F_B(\mf{m}^{\perp}))&=\int_{\mathcal{B}_j}\int_{F(\mf{m}^{\perp})}\omega=\int_{\mathcal{B}_j}\int_{\mf{m}^{\perp}\mf{s}_i^{\perp}B_i}\omega
=\int_{\mathcal{B}_j}\int_{\mf{d}\mf{m}^{-1}\mf{d}\mf{s}_i^{\perp}B_i}\omega=\int_{\mf{m}\mf{d}\mathcal{B}_j}\int_{\mf{d}\mf{s}_i^{\perp}B_i}\omega\\
&=\int_{\mf{m}\mf{d}\mathcal{B}_j}\int_{\mf{s}_i\mf{d}B_i}\omega=\int_{\mf{m}\mf{s}_jB_j'}\int_{\mf{s}_iA_i}\omega=\int_{\mf{m}\mf{s}_jB_j'}\int_{\mathcal{B}_i'}\omega.
\end{align*}
Here we are using the fact that $\omega=|z-w|^{-4}\;du\;dv\;dx\;dy$ is the invariant form and the relations in (\ref{regionequality}).
\end{proof}

\begin{lemma}
  \label{lem:measurecompare2}
For $\mf{m}$ in swap normal form, and $\mf{n}$ in invert normal form, there are equalities of measure
$$
\mu_B(F_B(\mf{m}))=\mu_A(F_A(\mf{m}^{-1})), \quad \mu_A(F_A(\mf{n}))=\mu_B(F_B(\mf{n}^{-1})).
$$
\end{lemma}
\begin{proof}
Consider the case where $\mf{m}=\mf{s}_i\dots\mf{s}_j$ so that $F_B(\mf{m})=\mf{m}\mf{s}_jB_j'\subseteq B_i'$ and $F_A(\mf{m}^{-1})=\mf{m}^{-1}\mf{s}_iA_i\subseteq A_j$ (all other cases are analogous) and let $\omega=|z-w|^{-4}\;du\;dv\;dx\;dy$ be the invariant form on the space of geodesics.  Then
\begin{align*}
\mu_B(F(\mf{m}))=\int_{\mathcal{B}_i'}\int_{F_B(m)}\omega=\int_{\mathcal{B}_i'}\int_{\mf{m}\mf{s}_jB_j'}\omega&=\int_{\mf{s}_iA_i}\int_{\mf{m}\mathcal{A}_j}\omega=\int_{\mf{m}^{-1}\mf{s}_iA_i}\int_{\mathcal{A}_j}\omega\\
\end{align*}
while
\begin{equation*}
\mu_A(F(\mf{m}^{-1}))=\int_{F_A(\mf{m}^{-1})}\int_{\mathcal{A}_j}\omega=\int_{\mf{m}^{-1}\mf{s}_iA_i}\int_{\mathcal{A}_j}\omega.
\end{equation*}
Here we are using the fact that $\omega=|z-w|^{-4}\;du\;dv\;dx\;dy$ is the invariant form and the relations in (\ref{regionequality}).
\end{proof}

\section{Dynamical systems on Lorentz and Descartes quadruples}\label{sec:lor-des}

The dynamical systems introduced in the previous section can be translated to ones on Lorentz and Descartes quadruples.  In this section, we explore these systems as well as how they interact with individual Apollonian packings.

\subsection{Dynamics on Lorentz quadruples}\label{sec:lorentz}

In analogy to Romik \cite{R}, we now define a dynamical system on integer Lorentz quadruples $a^2=b^2+c^2+d^2$ with $a>0$ which decreases the value of $a$ and terminates at one of the six quadruples
$$
(g,\pm g,0,0), \ (g,0,\pm g,0), \ (g,0,0,\pm g), \ g=\text{gcd}(a,b,c,d).
$$


Define the following dynamical system on $C = \{ (a,b,c,d) : a^2 = b^2 + c^2 + d^2, a > 0 \} \subset \RR^4$:
\begin{align*}
  T_{L}(a,b,c,d)=\left\{
\begin{array}{c}
L_1^{\perp}(a,b,c,d)^t \text{ if } 2a+b+c+d\leq a\\
L_2^{\perp}(a,b,c,d)^t \text{ if } 2a+b-c-d\leq a\\
L_3^{\perp}(a,b,c,d)^t \text{ if } 2a-b+c-d\leq a\\
L_4^{\perp}(a,b,c,d)^t \text{ if } 2a-b-c+d\leq a\\
\text{if none of the above then}\\
L_1(a,b,c,d)^t \text{ if } 2a-b-c-d \leq a\\
L_2(a,b,c,d)^t \text{ if } 2a-b+c+d\leq a\\
L_3(a,b,c,d)^t \text{ if } 2a+b-c+d\leq a\\
L_4(a,b,c,d)^t \text{ if } 2a+b+c-d\leq a\\
\end{array}
\right..
\end{align*}

Iteration of $T_L$ produces a word $W_1 W_2 \cdots W_n \cdots$ defined by $T_L^k(a,b,c,d) = W_k T_L^{k-1}(a,b,c,d)$.

\begin{theorem}
  \label{thm:TL}
  The word $W_1 W_2 \cdots W_n $ produced by iteration of $T_L$ on a primitive integer Lorentz quadruple $(a,b,c,d)$ is in swap normal form and satisfies 
  \[
    (a,b,c,d)^t =	W_1 W_2 \cdots W_k \mathbf{b}^t
     \]
     where $\mathbf{b}$ is one of the following \emph{simplest Lorentz quadruples}:
     \[
       (1,1,0,0), (1,0,1,0), (1,0,0,1), (1,-1,0,0), (1,0,-1,0), (1,0,0,-1),
     \]
     under iteration of $T_L$.
\end{theorem}

See Figure \ref{fig:lor-graph}.

\begin{proof}
  The map is well defined unless  
  \[
    a \pm b \pm c \pm d > 0
  \]
  for all choices of signs.  In this case $a < |b| + |c| + |d|$, which is impossible given $a^2 = b^2 + c^2 + d^2$.

  Note that the map preserves $\gcd(a,b,c,d)$, so it is well-defined 
on $\widetilde{C}$, the primitive integral Lorentz quadruples of $C$.  
  

  Therefore, it suffices to verify the following:
\[
  T_L( W_1 W_2 \cdots W_n \mathbf{b}^t) = 
  W_2 W_3 \cdots W_n \mathbf{b}^t.
\]
We do this by constructing an explicit intertwining map demonstrating that $(\PP^1(\mbb{Q}(i)), T_B)$ and $(\widetilde{C}, T_L)$ are conjugate.  Then the result will follow from Theorem \ref{thm:z-mfz}.

In analogy to work of Romik on Pythagorean triples, it is possible to scale this system to act on a sphere.  
Scaling so that $a=1$ ($X=b/a$, $Y=c/a$, $Z=d/a$) (call this projection $\pi$), we get a system on the sphere 
\begin{align*}
T_{sph}(X,Y,Z)=\left\{
\begin{array}{c}
\frac{(-1-Y-Z,-1-X-Z,-1-X-Y)}{2+X+Y+Z} \text{ if } 1+X+Y+Z<0\\
\frac{(-1+Y+Z,1+X-Z,1+X-Y)}{2+X-Y-Z} \text{ if } 1+X-Y-Z<0\\
\frac{(1+Y-Z,-1+X+Z,1-X+Y)}{2-X+Y-Z} \text{ if } 1-X+Y-Z<0\\
\frac{(1-Y+Z,1-X+Z,-1+X+Y)}{2-X-Y+Z} \text{ if } 1-X-Y+Z<0\\
\text{if none of the above, then}\\
\frac{(1-Y-Z,1-X-Z,1-X-Y)}{2-X-Y-Z} \text{ if } 1-X-Y-Z<0\\
\frac{(1+Y+Z,-1+X-Z,-1+X-Y)}{2-X+Y+Z} \text{ if } 1-X+Y+Z<0\\
\frac{(-1+Y-Z,1+X+Z,-1-X+Y)}{2+X-Y+Z} \text{ if } 1+X-Y+Z<0\\
\frac{(-1-Y+Z,-1-X+Z,1+X+Y)}{2+X+Y-Z} \text{ if } 1+X+Y-Z<0\\
\end{array}
\right..
\end{align*}

\begin{figure}[h]
\begin{center}
\includegraphics[scale=.5]{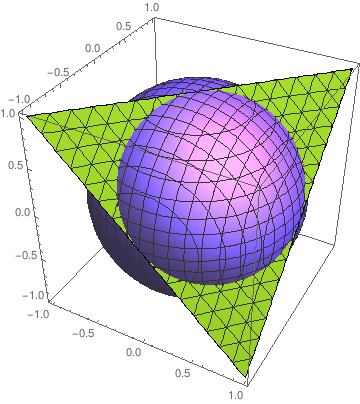}
\caption{The hyper-ideal tetrahedron in the Klein model}
\label{fig:idealtet}
\end{center}
\end{figure}

The regions defined on the sphere (Figure \ref{fig:idealtet}) are given by the intersection of the sphere with the hyper-ideal tetrahedron defined by the linear inequalities
$$
1+X+Y+Z\geq0, 1+X-Y-Z\geq0, 1-X+Y-Z\geq0, 1-X-Y+Z\geq0
$$
in the Klein projective model of hyperbolic space, and the cases of $T_{sph}$ are reflections in those geodesic planes and the planes of the dual tetrahedron.

The systems $(S^2,T_{sph})$ and $(\PP^1(\mbb{C}),T_{B})$ are conjugate, moving between the projective and the upper half-space models of hyperbolic space.  Specifically, after stereographic projection $(X,Y,Z)\mapsto(\frac{X}{1-Z},\frac{Y}{1-Z})=z$, rotating ($e^{-\pi i/4}z$), scaling ($z/\sqrt{2}$), shifting ($z+\frac{1+i}{2}$), and switching two circles ($\frac{\bar{z}}{-i\bar{z}+1}$), we obtain an intertwining map
$$
\phi(X,Y,Z)=\frac{i\bar{z}+1}{\bar{z}+1}=\frac{(1+Y-Z)+iX}{(1+X-Z)-iY}, \ T_{B}\circ\phi=\phi\circ T_{sph}.
$$
Under the composition $\phi\circ\pi$, one can associate a Lorentz quadruple $(a,b,c,d)$ with the complex point
\[
  z = \phi \circ \pi(a,b,c,d)^t := \frac{ a+c-d + bi }{a+b-d-ci},
\]
where $\pi$ is the scaling projection from above, the correspondence between fixed points being
$$
(1,\pm1,0,0)\mapsto\frac{1}{1-i}, \infty, \quad (1,0,\pm1,0)\mapsto 1+i,0, \quad (1,0,0,\pm1)\mapsto i,1.
$$
Then,
\[
  \mf{s}_i \circ \phi \circ \pi ( a,b,c,d )^t = \phi \circ \pi \circ L_i (a,b,c,d)^t, \quad
  \mf{s}_i^\perp \circ \phi \circ \pi ( a,b,c,d )^t = \phi \circ \pi \circ L_i^\perp (a,b,c,d)^t,
\]
and the conditions defining the cases of $T_L$ and $T_B$ correspond. 
Therefore, $T_B\circ\phi\circ\pi=\phi\circ\pi\circ T_L$.
\end{proof}

To summarize the above:  the following diagram commutes, $\phi$ is invertible, and there is a unique primitive representative in $\pi^{-1}(q)$ for rational $q\in S^2$.
\[ \xymatrix{
C \ar[r]^{T_L} \ar[d]_{\pi} & C \ar[d]^{\pi}\\
S^2 \ar[r]^{T_{sph}} \ar[d]_{\phi} & S^2 \ar[d]^{\phi}\\
\mbb{P}^1(\mbb{C}) \ar[r]^{T_B} & \mbb{P}^1(\mbb{C})\\
}\]



\subsection{Dynamics on Descartes quadruples}\label{section:descartes}

Under the change of variables of Section \ref{sec:quads}, the dynamical system of the last section becomes a 
dynamical system on primitive integer Descartes quadruples.  Define the following:
\begin{align*}
  T_{S}(a,b,c,d)=\left\{
\begin{array}{c}
S_1^{\perp}(a,b,c,d)^t \text{ if } a<0\\
S_2^{\perp}(a,b,c,d)^t \text{ if } b<0\\
S_3^{\perp}(a,b,c,d)^t \text{ if } c<0\\
S_4^{\perp}(a,b,c,d)^t \text{ if } d<0\\
\text{if none of the above then}\\
S_1(a,b,c,d)^t \text{ if } b+c+d<a\\
S_2(a,b,c,d)^t \text{ if } a+c+d<b\\
S_3(a,b,c,d)^t \text{ if } a+b+d<c\\
S_4(a,b,c,d)^t \text{ if } a+b+c<d\\
\end{array}
\right..
\end{align*}

Iteration of $T_S$ produces a word $W_1 W_2 \cdots W_n$ defined by $T_S^k(a,b,c,d)^t = W_k T_S^{k-1}(a,b,c,d)^t$.  

\begin{theorem}
  The word $W_1 W_2 \cdots W_n $ produced by iteration of $T_S$ on a primitive integer Descartes quadruple $(a,b,c,d)$ is in swap normal form and satisfies 
  \[
       (a,b,c,d)^t =	W_1 W_2 \cdots W_k \sigma(1,1,0,0)^t
     \]
     where $\sigma$ is some permutation of the entries of the vector.  In other words, any primitive integral Descartes quadruple $(a,b,c,d)$ eventually reaches one of the so-called \emph{simplest Descartes quadruples}
     \[
       (1,1,0,0), (1,0,1,0), (1,0,0,1), (0,1,1,0), (0,1,0,1), (0,0,1,1)
     \]
     under iteration of $T_S$.
\end{theorem}

The proof is immediate from the previous sections, using conjugation by $J$ defined in (\ref{Jmatrix}).

We now state an analogous system conjugate to $T_A$.  The proof is similar.

Define
\begin{align*}
  T_{I}(a,b,c,d)=\left\{
\begin{array}{c}
S_1(a,b,c,d)^t \text{ if } b+c+d<a\\
S_2(a,b,c,d)^t \text{ if } a+c+d<b\\
S_3(a,b,c,d)^t \text{ if } a+b+d<c\\
S_4(a,b,c,d)^t \text{ if } a+b+c<d\\
\text{if none of the above then}\\
S_1^{\perp}(a,b,c,d)^t \text{ if } a<0\\
S_2^{\perp}(a,b,c,d)^t \text{ if } b<0\\
S_3^{\perp}(a,b,c,d)^t \text{ if } c<0\\
S_4^{\perp}(a,b,c,d)^t \text{ if } d<0\\
\end{array}
\right..
\end{align*}

Iteration of $T_I$ produces a word $W_1 W_2 \cdots W_n$ defined by $T_I^k(a,b,c,d) = W_k T_I^{k-1}(a,b,c,d)$.  

\begin{theorem}
  The word $W_1 W_2 \cdots W_n $ produced by iteration of $T_I$ on a primitive integer Descartes quadruple $(a,b,c,d)$ is in invert normal form and satisfies 
  \[
       (a,b,c,d)^t =	W_1 W_2 \cdots W_k \sigma(1,1,0,0)^t
     \]
     where $\sigma$ is some permutation of the entries of the vector.  In other words, any primitive integral Descartes quadruple $(a,b,c,d)$ eventually reaches one of the so-called \emph{simplest Descartes quadruples}
     \[
       (1,1,0,0), (1,0,1,0), (1,0,0,1), (0,1,1,0), (0,1,0,1), (0,0,1,1)
     \]
     under iteration of $T_I$.
\end{theorem}


\subsection{Dynamics on Apollonian circle packings}

The Apollonian circle packist will be interested to consider how the dynamical systems interact with individual Apollonian circle packings.  Each Apollonian circle packing has a root quadruple, i.e. the largest four pairwise tangent circles in the packing.  The main result of this section is to show that the invert normal form word $W_1 W_2 \cdots W_n$ produced by the dynamical system $T_I$ is such that the longest substring $W_1 W_2 \cdots W_k$ consisting only of swaps will end with the root quadruple of the packing containing the initial quadruple.  In other words, the dynamical system $T_I$, or equivalently $T_A$, moves any quadruple to the root of its Apollonian circle packing, then inverts, then moves to the root, then inverts, etc. 

\begin{theorem}
  \label{thm:acp-only-dyn}
  The dynamical system $T_I$ moves a quadruple to the root of its packing via swaps before inverting. 
\end{theorem}

In particular, while it remains in a single Apollonian circle packing, the dynamical system $T_I$ agrees with the Reduction Algorithm for Descartes quadruples of \cite[Section 3]{GLMWY0}, until the last step (when Graham, Lagarias, Mallows, Wilks and Yan reorder the quadruple).

  The proof requires several lemmas.  A Descartes quadruple $(a,b,c,d)$ is a root quadruple of the Apollonian packing in which it resides if $a\leq 0\leq b\leq c\leq d$, and $a+b+c\geq d$.  Furthermore, it exists if the packing is integral, and is unique \cite[Section 3]{GLMWY2}.   

\begin{lemma}\label{fliptoroot}
Let $(x,y,z,w)$ be a Descartes quadruple such that the $i$-th coordinate is not maximal in the quadruple.  Let $S_{ii}=S_i^\perp S_i$.  Then $S_{ii}(x,y,z,w)^t$ is a root quadruple of the packing in which it resides, possibly after re-ordering the coordinates.
\end{lemma}

\begin{proof}
We have
\begin{equation*}
 \small{
S_{11}=\left(
\begin{array}{llll}
1&-2&-2&-2\\
-2&5&4&4\\
-2&4&5&4\\
-2&4&4&5\\
\end{array}
\right)\quad
S_{22}=\left(
\begin{array}{llll}
5&-2&4&4\\
-2&1&-2&-2\\
4&-2&5&4\\
4&-2&4&5\\
\end{array}
\right)}
\end{equation*}
\begin{equation*}
 \small{
S_{33}=\left(
\begin{array}{cccc}
5&4&-2&4\\
4&5&-2&4\\
-2&-2&1&-2\\
4&4&-2&4\\
\end{array}
\right)\quad
S_{44}=\left(
\begin{array}{cccc}
5&4&4&-2\\
4&5&4&-2\\
4&4&5&-2\\
-2&-2&-2&1\\
\end{array}
\right),}
\end{equation*}
We prove the lemma for the case $i=1$ and note that the other cases are identical.  We have
$$S_{11}(x,y,z,w)^t=(x-2y-2z-2w,-2x+5y+4z+4w, -2x+4y+5z+4w, -2x+4y+4z+5w)^t.$$
Since $x$ is not maximal among $x,y,z,w$, we have that the first coordinate of $S_1(x,y,z,w)^t$ is $\geq0$ (if $x$ is negative, $-x+2y+2z+2w$ is a sum of positive numbers and hence it is positive; if $x$ is nonnegative, it is enough to argue that the first coordinate of $S_1(x,y,z,w)^t$ is at least $x$).  Hence $x-2y-2z-2w\leq 0\leq -2x+5y+4z+4w, -2x+4y+5z+4w, -2x+4y+4z+5w$.  Without loss of generality, assume $y\leq z\leq w$, so that
$$x-2y-2z-2w\leq 0\leq -2x+5y+4z+4w\leq -2x+4y+5z+4w\leq-2x+4y+4z+5w$$
In order to show that the above is a root quadruple, we need only to show that 
\begin{equation}\label{rootcheck}
x-2y-2z-2w-2x+5y+4z+4w-2x+4y+5z+4w+2x-4y-4z-5w=-x+3y+3z+w\geq0.
\end{equation}
Using Descartes' theorem, and the fact that $w$ is maximal, we have that
$$w=\frac{2x+2y+2z+\sqrt{16xy+16xz+16yz}}{2}=x+y+z+2\sqrt{xy+xz+yz}.$$
So the expression in (\ref{rootcheck}) can be rewritten as
$$4y+4z+2\sqrt{xy+xz+yz}.$$
If $0\leq y\leq z$, then this is clearly nonnegative and we are done.  Suppose $y<0$.  Then the circle corresponding to $y$ in the Descartes quadruple $(x,y,z,w)$ contains the one of curvature $z$, and so the circle of curvature $z$ has smaller radius and hence larger curvature than the one of curvature $y$.  Thus $4y+4z>0$ and the expression in (\ref{rootcheck}) is nonnegative as desired.
\end{proof}

\begin{lemma}\label{roottoroot}
Let $(a,b,c,d)$ be a root quadruple of some packing.  Then $S_i^\perp(a,b,c,d)^t$ is a root quadruple of another packing, after re-ordering, for any $2\leq i\leq 4$. 
\end{lemma}

\begin{proof}
We consider the case where $i=2$ and note that the other cases are identical.  We have $S_2^t(a,b,c,d)^t=(2b+a,-b,2b+c,2b+d)^t$, and $-b\leq 0\leq 2b+a\leq2b+c\leq 2b+d$.  We compute
$$-b+2b+a+2b+c-2b-d=b+a+c-d\geq 0$$
since $(a,b,c,d)$ is a root quadruple, and hence $(2b+a,-b,2b+c,2b+d)$ is also a root quadruple after reordering.
\end{proof}

\begin{proof}[Proof of Theorem \ref{thm:acp-only-dyn}]
  The system $T_{I}$ generates a word satisfying $(a,b,c,d)^t = W_1 W_2 \cdots W_n \mf{b}$ where $\mf{b}$ is a simplest Descartes quadruple.  In this word, swaps are as far left as possible and inversions as far right as possible.  Therefore if the leftmost inversion $S_i^\perp$ occurs at some position $k$, i.e. $W_k = S_i^\perp$, it is because it is followed by $S_i$, or because it occurs as the last letter ($k=n$), or else because it is followed by another inversion $S_j^\perp$. 

  We assume that the leftmost inversion is $W_1$, and will show that $(a,b,c,d)^t$ is a root quadruple.  The theorem will follow.
  
  In the first case, $(a,b,c,d)^t = S_i^\perp S_i (x,y,z,w)^t$, where the $i$-th coordinate is not maximal (since it is a result of $S_i$ in the application of $T_S$).  Therefore, by the first lemma, $(a,b,c,d)$ is a root quadruple.

  In the second case, $(a,b,c,d)$ is created by an inversion from a simplest Descartes quadruple, which is, in particular, a root quadruple.  But it is not an inversion in a circle of curvature $0$, since that would not change the Descartes quadruple.  Therefore the second lemma applies.

  In the third case, $(a,b,c,d)^t = S_i^\perp S_j^\perp (x,y,z,w)^t$, where $i \neq j$.  Then, by induction (with the previous two cases as base cases and the second lemma as an inductive step), $(a,b,c,d)^t$ is a base quadruple.  (Since $i \neq j$, we know the $i$-th circle is not the largest in $S_j^\perp(x,y,z,w)^t$.)
\end{proof}


\section{Typical expansions of a point from two perspectives}
\label{sec:app}

What can one say about the chain of swaps and inversions that are used in the rational approximation of a typical point under iteration?  First, we consider the question as a limiting question on finite expansions.  In the second section, we consider the question for all expansions.  Finally, we provide some numerical data.

\subsection{Digit probabilities in finite expansions}
\label{sec:stats1}

We consider the behavior of the expansions of rational points (those points with finite expansion).  We use a measure of the \emph{height} of a rational point.  For example, in the case of Lorentz quadruples, we may define the set of points of height $N$ to be
\[ 
X_N:=\{(a,b,c,d)\in\mathbb{Z}^4 : a^2=b^2+c^2+d^2, 0<a \le N, \gcd(b,c,d)=1\}.
\]
We then consider $X_N$ to be a discrete probability space with uniform probability measure.  Under this measure, we can ask about the distribution of the $n$-th digit in the expansion, as $N \rightarrow \infty$.  In this section, we prove the following theorem.

Write $\delta(X,Y,Z)$ for the letter produced by applying $T_{Sph}$ to $(X,Y,Z)$.  Let $d\mu=\frac{1}{4\pi}dA$ be the normalized uniform area measure on the sphere.

\begin{theorem} \label{digit:exp}
  Under the uniform probability measure for $X_N$, the distribution of $n$-th digit in the expansion of a random Lorentz quadruple $(a,b,c,d)\in X_N$, converges to the distribution of $\delta$ under the measure 
  $\mathcal{F}^{n-1}(\mathbf{I})d\mu$, where $\mathbf{I}$ is the constant function $1$ and $\mathcal{F}$ is the the transfer operator of $T_{Sph}$.
In particular, the distributions of the $1$st and $2$nd digits converge to the distribution of $\delta$ under the measures
\[
  d\mu, \quad \left( \sum \frac{1}{(2 \pm X \pm Y \pm Z)^4 } \right) d\mu,
\]
respectively. (The sum is over all choices of sign combinations.)
\end{theorem}

For example, the probabilities of possible first digits approach the proportional areas of the circular and triangular regions on the sphere in Figure \ref{fig:idealtet}.  These are, respectively, the probability the first digit is an inversion, 
$$
\frac{1}{\pi}
\int_{1/\sqrt{2}}^{\infty}\int_{-\infty}^{\infty}\frac{4dxdy}{(1+x^2+y^2)^2}=2\left(1-\frac{1}{\sqrt{3}}\right)= 0.84529946\dots
$$
(the integrand is the pushforward of surface area on the sphere to the plane under stereographic projection); and the probability it is a swap,
$$
\frac{1}{4\pi}
\left(4\pi-8\pi\left(1-\frac{1}{\sqrt{3}}\right)\right)=\left(\frac{2}{\sqrt{3}}-1\right)= 0.15470053\dots.
$$

Note that the same result applies, via application of the change of coordinates $J$, to Descartes quadruples chosen uniformly among those whose sum of curvatures is less than $N$.  It is also worth remarking that this sort of result is limited in the sense that a slightly different way of measuring height, say the maximum of the curvatures is less than $N$, may potentially lead to different probabilities.

%

The transfer operator $\mathcal{F}:L^1(S^2,\mu)\rightarrow L^1(S^2,\mu)$ arising from the transformation $T_{Sph}$ is defined in the following way: for $f\in L^1(S^2,\mu)$, we have
\[(\mathcal{F}f)(\tilde{X})=\sum_{\tilde{Y}\in f^{-1}(\tilde{X})}g(\tilde{Y})f(\tilde{Y}),\]
where $g$ is the inverse of the Jacobian of $T_{sph}$. 
\begin{proof}[Proof of Theorem \ref{digit:exp}]
By definition, the transfer operator is given by 
\begin{align*}
( \mathcal{F}f)(X,Y,Z)&=\frac{1}{(2+X+Y+Z)^4}f\left(\frac{-1-Y-Z}{2+X+Y+Z},\frac{-1-X-Z}{2+X+Y+Z},\frac{-1-X-Y}{2+X+Y+Z}\right) \\& + \frac{1}{(2+X-Y-Z)^4}f\left(\frac{-1+Y+Z}{2+X-Y-Z},\frac{1+X-Z}{2+X-Y-Z},\frac{1+X-Y}{2+X-Y-Z}\right) \\& + \frac{1}{(2-X+Y-Z)^4}f\left(\frac{1+Y-Z}{2-X+Y-Z},\frac{-1+X+Z}{2-X+Y-Z},\frac{1-X+Y}{2-X+Y-Z}\right) \\& +  \frac{1}{(2-X-Y+Z)^4}f\left(\frac{1-Y+Z}{2-X-Y+Z},\frac{1-X+Z}{2-X-Y+Z},\frac{-1+X+Y}{2-X-Y+Z}\right)\\& +  \frac{1}{(2-X-Y-Z)^4}f\left(\frac{1-Y-Z}{2-X-Y-Z},\frac{1-X-Z}{2-X-Y-Z},\frac{1-X-Y}{2-X-Y-Z}\right)\\& +  \frac{1}{(2-X+Y+Z)^4}f\left(\frac{1+Y+Z}{2-X+Y+Z},\frac{-1+X-Z}{2-X+Y+Z},\frac{-1+X-Y}{2-X-Y+Z}\right)\\& +  \frac{1}{(2+X-Y+Z)^4}f\left(\frac{-1+Y-Z}{2+X-Y+Z},\frac{1+X+Z}{2+X+Y-Z},\frac{-1-X+Y}{2+X-Y+z}\right)\\& +  \frac{1}{(2+X+Y-Z)^4}f\left(\frac{-1-Y+Z}{2+X+Y-Z},\frac{-1-X+Z}{2+X+Y-Z},\frac{1+X+Y}{2+X+Y-Z}\right).
\end{align*}
Therefore the second part of the theorem follows from the first.  

The result follows if we show that for a given region $R$ in $S^2$, 
\begin{align} \label{eqn1}
  \lim_{N\rightarrow\infty}\frac{|\{(a,b,c,d)\in X_N : (b/a,c/a,d/a)\in R\}|}{|X_N|} = \mu(R).
\end{align}
This would show that the random vector $(b/a,c/a,d/a)$ converges in distribution to $\mu$. Therefore, the transformation $T_{sph}(X,Y,Z)$ has distribution $(\mathcal{F}(\mathbf{I}))(X,Y,Z)~d\mu$.

It suffices to consider regions $R$ on the unit sphere given by 
\[
R=\{ (\sin \theta \cos \phi,\sin \theta \sin \phi,\cos \theta):s_1<\theta<s_2,t_1<\phi<t_2\}.
\]
We let $S_{s_1,s_2,t_1,t_2}$ be the three dimensional region given by 
\[
S_{s_1,s_2,t_1,t_2}=\left\{ (x,y,z):x^2+y^2+z^2\leq 1,s_1<\arccos \frac{z}{\sqrt{x^2+y^2+z^2}}<s_2,
t_1<\arctan \frac{y}{x}<t_2\right\}.
\]
Now 
\begin{align*}
  &\frac{ |\{(a,b,c,d)\in X_N : (b/a,c/a,d/a)\in R\}| }{ |X_N| } \\
&= \frac{ \# \{ (a,b,c,d)\in \mathbb{Z}^4:a>0,\gcd (b,c,d)=1,b^2+c^2+d^2\leq N^2,(b/a,c/a,d/a)\in R\}}
{ \# \{ (a,b,c,d)\in \mathbb{Z}^4:a>0,\gcd (b,c,d)=1,b^2+c^2+d^2\leq N^2\}}\\
&=\frac{ \# \{ (b,c,d)\in \mathbb{Z}^3:\gcd (b,c,d)=1,(b/N,c/N,d/N)\in S_{s_1,s_2,t_1,t_2}\}}
{\# \{ (b,c,d)\in \mathbb{Z}^3:\gcd (b,c,d)=1,(b/N,c/N,d/N)\in S_{0,2\pi,0,\pi}\}}\\
&=\left( 1+o(1)\right)\frac{N^3\frac{1}{\zeta(3)}\text{ volume of }S_{s_1,s_2,t_1,t_2} }{N^3\frac{1}{\zeta(3)}\text{ volume of } S^2}=\left( 1+o(1)\right)\frac{1/3 (\cos s_1-\cos s_2)(t_2-t_1)}{4\pi/3}\\
&=\left( 1+o(1)\right)\frac{1}{4\pi}(\cos s_1-\cos s_2)(t_2-t_1)=\left( 1+o(1)\right)\frac{dA}{4\pi},
\end{align*} which proves \eqref{eqn1}.

\end{proof}

\subsection{Expansion statistics assuming ergodicity}
\label{sec:stats2}

The results in this section depend on the conjectural ergodicity of the systems $(\mathbb P^1(\mathbb C), T_A,\mu_A)$ and $(\mathbb P^1(\mathbb C), T_B,\mu_B)$ constructed in the previous section.  These systems are similar to the system constructed by Schmidt in \cite{S1}, which is ergodic (the main result of Schmidt's \cite{S2}).  Similarly, Romik constructs a measure-preserving system on the first quadrant of the unit circle in \cite{R} which he proves to be ergodic.  In both Romik's and Schmidt's work, ergodicity is then applied to provide information about the expansions of a typical point.

While we leave the proof of ergodicity to a subsequent paper, we state it as a conjecture here.

\begin{conjecture}\label{con:ergodicity}
The systems $(\mathbb P^1(\mathbb C), T_A,\mu_A)$ and $(\mathbb P^1(\mathbb C), T_B,\mu_B)$ are ergodic.
\end{conjecture}

Then standard theorems of ergodic theory would imply that for almost all $z \in \PP^1(\CC)$,
$$
\lim_{n\to\infty}\frac{1}{n}\sum_{i=1}^n\phi(T_A^iz)=\int\phi \ d\mu
$$
with $\phi$ an indicator function for various Farey circles and triangles (and similarly with $T_B$ in place of $T_A$).  In particular, by judicious choice of an indicator function, we can compute the frequencies of certain chains of swaps and inversions occurring in the typical expansion of a point.

\subsubsection{Example:  Two or three swaps in a row}
\label{example1}

The frequency of two swaps in a row $\{S_iS_j : i,j\in[1,4], i\neq j\}$, as well as the frequency of two inversions in a row $\{S_i^\perp S_j^\perp : i,j\in[1,4], i\neq j\}$ is $0.345299\dots$.  The frequency of three swaps (or inversions) in a row, $\{S_iS_jS_k : i,j,k\in[1,4], i\neq j, j\neq k\}$, is $0.246913\dots$.  

With reference to Figure \ref{fig:length2partition}, there are twelve triangular regions corresponding to two swaps in a row.  The total measure of these triangular regions, as a proportion of the total measure of the plane, is the frequency we wish to compute.  As noted in Section~\ref{sec:inv}, the total area with respect to $\mu_A$ or $\mu_B$ of all eight Farey circles and triangles (the entire plane) is $2\pi^2$.  

The area of a region with respect to $\mu_A$ and $\mu_B$ is exactly hyperbolic area multiplied by a factor of $\pi/4$.  
The hyperbolic area of a Euclidean disk of radius $r$ in the upper half-plane whose center has $y$-coordinate $b>0$ is
\begin{equation}
I(\alpha)=\int_{-1}^1\frac{2\sqrt{1-x^2}}{\alpha^2-(1-x^2)}dx=2\pi\left(\frac{\alpha}{\sqrt{\alpha^2-1}}-1\right) ({\mbox{cf. lemmas 2.2, 2.3 of \cite{S2}}}),
\end{equation}
where $\alpha=b/r>1$.  The hyperbolic area of a translate (in the $y$ direction by $\alpha$) of the ideal triangle with vertices $0,1,\infty$ is
\begin{equation}\label{Jint}
J(\alpha)=\int_0^1\frac{dx}{\alpha+\sqrt{x(1-x)}}=\pi-\frac{2\arccos(1/2\alpha)}{\sqrt{1-(1/2\alpha)^2}}.
\end{equation}
Note:  this is $\Phi(1/2\alpha)$ where $\Phi$ is as in Lemma~3.3 from \cite{S2}.  

Putting this together, we have that the frequency of two swaps in a row is 
$$\frac{12\cdot(\frac{\pi}{4}\cdot J(1))}{2\pi^2} = 0.345299\dots,$$
the frequency of three swaps in a row is
$$
\frac{12}{8\pi}\left(J(1)-I(4)\right)=0.246913\dots,
$$
and so on.  

We may use Lemmas \ref{lem:measurecompare1} and \ref{lem:measurecompare2} to translate integrals over Farey circles to integrals over Farey triangles, and vice versa.  In particular, the frequency of $n$ swaps will equal the frequency of $n$ inversions.

\subsubsection{Example:  Certain strings of Schmidt}
\label{example2}

The frequency in the expansion of almost every $z\in\mathbb{C}$ of a string of alternating swaps or inversions of length $n$,
$$
S_{i/j}\dots S_jS_iS_jS_i, \quad S_{i/j}^{\perp}\dots S_j^{\perp}S_i^{\perp}S_j^{\perp}S_i^{\perp} \quad (i\neq j \text{ fixed}),
$$
is $J(n-1)/8\pi$, where $J(\alpha)$ is as in (\ref{Jint}).  These are probabilities associated to a string of only swaps or only inversions having a common fixed point.  Consider the probability that one of the vertices is fixed for \textit{exactly} $n$ iterations by a string of only swaps or only inversions, i.e. the frequency of strings
$$
MS_{i/j}\dots S_jS_iS_jS_iN, \quad MS_{i/j}^{\perp}\dots S_j^{\perp}S_i^{\perp}S_j^{\perp}S_i^{\perp}N \quad (i\neq j \text{fixed}),
$$
where $M,N\neq S_i,S_j$ on the left and $M,N\neq S_i^{\perp},S_j^{\perp}$ on the right.  This frequency is (cf. \cite[Theorem 5.3A]{S2})
$$
\frac{1}{8\pi}(J(n-1)-2J(n)+J(n+1)).
$$
In particular, the conjectured values for length 1, 2 and 3 Schmidt strings are
\[
  0.084117\ldots, \quad
  0.007180\ldots, \quad
  0.002249\ldots
\]
respectively.

\subsection{Experimental data}

We ran two experiments.  In the first, we generated all Lorentz quadruples with $a < 200$ and computed the frequencies of certain substrings, as well as the frequencies of the first several digits.  The former are shown in Table \ref{table:freqs}.  The frequency of swaps and inversions in the first and second digits are shown in Table \ref{table:bydigit-freqs}.  These results are to be compared to Section \ref{sec:stats1}.  

In the second, we approximated the frequencies in a random expansion as follows:  we selected 100 points uniformly at random from the $[0,1]\times[0,1]$ square and computed the first 100 letters of the expansion iterating $T_A$ or $T_B$, and tabulated the frequency of certain substrings.  These results are to be compared to Section \ref{sec:stats2}.

\begin{table}
  \begin{tabular}{lllll}
    string & exp 1 & exp 2 & conjecture & unif freq \\
    \hline
    2 swaps, $\{S_iS_j\}$ & 0.338\ldots &0.328\ldots  & 0.345299\ldots & 12/44=0.273\dots \\
    2 inversions, $\{S_i^\perp S_j^\perp\}$  & 0.329\ldots & 0.347\ldots  & 0.345299\ldots & 12/44=0.273\dots\\
    3 swaps, $\{S_iS_jS_k\}$ & 0.235\ldots & 0.224\ldots  & 0.246913\ldots & 36/224=0.1607\dots \\
    3 inversions, $\{S_i^\perp S_j^\perp S_k^\perp\}$ & 0.220\ldots &0.251\ldots  & 0.246913\ldots & 36/224=0.1607\dots \\
  \end{tabular}
  \\
  \caption{Experimental results for frequencies of random expansions.  The column \emph{exp 1} refers to the frequencies from bounded Lorentz quadruples; while \emph{exp 2} refers to the frequences from random point expansions.  The column \emph{conjecture} refers to the conjectured value for random point expansions (compare to \emph{exp 2}).  The column \emph{unif freq} refers to the frequency expected if words were uniformly random. } 
  \label{table:freqs}
\end{table}

\begin{table}
  \begin{tabular}{lllll}
    digit & experiment & theorem \\
    \hline
    first swap & 0.161\ldots & 0.154\ldots \\
    first inversion & 0.838\ldots & 0.845\ldots \\
    second swap & 0.365\ldots &  \\
    second inversion & 0.634\ldots &  \\
  \end{tabular}
  \caption{Frequencies of digits or swaps in first or second digit of the expansions of bounded Lorentz quadruples.}
  \label{table:bydigit-freqs}
\end{table}

\section{Restriction to the real line and its ergodic invariant measure}
\label{sec:real}

\subsection{The dynamical system and its ergodic invariant measure}
In this section, we consider the action of $T_A$ and $T_B$ on the real line, thus obtaining a version of the Euclidean algorithm.  On $\PP^1(\mbb{R})$, $T_A$ and $T_B$ agree, the map being 
$$
t(x):=\left\{
\begin{array}{cc}
\mf{a}(x)=-x&x\in A=[-\infty,0]\\
\mf{b}(x)=\frac{x}{2x-1}&x\in B=[0,1]\\
\mf{c}(x)=2-x&x\in C=[1,\infty]\\
\end{array}
\right..
$$
The M{\"o}bius transformations $\mf{a},\mf{b},\mf{c}$ are reflections in the sides of the ideal hyperbolic triangle with vertices $\{0,1,\infty\}$, generating a group $\Gamma_{\mbb{R}}$ isomorphic to the free product $\mbb{Z}/(2)*\mbb{Z}/(2)*\mbb{Z}/(2)$.  The fixed points of $t$ are $\{0,1,\infty\}$ and iteration of $t$ takes rationals to one of the fixed points in finite time, depending on the parity of the numerator and denominator, preserving parity as $\Gamma_{\mbb{R}}$ is the kernel of the map $\PGL_2(\mbb{Z})\to\operatorname{SL}_2(\mbb{Z}/(2))$ (index $6$).  See the next subsection for a proof that orbits are finite on rationals.  Iteration of $t$ with input $x$ produces a word $\mf{x}=\mf{m}_1\dots$ in $\{\mf{a},\mf{b},\mf{c}\}$, $\mf{m}_i(t^{i-1}(x))=t(t^{i-1}(x))$, such that $\mf{m}_i\neq\mf{m}_{i+1}$.  We take the word to be finite if $x$ is rational.

We obtain rational approximations to $x$ by following the inverse orbit of $\{0,1,\infty\}$:
$$
\frac{p_{n,\alpha}}{q_{n,\alpha}}=\left(\prod_{i=1}^n\mf{x}_i\right)(\alpha), \ \alpha\in\{0,1,\infty\},
$$
which include the usual convergents from simple continued fractions.  The convergents can be constructed starting from $(1/1,\pm1/0,0/1)$, according as $x$ is positive or negative, and taking mediants, moving through the Farey tree.  Applying $\mf{a}$, $\mf{b}$, or $\mf{c}$ updates the first, second, or third position by taking the mediant $\frac{p}{q}\oplus\frac{r}{s}:=\frac{p+q}{r+s}$ of the other two entries.

For example for a random number
$$
x=0.4189513796210592\dots, \ \mf{x}=\mf{b}\mf{a}\mf{c}\mf{a}\mf{b}\mf{c}\mf{a}\mf{c}\mf{b}\mf{c}\mf{a}\mf{c}\mf{a}\mf{c}\mf{a}\mf{b}\mf{a}\mf{b}\mf{a}\mf{c}\dots,
$$
the first 20 convergents are (mediants in red, see Figure \ref{fig:triapprox})
{\small\begin{align*}
&(1/1,1/0,0/1), \ (1/1,\textcolor{red}{1/2},0/1), \ (\textcolor{red}{1/3},1/2,0/1), \ (1/3,1/2,\textcolor{red}{2/5}), \ (\textcolor{red}{3/7},1/2,2/5),\
(3/7,\textcolor{red}{5/12},2/5), \\ 
&(3/7,5/12,\textcolor{red}{8/19}), \ (\textcolor{red}{13/31},5/12,8/19), \ (13/31,5/12,\textcolor{red}{18/43}), \ (13/31,\textcolor{red}{31/74},18/43), \\ &(13/31,31/74,\textcolor{red}{44/105}), \  
(\textcolor{red}{75/179},31/74,44/105), \ 
(75/179,31/74,\textcolor{red}{106/253}), \\
&(\textcolor{red}{137/327},31/74,106/253),\ 
(137/327,31/74,\textcolor{red}{168/401}), \
(\textcolor{red}{199/475},31/74,168/401), \\
&(199/475,\textcolor{red}{367/876},168/401), \ 
(\textcolor{red}{535/1277},367/876,168/401),\ 
(535/1277,\textcolor{red}{703/1678},168/401), \\ 
&(\textcolor{red}{871/2079},703/1678,168/401), \ 
(871/2079,703/1678,\textcolor{red}{1574/3757})\\
&=(0.4189514\dots,0.4189511\dots,0.4189512\dots).
\end{align*}}

\begin{figure}
\begin{center}
\includegraphics[scale=.6]{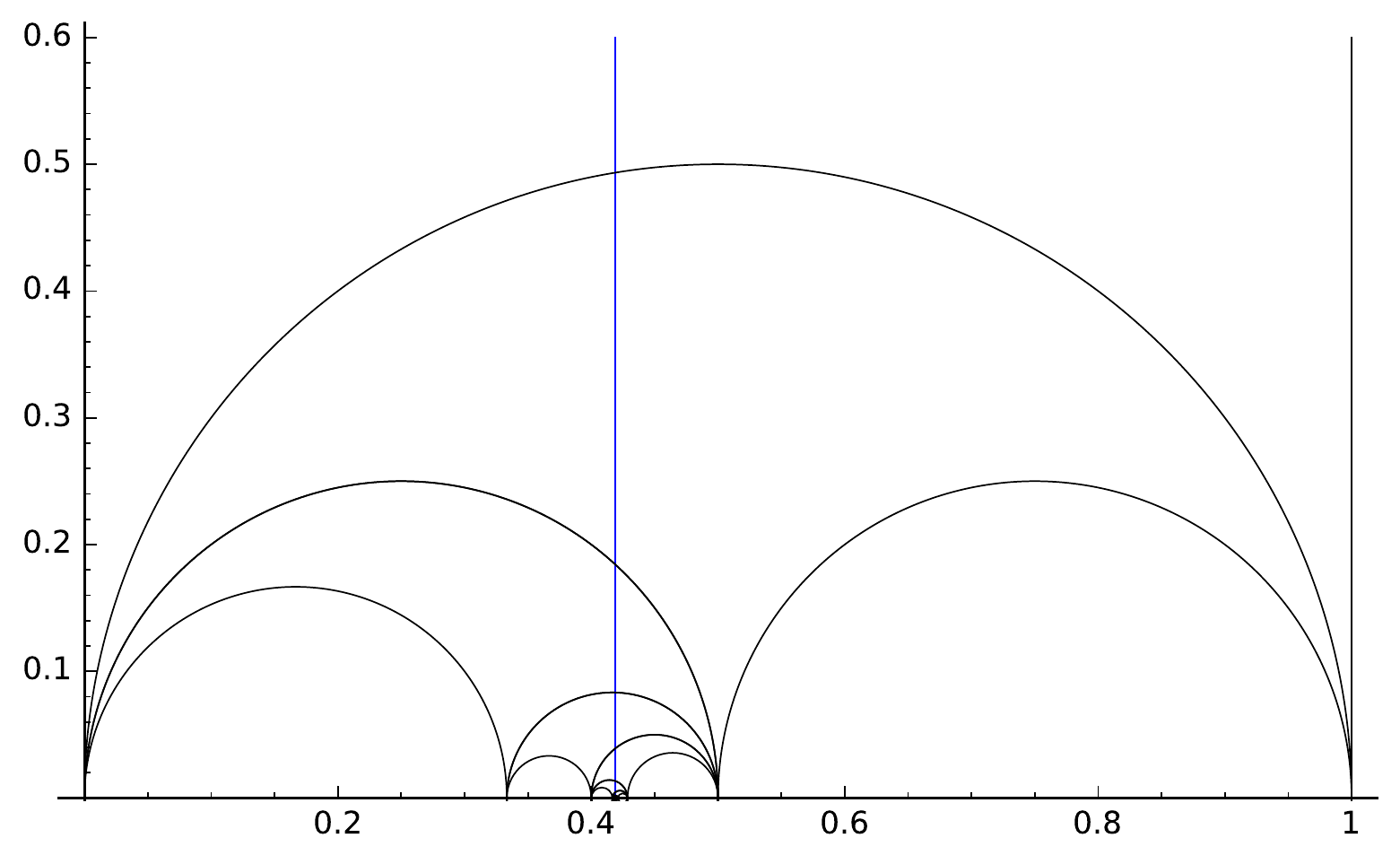}
\caption{Approximating $x=0.4189513796210592\dots, \ \mf{x}=\mf{b}\mf{a}\mf{c}\mf{a}\mf{b}\mf{c}\mf{a}\mf{c}\mf{b}\mf{c}\mf{a}\mf{c}\mf{a}\mf{c}\mf{a}\mf{b}\mf{a}\mf{b}\mf{a}\mf{c}\dots$.}
\label{fig:triapprox}
\end{center}
\end{figure}

The map $T(y,x)=(\mf{m}(y),\mf{m}(x))$, $\mf{m}(x)=t(x)$, extending $t$ in second coordinate, is a bijection on the space of geodesics
$$
\mc{G}_{\mbb{R}}=A\times B\cup A\times C\cup B\times C\cup B\times A\cup C\times A\cup C\times B\setminus \text{diag.}
$$
where there is an isometry invariant measure $dx\;dy\;|x-y|^{-2}$.  Pushing forward to the second coordinate gives the \textit{infinite} $t$-invariant measure
$$
f(x)\;dx=d\mu(x)=\left\{
\begin{array}{cc}
\frac{dx}{-x}&x<0,\\
\frac{dx}{x(1-x)}&0<x<1,\\
\frac{dx}{x-1}&x>1,\\
\end{array}\right..
$$
This dynamical system is clearly a cross-section of billiards in the ideal hyperbolic triangle:  the bi-infinite word $\mf{y}^{-1}\mf{x}$ in $\{\mf{a},\mf{b},\mf{c}\}$ corresponding to a geodesic $(y,x)$ records the sequence of collisions with the walls.  The return time, say for a geodesic $(y,x)\in[-\infty,0]\times[1,\infty]$, is given by
$$
r(y,x)=\frac{1}{2}\log\left(\frac{x(1-y)}{y(1-x)}\right).
$$
The return time is integrable with respect to $(x-y)^{-2}\;dy\;dx$, for instance
$$
\frac{1}{2}\int_{-\infty}^0\int_1^{\infty}\log\left(\frac{x(1-y)}{y(1-x)}\right)\frac{dx\;dy}{(y-x)^2}=\frac{\pi^2}{6}.
$$
Since this triangle reflection group has finite covolume, the system $(\mc{G}_{\mbb{R}},T,(x-y)^{-2}\;dy\;dx)$ is ergodic, implying ergodicity of $(P^1(\mbb{R}),t,\mu)$

\subsection{A Euclidean algorithm}
\label{Oureucalso}
By homogenizing $t$ above, we get a dynamical system on pairs of integers $(p,q)\in\mbb{Z}^2$ which halts when $p=q$ or one of $p$, $q$ is zero.
$$
(p,q)\mapsto\left\{
\begin{array}{ll}
a(p,q)=(-p,q)&q<0<p \text{ or } p<0<q, \text{ i.e. } p/q<0,\\
b(p,q)=(p,2p-q)&0<p<q \text{ or } q<p<0, \text{ i.e. } 0<p/q<1,\\
c(p,q)=(2q-p,q)&0<q<p \text{ or } p<q<0, \text{ i.e. } p/q>1.\\
\end{array}
\right.
$$
The $b$ step reduces $|q|$, the $c$ step reduces $|p|$, and after applying $a$, one of either $b$ or $c$ follows.  Hence the algorithm terminates.  If $(p,q)\neq(0,0)$ then the non-zero entry of the output is $\pm\text{gcd}(p,q)$, and working backwards allows us to write the gcd as a linear combination of $p$ and $q$ (using only the coefficients $-1$ and $2$ at each step).

\subsection{Dynamics on triples}
In this section, we make some remarks relating the triangle reflection group to a ``dual Apollonian group'' on the line and conjugate this group to act on Pythagorean triples (obtaining a system on the circle conjugate to $(P^1(\mbb{R}),t,\mu)$).

Here is a lower dimensional version of Descartes' theorem on curvatues.  Consider three real numbers, $a<b<c$ (one of which may be $\infty$, ordered as on a circle).  The inverses of the lengths of the interals they define (infinite length if one endpoints is $\infty$, negative length if the interval contains $\infty$) satisfy $(x+y+z)^2-(x^2+y^2+z^2)=0$.  In analogy with the dual Apollonian group $\langle\mf{s}_i^{\perp} : 1\leq i\leq4\rangle$, we can define a group acting on ordered triples of mutually tangent oriented intervals as ``inversions'' in the equivalent of ACC coordinates using indefinite binary quadratic forms.  If
$$
F=\left(
\begin{array}{cc}
a&-b\\
-b&c\\
\end{array}
\right), \ a,b,c\in\mbb{R}, \ \det(F)=-1,
$$
then the zero set of $F$, $(x-b/a)^2=1/a^2$, defines an oriented interval with curvature $a$, co-curvature $c$, and curvature-center $b$ (if $a=0$, then $b$ is $\pm1$ depending on orientation). Three forms $F_i$ (considered as geodesics of the upper half-plane model of $H^2$) are in ``triangular'' configuration if they define an ideal hyperbolic triangle with proper orientation.  The analogous group has generators
$$
A=\left(
\begin{array}{ccc}
-1&0&0\\
2&1&0\\
2&0&1\\
\end{array}
\right)
B=\left(
\begin{array}{ccc}
1&2&0\\
0&-1&0\\
0&2&1\\
\end{array}
\right)
C=\left(
\begin{array}{ccc}
1&0&2\\
0&1&2\\
0&0&-1\\
\end{array}
\right)
$$
each defining an inversion in the given interval.  The group generated by $A,B,C$ preserves the quadratic form $xy+xz+yz$, i.e.
$$
M^tQM=Q, \ Q=\frac{1}{2}\left(
\begin{array}{ccc}
0&1&1\\
1&0&1\\
1&1&0\\
\end{array}
\right), \ M\in\langle A,B,C\rangle.
$$
Three intervals/forms/geodesics $R$ are in triangular configuration if they satisfy
$$
R^tQR=\left(
\begin{array}{ccc}
0&-2&0\\
-2&0&0\\
0&0&-1\\
\end{array}
\right).
$$
The group $\Gamma_{\mbb{R}}$ is then a geometric realization of this group with respect to the base triple representing the ideal hyperbolic triangle with vertices $0$, $1$, $\infty$ (rows in $(c,a,b)$ coordinates)
$$
\left(
\begin{array}{ccc}
0&0&-1\\
0&2&1\\
-2&0&1\\
\end{array}
\right).
$$

The form $Q$ is rationally equivalent to the ``Pythagorean'' diagonal form.  For instance
$$
P=J^tQJ, \ P=\left(
\begin{array}{ccc}
1&0&0\\
0&-1&0\\
0&0&-1\\
\end{array}
\right), \ J=\left(
\begin{array}{ccc}
1&1&1\\
0&-1&0\\
1&1&-1\\
\end{array}
\right).
$$
Conjugating $A$, $B$, $C$ by $J$ gives
$$
A^J=\left(
\begin{array}{ccc}
3&2&2\\
-2&-1&-2\\
-2&-2&-1\\
\end{array}
\right), \ 
B^J=\left(
\begin{array}{ccc}
1&0&0\\
0&-1&0\\
0&0&1\\
\end{array}
\right), \ 
C^J=\left(
\begin{array}{ccc}
3&2&-2\\
-2&-1&2\\
2&2&-1\\
\end{array}
\right).
$$
We can define a dynamical system on triples of integers $(a,b,c)$ such that $a^2=b^2+c^2$, $a>0$ which reduces the value of $a$ and terminates at one of $(g,-g,0)$, $(g,0,\pm g)$, where $g$ is the GCD of $a$, $b$, and $c$:
$$
(a,b,c)\mapsto\left\{
\begin{array}{cc}
(3a+2a+2c,-2a-b-2c,-2a-2b-c)&b<0, \ c<0,\\
(a,-b,c)&b>0,\\
(3a+2b-2c,-2a-b+2c,2a+2b-c)&b<0, \ c>0.\\
\end{array}
\right.
$$
Dividing by $a^2$, we obtain a sytem on the circle $x^2+y^2=1$, conjugate to $t$
$$
(x,y)\mapsto
\left\{
\begin{array}{cc}
\left(\frac{-2-x-2y}{3+2x+2y},\frac{-2-2x-y}{3+2x+2y}\right)&x<0, \ y<0,\\
(-x,y)&x>0,\\
\left(\frac{-2-x+2y}{3+2x-2y},\frac{2+2x-y}{3+2x-2y}\right)&x<0, \ y>0.\\
\end{array}
\right.
$$

\subsection{Comparing to Romik's Work}

Romik's action on Pythagorean triples also gives a dynamical system on the real line with an ergodic invariant measure and another variation on the Euclidean algorithm \cite{R}.  For example, Romik's Euclidean algorithm in section 3.2 of \cite{R} is a dynamical system on nonnegative pairs of integers $(p,q)$ with $p>q$ where
 $$
 (p,q)\mapsto\left\{
 \begin{array}{ll}
 a(p,q)=(p-2q,q)&p-2q>q,\\
 b(p,q)=(q,p-2q)&q\geq p-2q>0,\\
 c(p,q)=(q,2q-p)&p-2q\leq0.\\
 \end{array}
 \right.
 $$
 For the sake of comparison, here is the greatest common divisor of $246$ and $113$ in the system described in section \ref{Oureucalso}:
 \begin{eqnarray*}(246,113)&\rightarrow&(-20,113)\rightarrow(20,113)\rightarrow(20,-73)\rightarrow(-20,-73)\rightarrow(-20,33)\rightarrow(20,33)\\&\rightarrow&(20,7)\rightarrow(-6,7)\rightarrow(6,7)\rightarrow(6,5)\rightarrow(4,5)\rightarrow(4,3)\rightarrow(2,3)\rightarrow(2,1)\rightarrow(0,1),
 \end{eqnarray*}
 while Romik's algorithm runs more quickly as follows: 
 \begin{eqnarray*}(246,113)&\rightarrow&(113,20)\rightarrow(73,20)\rightarrow(33,20)\rightarrow(20,7)\rightarrow(7,6)\rightarrow(6,5)\\ &\rightarrow&(5,4)\rightarrow(4,3)\rightarrow(3,2)\rightarrow(2,1)\rightarrow(1,0).
 \end{eqnarray*}

\section{Lorentz quadruples by size: another dynamical system}
\label{sec:lor}

In contrast to the approach taken so far, it may be desirable to prioritize the feature that the dynamical system moves to the root by travelling to the arithmetically simplest among the adjacent Lorentz quadruples.  This leads to a different dynamical system.

We say a Lorentz quadruple $(a,b,c,d)$ is normalized if $a \ge b \ge c \ge d \ge 0$.  We can normalize a quadruple by changing signs of its entries and reordering entries.  Order the normalized Lorentz quadruples lexicographically, so that the first, or least, is $(1,1,0,0)$.  We will refer to the position of the quadruples in this ordering as their \emph{height}; quadruples with the same normalization have the same height.  The height of a possibly non-normalized Lorentz quadruple is the height of its normalization.

%

The purpose of this section is to construct a dynamical system on quadruples which travels to the root by passing at each step to the adjacent quadruple of least height.  The paths to the origin according to this system form a tree.  Writing quadruples in normalized form, one obtains a tree organizing quaduples by height which is shown in Figure \ref{fig:dropdownlorentztree}.

\begin{figure}
\[
        \xymatrix@R=0.7em{
                & & & 9, 8, 4, 1 \; (6) \\
                & & & 11, 9, 6, 2 \; (6) \\
                & & 5, 4, 3, 0 \ar[ru] \ar[ruu] \ar[r] & 17, 12, 9, 8 \; (5) \\
 & & & 13, 12, 4, 3 \; (5)\\
 & & & 15, 11, 10, 2\; (5) \\
 & & 7, 6, 3, 2 \ar[ru] \ar[ruu] \ar[r] \ar[rd] \ar[rdd] &  19, 15, 10, 6\; (5)\\
 & & & 21, 16, 11, 8 \; (5) \\
 & & & 25, 16, 15, 12 \; (4)  \\
 & & & 17, 12, 12, 1 \; (4)\\
                1, 1, 0, 0 \ar[r] & 3, 2, 2, 1 \ar[ruuuuuuu] \ar[ruuuu] \ar[r] \ar[rddd] & 9, 7, 4, 4 \ar[ru] \ar[r] \ar[rd] \ar[rdd] & 19, 17, 6, 6 \; (4)\\
                                  & & & 25, 20, 12, 9 \; (5) \\
                                  & & &  33, 20, 20, 17 \; (3)\\
 & & 11, 7, 6, 6 \ar[r] \ar[rd] \ar[rdd] & 27, 23, 10, 10\; (4) \\
 & & & 29, 24, 12, 11\; (5)\\
 & & & 41, 24, 24, 23\; (3)\\
        }
\]
\caption{Quadruples are considered up to sign and ordering; we normalize them as all non-negative, with $a \ge b \ge c \ge d$.}
\label{fig:dropdownlorentztree}
\end{figure}
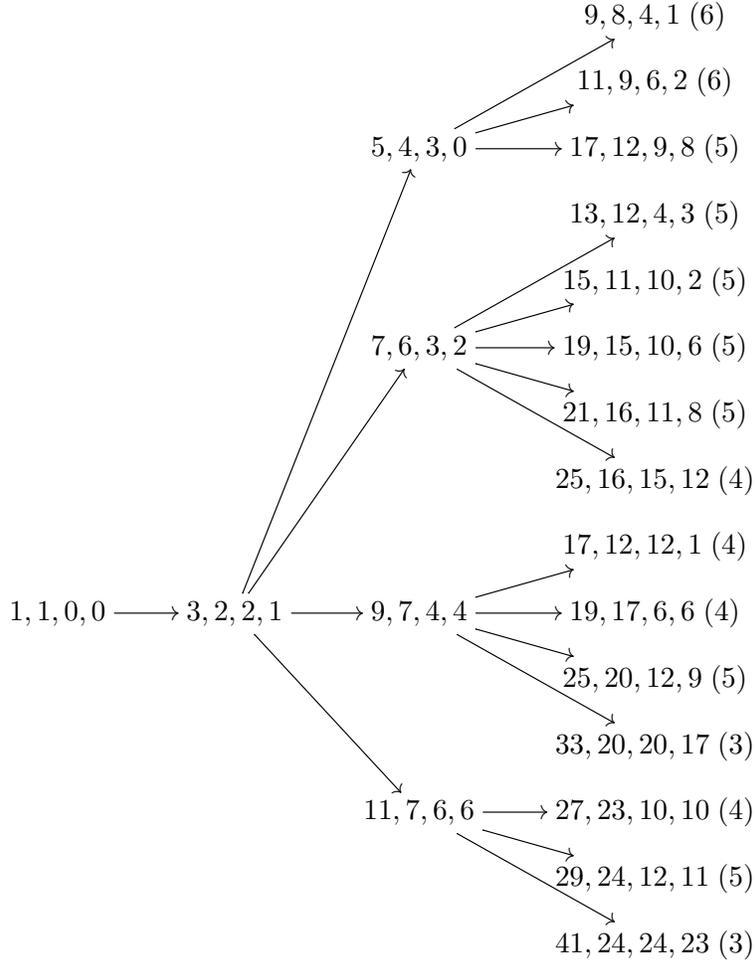

Let $\Lcal_+ \subseteq \Lcal$ represent the Lorentz quadruples $(a,b,c,d)$ that satisfy $a,b,c,d \ge 0$.  We define a dynamical system on $\Lcal_+$.  Write $A = 2a-b-c-d$, $B = a-c-d$, $C = a - b - d$ and $D = a-b-c$.  Define
\[
        T_D(a,b,c,d) = (|A|, |B|, |C|, |D|) = (A, |B|, |C|, |D|).
\]
Observe that $A > 0$, i.e. $2a > b+c+d$ for any Lorentz quadruple.

Define the seven matrices:
\[
      D_1 :=  \begin{pmatrix}
                2 & 1 & -1 & -1 \\
                -1 & 0 & 1 & 1 \\
                -1 & -1 & 0 & 1 \\
                -1 & -1 & 1 & 0
        \end{pmatrix},
        \quad
       D_2 := \begin{pmatrix}
                2 & -1 & 1 & -1 \\
                -1 & 0 & -1 & 1\\
                -1 & 1 & 0 & 1 \\
                -1 & 1 & -1 & 0 \\
        \end{pmatrix},
        \quad
        D_3 := \begin{pmatrix}
                2 & -1 & -1 & 1 \\
                -1 & 0 & 1 & -1 \\
                -1 & 1 & 0 & -1 \\
                -1 & 1 & 1 & 0 \\
        \end{pmatrix}
\]
\[
        D_4 := \begin{pmatrix}
                2 & -1 & -1 & -1 \\
                -1 & 0 & 1 & 1 \\
                -1 & 1 & 0 & 1 \\
                -1 & 1 & 1 & 0
        \end{pmatrix}
\]
\[
        D_5 := \begin{pmatrix}
                2 & -1 & 1 & 1 \\
                -1 & 0 & -1 & -1 \\
                -1 & 1 & 0 & -1\\
                -1 & 1 & -1 & 0 
        \end{pmatrix}
        \quad
        D_6 := \begin{pmatrix}
                2 & 1 & -1 & 1 \\
                -1 & 0 & 1 & -1 \\
                -1 & -1 & 0 & -1 \\
                -1 & -1 & 1 & 0 \\
        \end{pmatrix}
        \quad
        D_7 := \begin{pmatrix}
                2 & 1 & 1 & -1 \\
                -1 & 0 & -1 & 1 \\
                -1 & -1 & 0 & 1 \\
                -1 & -1 & -1 & 0
        \end{pmatrix}
\]

The dynamical system generates a word $M_1M_2\cdots$ from left-to-right in the $D_i$, since if $T_D(a,b,c,d) = (|A|,|B|,|C|,|D|)$, then $D_i(A,B,C,D)^t = (a,b,c,d)^t$ for exactly one $i$ (it is not possible that $a > c+d, b+d, b+c$).  In other words exactly one of the $D_i$ undoes $T_D$.  Explicitly, we let $D_i$ be
\[
        \left\{ \begin{array}{ll}
                        D_1 & D,C < 0, B> 0 \\
                        D_2 & D,B < 0, C> 0 \\
                        D_3 & C,B < 0, D> 0 \\
                        D_4 & B,C,D < 0 \\
                        D_5 & D,C > 0, B< 0 \\
                        D_6 & D,B > 0, C< 0 \\
                        D_7 & B,C > 0, D< 0 \\
                \end{array}
                        \right.
\]

We give $\mathcal{L}_+$ the structure of a graph as follows:  let two quadruples $\mathbf{b}_1$ and $\mathbf{b}_2$ be joined by an edge whenever $\mathbf{b}_2$ is obtained by $L_i \mathbf{b}_1$, followed by taking absolute values.

\begin{theorem}
        \label{thm:dropdown}
	For any $(a,b,c,d) \in \Lcal_+$, $T_D^{(n)}(a,b,c,d) = \mathbf{b}$ for some integer $n$, and some $\mathbf{b} \in \{ (1,1,0,0), (1,0,1,0), (1,0,0,1) \}$.  
        Then the word
	$      M = M_{1} \cdots M_{n}$
	generated by $T_D$ satisfies $(a,b,c,d) = M\mathbf{b}$.   Furthermore, the path $(T_D^{(k)}(a,b,c,d))_{k=0}^{n}$ to $\mathbf{b}$ on $\Lcal_+$ is of minimal length in the graph $\mathcal{L}_+$.  Finally, the path is the same as the path obtained by always travelling to the adjacent quadruple of smallest height.
\end{theorem}

\begin{lemma}
  \label{lemma:dropdown}
  Label $\mathcal{L}^+$ according to direction of decrease of height.  Then
\begin{enumerate}
  \item All edges are directed.
  \item \label{item:origin} the \emph{origin} vertices $\mathbf{b} \in \{(1,1,0,0),(1,0,1,0),(1,0,0,1)\}$ have no outward directed edges,
        \item \label{item:allout} aside from the origins, every vertex has at least one outward directed edge,
	\item \label{item:cycles} the only minimal cycles in the graph are squares, and
        \item \label{item:square} any square is isomorphic to
        \[
                \xymatrix{
                        \bullet \ar[r] \ar[d] & \bullet \ar[d] \\
                        \bullet \ar[r] & \bullet  \\
                }
        \]
\end{enumerate}
\end{lemma}

\begin{proof}

  Of the adjacent quadruples, the smallest is that with first entry $2a - b - c - d$.  We have $2a -b-c-d \ge a$ if and only if $a \ge b + c+d$.  However, since $a^2 = b^2+c^2+d^2$, and we are assuming $b,c,d \ge 0$, this would entail $b^2+c^2+d^2 \ge (b+c+d)^2$, which can only occur if at least two of $b,c,d$ are zero.  Hence, in the primitive case, this only occurs when the vertex is an origin.  Therefore the origins have no outward directed edges, but every other vertex does have an outward directed edge.

An edge is undirected if and only if the quadruples are the same up to reordering the non-negative quantities $b,c,d$.  The adjacent quadruples have largest entry $2a \pm b \pm c \pm d$.  Therefore an undirected edge can occur only if $a = \pm b \pm c \pm d$ for some choice of signs.  But since $a \ge b,c,d \ge 0$, this is only possible with at most one negative sign.  If $a=b+c+d$, we are in the case of the root, as above.  Otherwise, if $a=b+c-d$, an undirected edge implies that $a \ge b=c \ge d$, and the two quadruples are the same, including ordering, hence the same vertex.   This shows that the graph has no undirected edges.

Observe that $\mathcal{L}_+$ is by definition a union of images of the Cayley graph $\mathcal{C}_L$ under graph homomorphism. In particular, the presentation of the Super-Apollonian group implies that the only minimal cycles in $\mathcal{C}_L$ are squares.  So the minimal cycles in $\mathcal{L}_+$ consist of images of squares.

  Up to permuting the second through fourth coordinates, or changing their signs, the homomorphic images of a square in $\mathcal{C}_L$ is always labelled with first coordinates as follows, where edges are labelled with differences in the direction shown:
\[
                \xymatrix@C=5em{
                        2a - b + c + d \ar@{->}^{a-b+c-d}[r] \ar@{<-}_{a-b+c+d}[d] & 3a - 2b + 2c \ar@{<-}^{a-b+c+d}[d] \\
                        a \ar@{->}_{a-b+c-d}[r] & 2a - b + c - d \\
                }
        \]
	Since in a square of $\mathcal{C}_L$ the edge labels are non-zero, the direction of the edges is determined by these values, and it is evident that parallel edges must be directed in the same way, from which the assertion about directions on squares follows.  

       Finally, the non-existence of triangles in $\mathcal{L}_+$ follows from the diagram above, since such a triangle must be a homomorphic image of a square, and the edge labellings demonstrate that if one side of the square collapses, the opposite side also collapses.
\end{proof}

\begin{proof}[Proof of Theorem \ref{thm:dropdown}]
  By items \eqref{item:origin} and \eqref{item:allout} of Lemma \ref{lemma:dropdown}, and the well-ordering principle for heights, we know that from any quadruple there is a path to the origin which respects the direction of the graph.  The path of greatest decrease is unique since there is a unique adjacent quadruple of least height.  This last fact arises as follows:  suppose $2a - b -c-d = 2a \pm b \pm c \pm d$ for some other choice of signs.  Then one of $b,c,d$ is zero.  If two entries are zero, we are at the root.  If one entry is zero, comparing matrices $L_i$ shows that the two quadruples are equal, so there is just one quadruple of least height.

        Therefore the graph formed of fastest-dropping paths consists of three trees, which together span the graph.
         
        It remains to show minimality of the path lengths.
	Compare the path of greatest descent to the image of the swap normal form path.  Both descend the origin, and both respect the direction of the graph (as observed in the proof of Theorem \ref{thm:TL}).  The moves that take one path to the other consist of square crossings (by Lemma \ref{lemma:dropdown} \eqref{item:cycles}).  Let the weight of a path be given by the sum of $1$ for each edge traversed respecting the graph direction, and $-1$ for each edge traversed against the graph direction.  Then square crossings preserve weight, by Lemma \ref{lemma:dropdown} \eqref{item:square}, so both paths have equal weight.  Since both paths respect graph direction, this implies they are of the same length.  
 
%
%
\end{proof}

%
\section*{Appendix: Summary of A. L. Schmidt's continued fractions} \label{sec:schmidt-nakada}

Here we give a quick summary of Asmus Schmidt's continued fraction algorithm \cite{S1}, its ergodic theory \cite{S2}, and further results of Hitoshi Nakada concerning these \cite{N1}, \cite{N2}, \cite{N3}.  We also explicitly relate Schmidt's system to ours.

Define the following matrices in PGL$(2,\mathbb{Z}[i])$:
\begin{align*}
V_1&=\left(
\begin{array}{cc}
1&i\\
0&1\\
\end{array}
\right), \ 
V_2=\left(
\begin{array}{cc}
1&0\\
-i&1\\
\end{array}
\right), \ 
V_3=\left(
\begin{array}{cc}
1-i&i\\
-i&1+i\\
\end{array}
\right),\\
E_1&=\left(
\begin{array}{cc}
1&0\\
1-i&i\\
\end{array}
\right), \
E_2=\left(
\begin{array}{cc}
1&-1+i\\
0&i\\
\end{array}
\right), \
E_3=\left(
\begin{array}{cc}
i&0\\
0&1\\
\end{array}
\right), \
C=\left(
\begin{array}{cc}
1&-1+i\\
1-i&i\\
\end{array}
\right).
\end{align*}

Note that
$$
S^{-1}V_iS=V_{i+1}, \ S^{-1}E_iS=V_{i+1}, \ S^{-1}CS=C \text{ (indices modulo 3)}
$$
where $S$ is the order three elliptic element 
$\left(
\begin{array}{cc}
0&-1\\
1&-1\\
\end{array}\right),$ and that
$$
\mf{c}\circ m\circ\mf{c}=m^{-1}
$$
for the M{\"o}bius transformations $m$ induced by $\{V_i,E_i,C\}$ (here $\mf{c}$ is complex conjugation).

In \cite{S1} Schmidt uses infinite words in these letters to represent complex numbers as infinite products $z=\prod_nT_n, T_n\in\{V_i,E_i,C\}$ in two different ways.  Let $M_N=\prod_{n=1}^NT_n$.  We have \emph{regular chains}
$$
\det M_N=\pm1\Rightarrow T_{n+1}\in\{V_i,E_i,C\}, \ \det M_N=\pm i\Rightarrow T_{n+1}\in\{V_i,C\},
$$
representing $z$ in the upper half-plane $\mathcal{I}$ (the model circle) and \emph{dually regular chains}
$$
\det M_N=\pm i\Rightarrow T_{n+1}\in\{V_i,E_i,C\}, \ \det M_N=\pm 1\Rightarrow T_{n+1}\in\{V_i,C\},
$$
representing $z\in\{0\leq x\leq1,y\geq0, |z-1/2|\geq1/2\}=:\mathcal{I}^*$ (the model triangle).  The model circle is a disjoint union of four triangles and three circles, and the model triangle is a disjoint union of three triangles and one circle (pictured in Figure \ref{fig:NS1}):

\begin{align*}
\mathcal{I}&=\mathcal{V}_1\cup\mathcal{V}_2\cup\mathcal{V}_3\cup\mathcal{E}_1\cup\mathcal{E}_2\cup\mathcal{E}_3\cup\mathcal{C}\\
\mathcal{I}^*&=\mathcal{V}_1^*\cup\mathcal{V}_2^*\cup\mathcal{V}_3^*\cup\mathcal{C}^*
\end{align*}

where
\begin{align*}
\mathcal{V}_i=v_i(\mathcal{I}), \ \mathcal{E}_i=e_i(\mathcal{I}^*), \ \mathcal{C}=c(\mathcal{I}^*), \ \mathcal{V}_i^*=v_i(\mathcal{I}^*), \ \mathcal{C}^*=c(\mathcal{I}),
\end{align*}
(lowercase letters indicating the M{\"o}bius transformation associated to the corresponding matrix).

\begin{figure}
\begin{tikzpicture}
\draw (-4,0)--(4,0);
\draw (-1,1)circle(1);
\draw (1,1)circle(1);
\draw (-4,2)--(4,2);
\draw(8,0)arc(0:180:1);
\draw(7,2)circle(1);
\draw(6,0)--(6,4);
\draw(8,0)--(8,4);
\draw(0,.25)node{$\mc{E}_1$};
\draw(3,1)node{$\mc{E}_2$};
\draw(-3,1)node{$\mc{E}_3$};
\draw(0,3)node{$\mc{V}_1$};
\draw(-1,1)node{$\mc{V}_2$};
\draw(1,1)node{$\mc{V}_3$};
\draw(0,1.75)node{$\mc{C}$};

\draw(7,3.5)node{$\mc{V}_1^*$};
\draw(6.25,1)node{$\mc{V}_2^*$};
\draw(7.75,1)node{$\mc{V}_3^*$};
\draw(7,2)node{$\mc{C}^*$};
\end{tikzpicture}
\caption{Model circle and model triangle of Schmidt.}
\label{fig:NS1}
\end{figure}
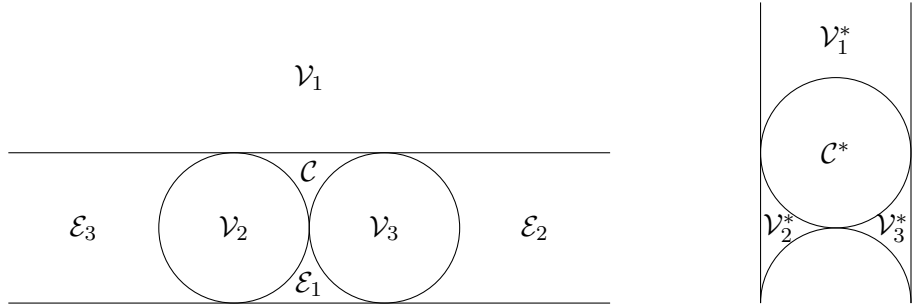

\begin{figure}
\begin{tikzpicture}
\draw (-4,0)--(4,0);
\draw (-1,1)circle(1);
\draw (1,1)circle(1);
\draw (-4,2)--(4,2);

\draw(-1,0)--(-1,-4);
\draw(1,0)--(1,-4);
\draw(-1,0)arc(180:360:1);
\draw(0,-2)circle(1);

\draw(0,.25)node{$\mc{E}_1$};
\draw(3,1)node{$\mc{E}_2$};
\draw(-3,1)node{$\mc{E}_3$};
\draw(0,3)node{$\mc{V}_1$};
\draw(-1,1)node{$\mc{V}_2$};
\draw(1,1)node{$\mc{V}_3$};
\draw(0,1.75)node{$\mc{C}$};

\draw(0,-3.5)node{$\overline{\mc{V}_1^*}$};
\draw(-.75,-1)node{$\overline{\mc{V}_2^*}$};
\draw(.75,-1)node{$\overline{\mc{V}_3^*}$};
\draw(0,-2)node{$\overline{\mc{C}^*}$};

\end{tikzpicture}
\begin{tikzpicture}
\draw(0,0)arc(0:180:1);
\draw(-1,2)circle(1);
\draw(-2,0)--(-2,4);
\draw(0,0)--(0,4);

\draw(-4,0)--(2,0);
\draw(-4,-2)--(2,-2);
\draw(-2,-1)circle(1);
\draw(0,-1)circle(1);

\draw(-1,3.5)node{$\mc{V}_1^*$};
\draw(-1.75,1)node{$\mc{V}_2^*$};
\draw(-.25,1)node{$\mc{V}_3^*$};
\draw(-1,2)node{$\mc{C}^*$};

\draw(-1,-.25)node{$\overline{\mc{E}_1}$};
\draw(1.5,-1)node{$\overline{\mc{E}_2}$};
\draw(-3.5,-1)node{$\overline{\mc{E}_3}$};
\draw(-1,-3)node{$\overline{\mc{V}_1}$};
\draw(-2,-1)node{$\overline{\mc{V}_2}$};
\draw(0,-1)node{$\overline{\mc{V}_3}$};
\draw(-1,-1.75)node{$\overline{\mc{C}}$};
\end{tikzpicture}
\caption{Regions for invertible extension of Nakada.}
\label{fig:SN2}
\end{figure}
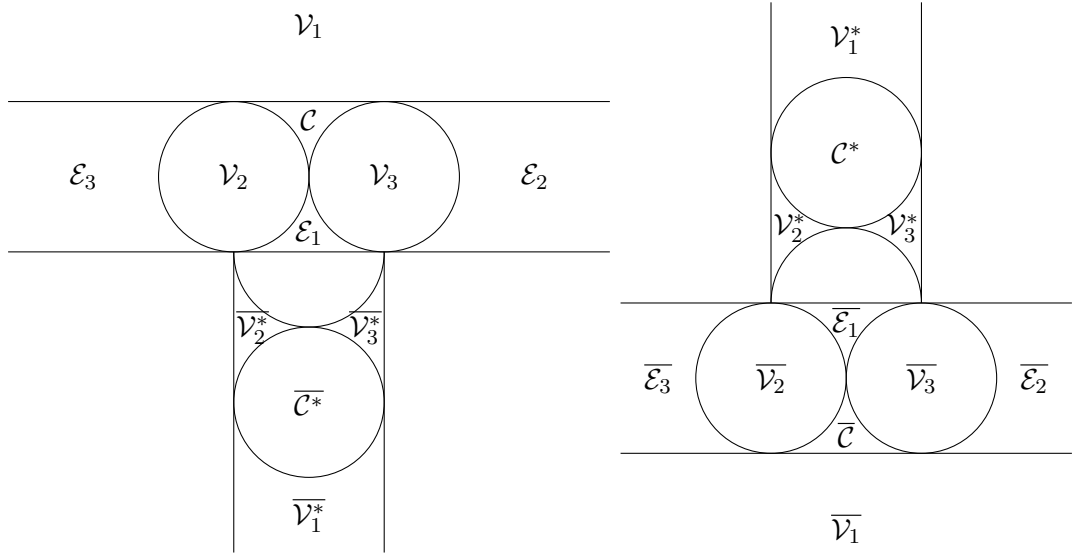

By considering $z=\prod_nT_n$ we obtain rational approximations $p_i^{(N)}/q_i^{(N)}$ to $z$ by

$$
M_N
\left(
\begin{array}{ccc}
1&0&1\\
0&1&1\\
\end{array}
\right)
=
\left(
\begin{array}{ccc}
p^{(N)}_1&p^{(N)}_2&p^{(N)}_3\\
q^{(N)}_1&q^{(N)}_2&q^{(N)}_3\\
\end{array}
\right)
$$
(which are the orbits of $\infty,0,1$ under the partial products $m_N=t_1\circ\dots\circ t_N$).  In \cite[Theorem 2.5]{S1} a quality of approximation is given: if $|z-p/q|<\frac{1}{(1+1/\sqrt{2})|q|^2}$, then $p/q$ is a convergent to $z$.

The shift map $T$ on $X=\mathcal{I}\cup\mathcal{I}^*=\{\text{chains, dual chains}\}$ maps $X$ to itself via M{\"o}bius transformations, specifically (mapping $\mathcal{V}_i, \mathcal{C}^*$ onto $\mathcal{I}$ and $\mathcal{V}_i^*, \mathcal{E}_i, \mathcal{C}$ onto $\mathcal{I}^*$)
$$
T(z)=
\left\{
\begin{array}{cc}
v_i^{-1}z& z\in\mathcal{V}_i\cup\mathcal{V}_i^*\\
e_i^{-1}z&z\in\mathcal{E}_i\\
c^{-1}z&z\in\mathcal{C}\cup\mathcal{C}^*\\
\end{array}
\right.
.
$$
The shift $T:X\to X$ is shown to be ergodic (\cite[Theorem 5.1]{S2}) with respect to the following probability measure
$$
\tilde{f}(z)=
\left\{
\begin{array}{cc}
\frac{1}{2\pi^2}(h(z)+h(sz)+h(s^2z))&z=x+yi\in\mathcal{I}\\
\frac{1}{2\pi}\frac{1}{y^2}&z=x+yi\in\mathcal{I}^*
\end{array}
\right.
$$
where
$$
h(z)=\frac{1}{xy}-\frac{1}{x^2}\arctan\left(\frac{x}{y}\right).
$$
By inducing to $X\setminus\cup_i\left(\mc{V}_i\cup\mc{V}_i^*\right)$ Schmidt gives ``faster'' convergents $\hat{p}^{(n)}_{\alpha}/\hat{q}^{(n)}_{\alpha}$ and a sequence of exponents $e_n$ ($1$ for $E_i$, $C$, and the return time $k$ for $V_i^k$).  He then gives results analogous to those of simple continued fractions via the pointwise ergodic theorem, including the arithmetic and geometric mean of the exponents which exist for almost every $z$ (\cite[Theorem 5.3]{S2}):
$$
\lim_{n\to\infty}\left(\prod_{i=1}^ne_i\right)^{1/n}=1.26\dots, \ \lim_{n\to\infty}\frac{1}{n}\sum_{i=1}^ne_i=1.6667\dots.
$$
In \cite{N1}, Nakada constructs an invertible extension of $T$ on a space of geodesics in two copies of three-dimensional hyperbolic space.  In one copy we take geodesics from $\overline{\mc{I}^*}$ to $\mc{I}$ and in the other the geodesics from $\overline{\mc{I}}$ to $\mc{I}^*$ where the overline indicates complex conjugation.  The regions are pictured in Figure \ref{fig:SN2}.  The extension acts as Schmidt's $T$ depending on the second coordinate.  Nakada doesn't provide a second proof of ergodicity, but quotes Schmidt's result.  Also in \cite{N1}, results about the the density of Gaussian rationals $p/q$ that appear as convergents and satisfy $|z-p/q|<c/|q|^2$ are obtained.  For instance according to \cite[Theorem 7.3]{N1}, for almost every $z\in X$ and $0<c<\frac{1}{1+1\sqrt{2}}$, it holds that
$$
\lim_{N\to\infty}\frac{1}{N}\#\{p/q\in\mbb{Q}(i) : p/q=p_i^{(n)}/q_i^{(n)}, \ 1\leq n\leq N, \ i=1,2,3, \ |z-p/q|<c/|q|^2\}=\frac{c^2}{\pi}.
$$
In \cite[Main Theorem]{N2} Nakada describes the rate of convergence of Schmidt's convergents.  Namely for almost every $z$,
$$
\lim_{n\to\infty}\frac{1}{n}\log|q_i^{(n)}|=\frac{E}{\pi}, \ \lim_{n\to\infty}\frac{1}{n}\log\left|z-\frac{p_i^{(n)}}{q_i^{(n)}}\right|=-\frac{2E}{\pi}, \ E=\sum_{k=0}^{\infty}\frac{(-1)^k}{(2k+1)^2}.
$$
For reference, the relationship between the Super-Apollonian M{\"o}bius generators $\{\mf{s}_i,\mf{s}_i^{\perp}\}$ and Schmidt's $\{v_i,e_i,c\}$ are
$$
\mf{s}_1=c^2\circ\mf{c}, \ \mf{s}_2=e_1^2\circ\mf{c}, \ \mf{s}_3=e_2^2\circ\mf{c}, \ \mf{s}_4=e_3^2\circ\mf{c},
$$
$$
\mf{s}_1^{\perp}=1\circ\mf{c}, \ \mf{s}_2^{\perp}=v_1^2\circ\mf{c}, \ \mf{s}_3^{\perp}=v_2^2\circ\mf{c}, \ \mf{s}_4^{\perp}=v_3^2\circ\mf{c}.
$$

\FloatBarrier

\end{document}